\documentclass[12pt,leqno]{amsart}

\usepackage{amsmath,amscd,amsthm,amsxtra,amssymb}
\usepackage{epsfig,graphics,color,colortbl}
\usepackage{amssymb,latexsym}
\usepackage{mathrsfs}
\usepackage[poly,all]{xy}
\usepackage{marginnote}
\usepackage{xspace}
\usepackage[colorlinks=true, pdfstartview=FitV, linkcolor=blue,citecolor=blue,urlcolor=blue]{hyperref}

\usepackage{yfonts}
\usepackage{enumerate}

\usepackage[usenames,dvipsnames,svgnames,table]{xcolor}

\usepackage[normalem]{ulem}  

\usepackage[colorinlistoftodos]{todonotes}

\allowdisplaybreaks[3]

\newcounter{myc}

\newdir{ >}{{}*!/-10pt/@{>}}

\theoremstyle{plain}
\newtheorem{thm}{\bf Theorem}[section]
\newtheorem{df}[thm]{\bf Definition}
\newtheorem{prop}[thm]{\bf Proposition}
\newtheorem{corollary}[thm]{\bf Corollary}
\newtheorem{lem}[thm]{\bf Lemma}

\theoremstyle{definition}
\newtheorem{ex}[thm]{\bf Example}
\newtheorem{rem}[thm]{\bf Remark}

\newcommand{\nc}{\newcommand}
\nc{\Prop}{\begin{prop}}
\nc{\enprop}{\end{prop}}
\nc{\Lemma}{\begin{lem}}
\nc{\enlemma}{\end{lem}}
\nc{\Exam}{\begin{ex}}
\nc{\enexam}{\end{ex}}
\nc{\Th}{\begin{thm}}
\nc{\enth}{\end{thm}}
\nc{\Cor}{\begin{corollary}}
\nc{\encor}{\end{corollary}}
\nc{\Def}{\begin{df}}
\nc{\edf}{\end{df}}
\nc{\Rem}{\begin{rem}}
\nc{\enrem}{\end{rem}}

\renewcommand{\le}{\leqslant}
\renewcommand{\ge}{\geqslant}

\newenvironment{red}
{\relax\color{red}}
{\hspace*{.5ex}\relax}

\newcommand{\ber}{\begin{red}}
\newcommand{\er}{\end{red}}

\newenvironment{blue}
{\relax\color{Dandelion}}
{\hspace*{.5ex}\relax}

\newcommand{\beb}{\begin{blue}}
\newcommand{\eb}{\end{blue}}

\nc{\berm}{\ber {}\marginnote{\fbox{\scshape\lowercase{M}}}}
\nc{\bers}{\ber {}\marginnote{\fbox{\scshape\lowercase{S}}}}
\nc{\bermh}{\ber {}\marginnote{\fbox{\scshape\lowercase{MH}}}}
\nc{\berE}{\ber {}\marginnote{\fbox{\scshape\lowercase{E}}}}

\newcommand{\C}{{\mathbb C}}
\newcommand{\Q}{\mathbb {Q}}
\newcommand{\Z}{{\mathbb Z}}
\newcommand{\R}{\ms{1mu}{\mathbb R}}

\newcommand{\one}{{\bf{1}}}
\newcommand{\seteq}{\mathbin{:=}}

\newcommand{\hd}{{\mathrm{hd}}}      					 
\newcommand{\To}[1][{\hs{0.8ex}}]{\xrightarrow{\ms{7mu}{#1}\ms{7mu}}}

\newcommand{\Sp}{\mathrm{span}_{\R_{\ge0}}}  	

\newcommand{\g}{\mathfrak{g}}

\newcommand{\Uq}[1][{\mathfrak{g}}]{{U_q(#1)}}
\newcommand{\Uqm}[1][{\mathfrak{g}}]{{U_q^-(#1)}}
\newcommand{\Uqp}[1][{\mathfrak{g}}]{{U_q^+(#1)}}

\newcommand{\Hom}{\operatorname{Hom}}
\newcommand{\HOM}{\mathrm{H{\scriptstyle OM}}}
\newcommand{\END}{\mathrm{E\scriptstyle ND}}
\newcommand{\End}{\operatorname{End}}

\newcommand{\isoto}[1][]{\mathop{\xrightarrow%
[{\raisebox{.3ex}[0ex][.3ex]{$\scriptstyle{#1}$}}]%
{{\raisebox{-.6ex}[0ex][-.6ex]{$\mspace{2mu}\sim\mspace{2mu}$}}}}}

\newcommand{\Mod}{\text{-}\mathrm{Mod}}
\newcommand{\gmod}{\text{-}\mathrm{gmod}}

\newcommand{\proj}{\text{-}\mathrm{proj}}

\def\T{{\mathcal T}}

\newcommand{\conv}{{\mathbin{\scalebox{1.1}{$\mspace{1.5mu}\circ\mspace{1.5mu}$}}}}
\newcommand{\hconv}{\mathbin{\scalebox{.9}{$\nabla$}}}
\newcommand{\sconv}{\mathbin{\scalebox{.9}{$\Delta$}}}

\renewcommand{\Im}{\on{Im}}
\newcommand{\de}{\on{\textfrak{d}}}

\newcommand{\cmA}{\mathsf{A}}  
\newcommand{\wlP}{\mathsf{P}}   
\newcommand{\rlQ}{\mathsf{Q}}   
\newcommand{\weyl}{\mathsf{W}}  
\newcommand{\prD}{\Delta_+}            
\newcommand{\nrD}{\Delta_-}            
\newcommand{\sg}{\mathfrak{S}}   
\newcommand{\Po}{\wlP}
\newcommand{\rtlp}{\rtl_+}
\newcommand{\qQ}{\mathcal{Q}}
\newcommand{\bQ}{\overline{\qQ}}

\newcommand\Aq[1][{\n(w)}]{A_q(#1)}  
\newcommand{\Aqn}{\Aq[\mathfrak n]}  

\newcommand{\wt}{\mathrm{wt}} 		
\nc{\cor}{\mathbf{k}\ms{1mu}}
\newcommand{\bR}{\cor}
\nc{\corp}{\cor}
\newcommand{\catC}{ \mathscr{C}}  	
\newcommand{\tcatC}{ \widetilde{\mathscr{C}}}  	
\newcommand{\catT}{ \mathcal{T}}  	
\newcommand{\lT}{ \widetilde{\mathcal{T}}}  	

\newcommand{\catTc}{ \mathcal{T}_{\mathrm{br}}}  	
\newcommand{\lRg}[1][w]{ \tcatC_{#1} }  	
\newcommand{\dM}{ \mathsf{M }}              
\newcommand{\dC}{ \mathsf{C }}              
\newcommand{\gW}{\mathsf{W}}
\newcommand{\sgW}{\mathsf{W}^*}
\newcommand{\tf}{{\widetilde{f}}}  		
\newcommand{\te}{{\widetilde{e}}}  		
\newcommand{\tF}{\widetilde{F}}  		
\newcommand{\tE}{\widetilde{E}}  		
\newcommand{\ep}{\varepsilon}  		
\newcommand{\ph}{\varphi}  		
\newcommand{\trivialM}{\mathbf{1}} 	

\newcommand{\Ht}{\mathrm{ht}} 		

\newcommand{\nR}{\mathrm{R}^{\mathrm{norm}}} 		
\newcommand{\RR}{\mathrm{R}} 				
\newcommand{\coR}{R\ms{1mu}} 				
\newcommand{\La}{\Lambda} 			
\newcommand{\tLa}{\widetilde{\Lambda}} 			
\newcommand{\Dd}{\text{ \textfrak{d}}} 			
\newcommand{\Res}{\mathrm{Res}\ms{1mu}} 			

\newcommand{\Ma}{{\ms{1.5mu}\widehat{\mathsf{M}}}}

\newcommand{\Na}{{\widehat{\mathsf{N}}}}
\newcommand{\Mm}{{\ms{1mu}\mathsf{M}}}
\newcommand{\z}{{z_\Ma}}

\newcommand{\gr}{\mathrm{gr}}
\newcommand{\triv}{{\mathbf{1}}}   				
\newcommand{\id}{\ms{1mu}{\mathsf{id}}\ms{1mu}}   				
\newcommand{\Ds}{{\mathcal{D}}}   				
\newcommand{\Hm}{{\mathrm{H}}}   				
\newcommand{\gHm}{\mathrm{H}^{\gr}}   				
\newcommand{\dphi}{{\phi}}   				
\newcommand{\gH}{\mathrm{H}}   				
\newcommand{\gL}{\mathrm{L}}   				
\newcommand{\gzeta}{\zeta^\gr}   				
\newcommand{\gT}{T^\gr}   				
\newcommand{\gtT}{\widetilde{T}^\gr}   				

\newcommand{\lG}{\Gamma}   					

\nc{\be}{\begin{enumerate}}
\newcommand{\bnum}{\be[{\rm(i)}]}
\newcommand{\bna}{\be[{\rm(a)}]}

\newcommand{\rtl}{\rlQ}

\newcommand{\etens}{\boxtimes}

\newcommand{\rmat}[1]{\ms{1.5mu}{\mathbf{r}}_%
{\mspace{-2mu}\raisebox{-.6ex}{${\scriptstyle{#1}}$}}}

\newcommand{\Rr}{\rmat}
\newcommand{\shc}{\mathcal{C}}
\newcommand{\sht}{\mathcal{T}}

\newcommand{\tC}{\widetilde{C}}

\newcommand{\Ob}{\on{Ob}}

\nc{\ms}{\mspace}
\nc{\cl}{\colon}
\nc{\ro}{{\rm (}}
\nc{\rf}{{\rm )}\xspace}
\nc{\noi}{\noindent}
\nc{\bl}{\bigl(}
\nc{\br}{\bigr)}

\newenvironment{myequation}
{\relax\setlength{\arraycolsep}{1pt}\begin{eqnarray}}
{\end{eqnarray}}
\newenvironment{myequationn}
{\relax\setlength{\arraycolsep}{1pt}\begin{eqnarray*}}
{\end{eqnarray*}}

\nc{\eq}{\begin{myequation}}
\nc{\eneq}{\end{myequation}}
\nc{\eqn}{\begin{myequationn}}
\nc{\eneqn}{\end{myequationn}}

\newenvironment{myarray}[1]{\relax\setlength{\arraycolsep}{1pt}
\begin{array}{#1}}{\end{array}\relax}

\newcommand{\ba}{\begin{myarray}}
\newcommand{\ea}{\end{myarray}}

\nc{\hs}{\hspace*}
\nc{\set}[2]{\left\{{#1}\mid{#2}\right\}}
\nc{\snoi}{\smallskip\noi}
\nc{\al}{\alpha}
\nc{\rmz}{\setminus\{0\}}
\nc{\tens}{\otimes}
\nc{\vphi}{\varphi}
\nc{\ee}{\end{enumerate}}
\nc{\la}{\lambda}
\nc{\bc}{\begin{cases}}
\nc{\ec}{\end{cases}}
\nc{\qtq}[1][and]{\quad\text{#1}\quad}
\nc{\qt}[1]{\quad\text{#1}}
\nc{\dual}{{\displaystyle{\ms{1mu}\star}}}
\nc{\wle}{\preceq}
\nc{\epito}{\twoheadrightarrow}
\nc{\epiTo}[1][]{\xymatrix@C=4ex{{}\ar@{->>}[r]^-{#1}&{}}}
\nc{\Proof}{\begin{proof}}
\nc{\lan}{\langle}
\nc{\ran}{\rangle}
\nc{\ang}[1]{\lan{#1}\ran}
\nc{\QED}{\end{proof}}
\nc{\soplus}{\mathop{\scalebox{.65}{\raisebox{.2ex}{$\displaystyle\bigoplus$}}}}
\nc{\eps}{\varepsilon}
\nc{\on}{\operatorname}
\nc{\supp}{\on{supp}}
\nc{\sct}{strongly commute\xspace}
\nc{\scts}{strongly commutes\xspace}
\nc{\bce}{\eta}			
\nc{\height}[1]{\vert{#1}\vert}
\nc{\braid}{{\ms{1mu}\mathrm{br}}}
\nc{\gp}{\mathfrak{p}}
\nc{\wtl}{\wlP}
\nc{\ra}{real and admits an affinization}
\nc{\ras}{real and admit affinizations}
\nc{\shf}{\mathcal{F}}
\nc{\Cw}{\catC_w}
\nc{\tCw}{\lRg}
\nc{\akew}[1][2ex]{\rule[-1ex]{#1}{0ex}}
\nc{\ake}[1][2ex]{\rule[-1ex]{0ex}{#1}}
\nc{\akete}[1][2ex]{\rule[#1]{0ex}{#1}}
\nc{\tRm}{(R\gmod)\widetilde{\mbox{$\ake[2.5ex]\akew[.9ex]$}}}
\nc{\monoTo}[1][]{\xymatrix{\ar@{>->}[r]^-{{#1}}&}}
\nc{\monoto}[1][]{\rightarrowtail}
\nc{\tX}{\widetilde{X}}
\nc{\corps}{\corp}
\nc{\tL}{\widetilde{L}}
\nc{\prtl}{\rtl_+}
\nc{\tK}{\widetilde{K}}
\nc{\tep}{\widetilde\ep}
\nc{\teps}{\widetilde\ep}
\nc{\teta}{\widetilde\eta}
\nc{\ga}{\mathfrak{a}}
\nc{\scbul}{{\,\raise1pt\hbox{$\scriptscriptstyle\bullet$}\,}}
\nc{\bwr}{\mbox{\large$\wr$}}
\nc{\tR}{\widetilde{\mathscr{R}}}
\nc{\lS}{\mathsf{S}}
\nc{\lZ}{\mathcal{Z}}
\nc{\prolim}{\mathop{\varprojlim}\limits}
\nc{\sym}{\sg}

\nc{\txi}{\tilde{\xi}}
\nc{\rl}{\rlQ}

\nc{\rev}{{\ms{1mu}\mathrm{rev}}}

\nc{\Prob}{\begin{problem}}
\nc{\enprob}{\end{problem}}
\nc{\Quest}{\begin{question}}
\nc{\enques}{\end{question}}
\nc{\Conj}{\begin{conjecture}}
\nc{\enconj}{\end{conjecture}}

\nc{\bigtens}{\mathop{\scalebox{.75}%
{\raisebox{.2ex}{$\displaystyle\bigotimes$}}\ms{2mu}}\limits}
\nc{\ble}{\preccurlyeq}
\nc{\bge}{\succcurlyeq}
\nc{\st}[1]{\{{#1}\}}
\nc{\wg}{affinizable real simple\xspace}
\nc{\re}{\mathrm{re}}

\nc{\Sym}{\mathfrak{S}}
\nc{\Rmat}{\mathrm{R}}
\nc{\Rn}{\Rmat^{\mathrm{norm}}}
\nc{\Rnor}{\mathrm{R}^{\mathrm{norm}}}
\nc{\cct}{\mathbin{\scalebox{1.2}{$\star$}}}
\nc{\mnoi}{\medskip\noi}
\nc{\Cs}[1][w]{\catC_{{#1}}^{\ms{2mu}*}}
\nc{\tCs}[1][w]{ \widetilde{\mathscr{C}}^{\ms{2mu}*}_{#1}}
\nc{\pwtl}{\wlP_+}
\nc{\stens}{\mathop\otimes\limits}

\newcommand{\indlim}[1][]{\mathop{\varinjlim}\limits_{#1}}
\nc{\afr}{affreal\xspace}
\nc{\ssim}{\raisebox{-.8ex}[.5ex][.2ex]{$\sim$}}

\numberwithin{equation}{section}

\title[Localizations for quiver Hecke algebras II]
{Localizations for quiver Hecke algebras II}

\author[M. Kashiwara]{Masaki Kashiwara}
\thanks{The research of M.\ Kashiwara
was supported by Grant-in-Aid for Scientific Research (B) 20H01795,
Japan Society for the Promotion of Science.}
\address[M. Kashiwara]{
Kyoto University Institute for Advanced Study, Research Institute
for Mathematical Sciences, Kyoto University, Kyoto 606-8502, Japan
\& Korea Institute for Advanced Study, Seoul 02455, Korea }
\email[M. Kashiwara]{masaki@kurims.kyoto-u.ac.jp}

\author[M. Kim]{Myungho Kim}
\address[M. Kim]{Department of Mathematics, Kyung Hee University, Seoul 02447, Korea}
\email[M. Kim]{mkim@khu.ac.kr}
\thanks{The research of M.\ Kim was supported by the National Research Foundation of
Korea (NRF) Grant funded by the Korea government(MSIP)
(NRF-2017R1C1B2007824  and NRF-2020R1A5A1016126).}

\author[S.-j. Oh]{Se-jin Oh}
\thanks{ The research of S.-j.\ Oh was supported by the Ministry of Education of the Republic of Korea and the National Research Foundation of Korea (NRF-2022R1A2C1004045).}
\address[S.-j. Oh]{Department of Mathematics, Ewha Womans University, Seoul 03760, Korea}
\email[S.-j. Oh]{sejin092@gmail.com}

\author[E. Park]{Euiyong Park}
\thanks{The research of E.\ Park was supported by the National Research Foundation of Korea (NRF) Grant funded by the Korea Government(MSIP)(NRF-2020R1F1A1A01065992 and NRF-2020R1A5A1016126).}
\address[E. Park]{Department of Mathematics, University of Seoul, Seoul 02504, Korea}
\email[E. Park]{epark@uos.ac.kr}

\keywords{Categorification, Localization, Monoidal category, Quantum unipotent coordinate ring, Quiver Hecke algebra}

\subjclass[2010]{18D10, 16D90,  81R10}

\date{August 2}

\begin{document}

\maketitle

\begin{abstract}
We prove that  the  localization  $\tcatC_w$ of the monoidal category $\catC_w$ is rigid, and  the category $\catC_{w,v}$ admits a localization via a real commuting family of central objects. 

For a  quiver Hecke algebra $R$ and an element $w$ in the Weyl group, the subcategory $\catC_w$ of the category $R\gmod$ of finite-dimensional graded $R$-modules  categorifies the quantum unipotent coordinate ring $\Aq$.
In  the previous paper,  we constructed 
 a monoidal category $\tcatC_w$  such that it contains $\catC_w$ and the objects $\set{\Mm(w\La_i,\La_i)}{i\in I}$ corresponding to  the  frozen variables are invertible.
 In this paper, we show that there is a monoidal equivalence between the category $\tcatC_w$ and $(\tcatC_{w^{-1}})^\rev$. 
Together with the  already known  left-rigidity of $\tcatC_w$,
it follows that the monoidal category $\tcatC_w$  is rigid.

If $v\ble w$ in the Bruhat order, there is a  subcategory $\catC_{w,v}$ of $\catC_w$ which categorifies the doubly-invariant algebra $^{N'(w)} \C[N]^{N(v)}$. 
We prove  that the family $\bl \Mm(w\La_i,v\La_i)\br_{i\in I}$ of simple $R$-module forms a real commuting family of graded central objects in the category $\catC_{w,v}$ so that there is a localization $\tcatC_{w,v}$ of $\catC_{w,v}$ in which $\Mm(w\La_i,v\La_i)$ are invertible.
Since the localization of the algebra $^{N'(w)} \C[N]^{N(v)}$ 
by the family of  the isomorphism classes of  $\Mm(w\La_i,v\La_i)$ is isomorphic to the coordinate ring $\C[R_{w,v}]$ of  the open Richardson variety associated with $w$ and $v$, 
the localization $\tcatC_{w,v}$ categorifies
the coordinate ring $\C[R_{w,v}]$.

\end{abstract}

\setcounter{tocdepth}{4}
\tableofcontents

\section*{Introduction}
In the previous work \cite{KKOP21}, we developed a general procedure for localizations of monoidal categories  and  studied in detail the case that the categories consist of modules over quiver Hecke algebras. 
This paper is a continuation of \cite{KKOP21}.
Roughly speaking, the localization of a  monoidal category in \cite{KKOP21} is a procedure to find  a larger monoidal category in which the prescribed objects are invertible.
Let $\cor$ be a commutative ring. For a $\cor$-linear monoidal category $\catT$, a pair $(C,\coR_C)$ of an object $C$ and a natural transformation $\coR_C\cl(C\tens -) \to (-\tens C)$ is called a \emph{braider in $\catT$} if $\coR_C$ is compatible with the tensor product $\tens$.
A family $\{(C_i,\coR_{C_i})\}$ of braiders in $\catT$ is called a \emph{real commuting family} if
$\coR_{C_i}{(C_i)}\in \cor^\times \id_{C_i\tens C_i}$ and  $\coR_{C_i}(C_j) \circ \coR_{C_j}(C_i)\in \cor^\times \id_{C_j\tens C_i}$. 
Then one can construct a  $\cor$-linear monoidal category $(\lT, \tens ,\one)$ and a monoidal functor $\Phi\cl \catT \to \lT$ such that the object $\Phi(C_i)$ are invertible and the morphisms $\Phi(R_{C_i}(X))\cl \Phi(C_i)\tens \Phi(X) \to \Phi(X)\tens \Phi(C_i)$ are isomorphisms for all $i$ and all $X \in \catT$.
Moreover the pair $(\lT,\Phi)$ is universal with respect to these properties.  
By the construction, any object in $\lT$ is of the form $X \tens (\displaystyle\otimes_i C_i^{\tens a_i})$ for some $X \in \catT$ and $a_i\in \Z$.
There is  also a \emph{graded version} of localization.
Assume that a graded monoidal category $\T$ has a decomposition $\T = \soplus_{\la \in \La} \T_\la$ for some abelian group $\La$, which is compatible with $\tens$, and the grading shift operator $q$.
Then one can define the notion of \emph{graded braider} $(C,\coR_C,\phi_C)$, where 
$\phi\cl\La \to\Z$ is a group homomorphism and  $\coR_C(X) \in  \Hom_{\T} (C\tens X, q^{\phi(\la)} X \tens C)$ for $X\in\catT_\la$.
For a real commuting family of graded braiders in $\T$, there exists a monoidal category $\lT$ and a functor $\Phi\cl \catT \to \lT$ which have the same properties as in the ungraded cases.

\smallskip
One of the motivations to develop such a general procedure is to localize monoidal categories consisting of modules over \emph{quiver Hecke algebras}.
Let $\g$ be a symmetrizable Kac-Moody algebra and $\rtl_+$ the root lattice of $\g$.
The quiver Hecke algebra associated with $\g$  is a family $\st{R(\beta)}_{\beta\in \rtl_+}$ of $\Z$-graded associative algebras over $\cor$ 
such that the Grothendieck ring $K(R\gmod)$ is isomorphic to the 
\emph{quantum unipotent coordinate ring} $\Aqn$ which is isomorphic to the dual $(\Uqp)^*$ of the half of the quantum group $\Uq$ (\cite{KL09, R08}).
Here $R\gmod$ denotes the direct sum of the categories of finite-dimensional graded $R(\beta)$-modules.
For an $R(\beta)$-module $M$ and an $R(\gamma)$-module $N$, the \emph{convolution product} 
$M\conv N$ is the $R(\beta+\gamma)$-module induced from the $R(\beta)\tens R(\gamma)$-module $M\tens N$ through the (non-unital) algebra embedding $R(\beta)\tens R(\gamma) \to R(\beta+\gamma)$.
The category $R\gmod$ together with the convolution product  is a monoidal category and the convolution product corresponds to the multiplication of the algebra $\Aqn$. 
One of main advantages in the case  $\catT=R\gmod$ is that for any simple module $C$ in $R\gmod$, there exists a unique \emph{non-degenerate graded braider}  $(C,\coR_C, \phi_C)$.  Here a braider is non-degenerate  if $\coR_C(L(i))$ doesn't vanish for any $i$, where $L(i)$ denotes the unique graded simple module over $R(\al_i)$ and $\set{\al_i}{i\in I}$  is the set of simple roots of $\g$.
Hence one can consider localizations for various monoidal subcategories of $R\gmod$.

For an element $w$ in the Weyl group $W$ of $\g$, there is a subalgebra $\Aq$ of $\Aqn$ called the \emph{quantum unipotent coordinate ring  associated with $w$}
whose  limit  at $q=1$ becomes the coordinate ring of the unipotent subgroup $N(w)$ associated with $w$. Note that the Lie algebra of $N(w)$ is  $\mathfrak n(w)\seteq\hs{-3ex}\displaystyle\soplus_{\alpha \in \Delta_+\cap w\Delta_-}\hs{-2ex} \g_{\alpha}$, where $\Delta_{\pm}$ is the set of positive/negative roots of $\g$.
The algebra $\Aq$ is interesting since it equips a quantum cluster algebra structure (\cite{GLS13, GY14,GY17}).  
We denote the set of frozen variables by $\{D(w\La_i,\La_i)\}_{i\in I}$. 
Note that the element $D(w\La_i,\La_i)$ is not invertible in the algebra $\Aq$. 
If one localizes $\Aq$ at the set $\{D(w\La_i,\La_i)\}_{i\in I}$, then one gets a $q$-deformation $\Aq[N^w]$ of the coordinate ring $\C[N^w]$, where $N^w$ denotes  the \emph{unipotent cell associated with $w$} which can be identified with an open subset of $N(w)$ (\emph{De Concini-Procesi isomorphism}, see \cite[Theorem 4.13]{KO21}).
Now one can naturally associate the algebra $\Aq$ with a full subcategory $\catC_w$ of $R\gmod$ 
whose Grothendieck ring $K(\catC_w)$ is isomorphic to $\Aq$.
It is interesting  not only that the Grothendieck ring is isomorphic to $\Aq$, but also the category $\catC_w$ reflects the quantum cluster algebra structure on $\Aq$. Indeed, every cluster monomial in $\Aq$ corresponds to a real simple module in $\catC_w$, provided $\g$ is symmetric and $\cor$ is a field of characteristic zero (\cite{KKKO18}). 
Now the localization at the category level is exactly as one might imagine: each element $D(w\La_i,\La_i)$ corresponds to a simple module  $\Mm(w\La_i,\La_i)$ in  $\catC_w$, and the set $(\Mm(w\La_i,\La_i),R_{\Mm(w\La_i,\La_i)}, \phi_{\La_i})_{i\in I}$ forms a real commuting family of graded braiders in $\catC_w$ (\cite[Proposition 5.1]{KKOP21}). 
Hence the localization $\tcatC_w$ of $\catC_w$ with respect to the family categorifies the $q$-deformation  $\Aq[N^w]$ of the coordinate ring $\C[N^w]$ of the unipotent cell $N^w$ (\cite[Corollary 5.4]{KKOP21}). 
We emphasize that the set $(\Mm(w\La_i,\La_i),R_{\Mm(w\La_i,\La_i)}, \phi_{\La_i})_{i\in I}$ is also a real commuting family of graded braiders in the category $R\gmod$ and hence one has the localization $\tRm[w]$ of $R\gmod$ with respect to the family. Let us denote $Q_w\cl R\gmod \to \tRm[w]$ the localization functor.  Since the composition $\catC_w \hookrightarrow R\gmod \To[Q_w] \tRm[w]$ factors through the localization $\tcatC_w$, one obtains a functor $\iota_w\cl \tcatC_w \to \tRm[w]$, which turns out to be an equivalence (\cite[Theorem 5.9]{KKOP21}). 
This property  enables us  in particular  to show that the monoidal category  $\tcatC_w$ is left-rigid,i.e., every object in $\tcatC_w$ admits a left-dual (\cite[Corollary 5.11]{KKOP21}).  
This remarkable feature  can be understood as a (monoidal) categorification of the quantum twisted map on the algebra $\Aq[N^w]$ (see \cite{KO21} and references therein). 

\smallskip
One of the main results of this paper is that the category  $\tcatC_w$  is also right-rigid, 
i.e., every object in $\tcatC_w$ admits a right-dual (Theorem \ref{th:rigid}).  
We achieve this by showing that there is a monoidal equivalence between the category $\tcatC_w$ and $(\tcatC_{w^{-1}})^\rev$ (Theorem \ref{th:main1}). 
Here $\catT^\rev$ denotes the monoidal category $(\catT, \tens^\rev)$ where the reversed tensor product $\tens^\rev$ is defined by
$X\tens^\rev Y \seteq Y\tens X$ and $f\tens^\rev g\seteq g\tens f$ for any objects $X,Y 
$ and morphisms $f,g\in \catT$. 
Then the left-rigidity of $\tcatC_{w^{-1}}$ implies the right-rigidity of $\tcatC_{w}$.
The strategy for constructing the equivalence is briefly as follows.
There is an algebra automorphism $\psi$ on $R(\beta)$ (see \eqref{eq:psi})  
which induces a monoidal equivalence $\psi_*\cl R\gmod \to (R\gmod)^\rev$.
Then the composition $R\gmod\To[\psi_*] (R\gmod)^\rev\To[Q_{w^{-1}}] (\tRm[{w^{-1}}])^\rev$ factors as
$R\gmod \To[Q_w] \tRm[w]  \To[F_{w^{-1}}] (\tRm[{w^{-1}}])^\rev$, and the functor $F_{w^{-1}}$ is the desired equivalence of categories. 
One of key conditions to get such a factorization is that $(Q_{w^{-1}}\circ \psi_*)(\Mm(w\La_i,\La_i))$ is invertible in 
$R\gmod[{w^{-1}}]\simeq \tcatC_{w^{-1}}$ for each $i\in I$.
Hence 
we need to study the structure of modules $\psi_*(\Mm(w\La_i,\La_i))$. 
Recall that the modules $\Mm(w\La_i,\La_i)$ 
are examples of the \emph{determinantial modules} which have been studied in detail in \cite{KKOP18, KKOP21}. 
In general,  for a dominant integral weight $\La$ and Weyl group elements $v\ble u$ in the Bruhat order, there exists a distinguished element $D(u\La, v\La)$ of $\Aqn$, called the \emph{unipotent quantum minor}, and the \emph{determinantial module} $\Mm(u\La,v\La)$ is the simple module corresponding to $D(u\La, v\La)$ under the isomorphism  $K(R\gmod)\simeq \Aqn$. 
Even the module $\psi_*(\Mm(w\La_i,\La_i))$ 
is no longer a determinantial module in general, it turns out that it  shares many of properties of determinantial modules. 
We characterize such a  family of simple modules 
 and call them the \emph{generalized determinantial modules} (see Theorem \ref{thm:gdm}). 
It enables us to calculate the module $\psi_*(\Mm(w\La_i,\La_i))$   quite explicitly and to see that  $(Q_{w^-1} \circ \psi_*)(\Mm(w\La_i,\La_i))$ is invertible in  $\tcatC_{w^{-1}}$. 

\smallskip
The other main result of this paper is a localization of the category $\catC_{w,v}$ for a pair of Weyl group elements $w,v$ such that  $v\ble w$ in the Bruhat order.
The category $\catC_w$ can be characterized as the full subcategory consisting of modules $M\in R(\beta)\gmod$ such that $\Res_{\gamma,\beta-\gamma} (M) \not= 0$ 
implies that 
$\gamma \in \rtl_+\cap w\rtl_-$. 
Here $\Res_{\al,\beta}$  denotes the restriction functor $R(\al+\beta)\gmod \to R(\al) \tens R(\beta)\gmod $.
Similarly, we define the category $\catC_{*, v}$  as the full subcategory $R\gmod$ consisting of modules $M\in R(\beta)\gmod$ such that $\Res_{\beta-\gamma, \gamma} (M) \not= 0$ 
implies that  $\gamma \in \rtl_+\cap v\rtl_+$. 
We set $\catC_{w,v}\seteq\catC_{w} \cap \catC_{*,v}$. 
Then the Grothendieck ring $K(\catC_{w,v})$  can be understood as a $q$-deformation of the doubly-invariant algebra $^{N'(w)} \C[N]^{N(v)}$, where $N$ is the unipotent radical, $N_-$ is the opposite of $N$, and 
$N'(w)\seteq N\cap(wNw^{-1})$, $N(v)\seteq N\cap(vN_-v^{-1})$ (see \cite[Remark 2.19]{KKOP18}).
It is known that the localization $^{N'(w)} \C[N]^{N(v)}$ at the set $\st{D(w\La_i,v\La_i)}_{i \in I}$ is isomorphic to the coordinate ring $\C[R_{w,v}]$ of the \emph{open Richardson variety $R_{w,v}$ associated with $w$ and $v$} (\cite[Theorem 2.12]{Lec16}).  
Hence a localization of the category $\catC_{w,v}$ with respect to the set of determinantial modules $\{\Mm(w\La_i,v\La_i)\}_{i\in I}$ would give a categorification of (a $q$-deformation of) the coordinate ring $\C[R_{w,v}]$. 
We  show that there exists a  graded braider $(\Mm(w\La_i,v\La_i),R_{\Mm(w\La_i,v\La_i)},\phi_{w,v,\La_i})$ in the category $\catC_{*,v}$ (Proposition \ref{prop:psiwvla}). 
The key idea for this is to take a restriction of the braider $\coR_{\Mm(w\La_i,\La_i)}$ in $R\gmod$. 
Indeed, for an object $X\in \catC_{*,v}\cap \bl R(\gamma)\gmod \br $ one can show that 
the restriction $\Res_{\gamma+\alpha, \beta} \bl \coR_{\Mm(w\La_i,\La_i) }(X)\br $ is equal to
$\bl\Mm(w\La_i,v\La_i)\conv X\br\tens \Mm(v\La_i,\La_i)\to 
\bl X\conv \Mm(w\La_i,v\La_i)\br\tens \Mm(v\La_i,\La_i)$, where  $\beta = \La_i-v\La_i$, and $\al=v\La_i-w\La_i$.
We have such a nice form of restriction because the pairs 
$\bl\Mm(w\La_i,v\La_i),\Mm(v\La_i,\La_i)\br$  and $\bl X,  \Mm(v\La_i,\La_i)\br$ are distinguished pairs of simple modules called \emph{unmixed pairs} (see Section \ref{subsubsec:unmixed}).
Now since  $\End(\Mm(v\La_i,\La_i))\simeq \cor$, we obtain the desired
 homomorphism $R_{\Mm(w\La_i,v\La_i)}(X) $ from  $\Mm(w \La_i ,v\La_i)\conv X$ to $X\conv \Mm(w\La_i,\La_i)$.
 It follows that the family
 $\{(\Mm(w\La_i,v\La_i),R_{\Mm(w\La_i,v\La_i)},\phi_{w,v,\La_i})\}_{i\in I}$ is a real commuting family of graded braiders in the category $\catC_{*,v}$ due to the corresponding properties of 
 the family  in $R\gmod$
$\{(\Mm(w\La_i,\La_i), \coR_{\Mm(w\La_i,\La_i)}, \phi_{\La_i})\}_{i\in I}$.  
Let $\tcatC_{*,v}[w]$ and $\tcatC_{w,v}$ be the localization of  $\catC_{*,v}$ and $\catC_{w,v}$, respectively via 
the family $(\Mm(w\La_i,v\La_i),R_{\Mm(w\La_i,v\La_i)},\phi_{w,v,\La_i})$. 
Then similarly to the case of $\tRm[w]$ and $\tcatC_{w}$, there is  a monoidal equivalence between  $\tcatC_{*,v}[w]$ and $\tcatC_{w,v}$ (Theorem \ref{thm:equiv}).
 It is expected that the category  $\tcatC_{w,v}$  gives a monoidal categorification of a $q$-deformation of the cluster algebra arising from the open Richardson variety $R_{w,v}$ given in \cite{Lec16}. 

\smallskip
Let us explain some miscellaneous results in this paper which are not only used for the main theorems but also interesting by themselves.
We characterize the simple modules that vanish under the localization functor $Q_w\cl R\gmod \to \tRm[w]$. 
Recall that the self-dual simple modules in $R\gmod$ are in bijection with the crystal basis $B(\infty)$ of $\Aqn$ (\cite{LV11}).
It turns out that a simple module $M$ doesn't vanish under $Q_w$ if and only if $M$ matches an element  $b$ in $B_w(\infty)$, where $B_w(\infty)$ is a subset of $B(\infty)$  introduced in \cite{Kas93}, which is the limit of the Demazure crystals.
Let $I_w$ be the subspace of $\Aq[n]$ spanned by the upper global basis elements corresponding to the elements crystal basis in $B(\infty) \setminus B_w(\infty)$. 
Then $I_w$ is a two-sided ideal and the quotient $\Aqn/I_w$ is called the \emph{quantum closed unipotent cell} in \cite{Kimura12}. 
Then the equivalence $\iota_w \cl \tcatC_w \to \tRm[w]$ can be understood as a categorification of the isomorphism in \cite[Theorem 4.13]{KO21}, provided $\g$ is symmetric and $\cor$ is a field of characteristic zero (see Remark \ref{rm:KO}).

For simple modules $X\in R(\beta)\gmod$ and $Y\in R(\gamma)\gmod$ such that one of them is \emph{\afr} (see Definition \ref{def:affreal}), there is a  distinguished homomorphism  $\rmat{X,Y}\seteq X\conv Y \to Y\conv X$, called the \emph{$R$-matrix}.  
As the $R$-matrix is a crucial feature of the category $R\gmod$, the integers $\La(X,Y)$ and $\tLa(X,Y)$ (which is non-negative) play important roles
in the representation theory  of the quiver Hecke algebras (see, for example, \cite{KKOP18}).
Let $L$, $M$ and $N$ be  simple modules and  assume that $L$  is \afr. Then we show that  for any simple subquotient $S$ of $M\conv N$ we have $\tLa(L,M)\le \tLa(L,S)$  (Theorem \ref{th:strong}). 
This theorem is strong in the sense that  $\tLa(L,S)$ is bounded below by a number which doesn't involve with $N$ at all. 
A corollary  of this theorem (Corollary \ref{cor:Normal}) is an analogue of of \cite[Lemma 4.17]{KKOP20}, which holds in the category of finite-dimensional modules over quantum affine algebras. 

\smallskip
This paper is organized as follows.
In Section 1, we recall some preliminaries containing the localization of monoidal categories, quiver Hecke algebras, and determinantial modules, etc.
In Section 2, we develop some features in $R\gmod$, including unmixed pairs, normal sequences, and generalized determinantial modules. 
In Section 3, we show that there is a monoidal equivalence between $\tcatC_w$ and $\bl\catC_{w^{-1}}\br^\rev$ which implies that $\tcatC_w$ is rigid.
In Section 4. we construct a real commuting family of graded braiders $\{(\Mm(w\La_i,v\La_i),R_{\Mm(w\La_i,v\La_i)},\phi_{w,v,\La_i})\}_{i\in I}$ in the category $\catC_{*,v}$.  
It turns out that the localization $\catC_{*,v}[w]$ of $\catC_{*,v}$ and 
the localization $\tcatC_{w,v}$ of $\catC_{w,v}$ are monoidally equivalent.

\medskip

{\bf Acknowledgments}
 We thank Yoshiyuki Kimura for his suggestion and fruitful discussion.

\section{Preliminaries} \label{Sec: Preliminaries}
\subsection{Localizations of monoidal categories via braiders}\

\subsubsection{Monoidal categories} \label{subsub:mon}

A \emph{monoidal category} (or \emph{tensor category}) is a datum consisting of
\begin{enumerate}[\rm (a)]
\item a category $\catT$,
\item  a bifunctor $\cdot \otimes \cdot \cl \catT \times \catT \rightarrow \catT$,
\item an isomorphism $a(X,Y,Z)\cl (X \otimes Y) \otimes Z \buildrel \sim \over \longrightarrow X \otimes (Y \otimes Z)$ which is functorial in $X,Y,Z \in \catT$,
\item an object $\triv$, called an \emph{unit object}, endowed with
an isomorphism  $\epsilon\cl
\triv \otimes \triv \isoto\triv$
\end{enumerate}
such that
\be[{(1)}]
\item
the diagram below commutes for all $X,Y,Z,W \in \catT$:
$$
\xymatrix{
( ( X \otimes Y ) \otimes Z ) \otimes W  \ar[d]_{ a( X,Y,Z) \otimes W  }  \ar[rr]^{ a(X \otimes Y,Z,W) } &&  \ar[dd]^{ a( X,Y,Z \otimes W)  }    ( X \otimes Y ) \otimes (Z  \otimes W)   \\
 (  X \otimes (Y  \otimes Z) ) \otimes W  \ar[d]_{ a( X,Y\otimes Z, W)  }   &&    \\
 X \otimes ( (Y \otimes Z ) \otimes W ) \ar[rr]_ {X \otimes a( Y,Z,W)} &&    X \otimes ( Y \otimes (Z \otimes W )) \,,
}
$$
\item the functors $\catT\ni X\mapsto \one \tens X\in \catT$
and  $\catT\ni X\mapsto X\tens \one\in\catT$ are fully faithful.
\ee

We have canonical isomorphisms $\triv\tens X\simeq X\tens \triv\simeq X$
for any $X\in\catT$.
For $n \in \Z_{> 0}$ and $X\in\catT$, we set $X^{\otimes n} = \underbrace{X \otimes \cdots \otimes X}_{n \text{ times}}$, and $X^{\otimes\, 0} = \triv$.

For monoidal categories $\catT$ and $\catT'$, a functor $ F \cl  \catT \rightarrow \catT'$ is called \emph{monoidal}
if it is endowed with an isomorphism $\xi_F\cl F(X \otimes Y)  \buildrel \sim \over \longrightarrow F(X) \otimes F(Y)$ which is functorial in $X,Y \in \catT$ such that
the diagram
$$
\xymatrix{
F( ( X \otimes Y ) \otimes Z )   \ar[d]_{ \xi_F(X \otimes Y, Z)  }  \ar[rrr]^{ F( a(X , Y,Z)) } &&&  \ar[d]^{ \xi_F(X , Y \otimes Z)  }   F ( X \otimes (Y   \otimes  Z) )   \\
 F(  X \otimes Y)   \otimes F(Z)   \ar[d]_{ \xi_F(X , Y)\otimes F( Z)  }   &&& \ar[d]^{ F(X)\otimes \xi_F( Y \otimes Z)  }   F ( X )\otimes F(Y   \otimes  Z) )   \\
 ( F(X) \otimes F (Y) ) \otimes F(Z ) \ar[rrr]_ { a(F(X) , F(Y),F(Z))  } &&&     F ( X ) \otimes (F(Y)   \otimes  F(Z) )  
}
$$
commutes for all $X,Y,Z \in \catT$. We omit to write $\xi_F$ for simplicity.
A monoidal functor $F$ is  called \emph{unital} if $(F(\triv), F(\epsilon))$ is a unit object.
In this paper, we simply write
a ``monoidal functor'' for a unital monoidal functor.

We say that a monoidal category $\catT$ is an \emph{additive} (resp.\ \emph{abelian}) monoidal category if $\catT$ is additive (resp.\ abelian) and the bifunctor $\cdot \otimes \cdot$ is bi-additive.
Similarly, for a commutative ring $\bR$, a monoidal category $\catT$ is \emph{$\bR$-linear} if $\catT$ is $\bR$-linear and the bifunctor $\cdot \otimes \cdot$ is $\bR$-bilinear.

An object $X \in \catT$ is \emph{invertible} if the functors $\catT \rightarrow \catT$ given by $Z \mapsto Z \otimes X$ and $Z \mapsto X \otimes Z$ are equivalence of categories.
If $X$ is invertible, then one can find an object $Y$ and isomorphisms $f\cl X \otimes Y \buildrel \sim\over \rightarrow \triv$ and
$g\cl Y \otimes X \buildrel \sim\over \rightarrow \triv$ such that the diagrams below commute:
$$
\xymatrix{
X \otimes Y \otimes X  \ar[rr]^{ f \otimes X } \ar[d]_{ X \otimes g  } &&  \ar[d]  \triv \otimes X   \\
 X \otimes \triv \ar[rr]_ {  } &&   X,
}
\quad
\xymatrix{
Y \otimes X \otimes Y  \ar[rr]^{ g \otimes Y } \ar[d]_{ Y \otimes f  } &&  \ar[d]  \triv \otimes Y   \\
 Y \otimes \triv \ar[rr]_ {  } &&   Y.
}
$$
The triple $(Y,f,g)$ is unique up to a unique isomorphism. We write $Y = X^{\otimes -1}$.

\medskip
For a monoidal category $\sht$, we define a new monoidal category $\sht^\rev$ as the
 category $\sht$ endowed with the new bifunctor
$\tens^\rev$ defined by
$X\tens^\rev Y \seteq Y\tens X$ and $f\tens^\rev g\seteq g\tens f$ for any objects $X,Y$ in $\sht$ and   for any morphisms $f,g$ in $\sht$, respectively. 
The associativity constraints are given as  $a^\rev(X,Y,Z)\seteq a(Z,Y,X)^{-1}$.
The unit $(\one,\epsilon)$ of $\sht$ serves as a  unit of $\sht^\rev$, too. 
Let $\shf \cl\shc \to \sht$ be a monoidal functor. Then $\shf \cl\shc^\rev \to \sht^\rev$ is again a monoidal functor.

\medskip

A pair of  morphisms $\ep\cl X \otimes  Y \rightarrow \triv $ and $ \eta\cl \triv \rightarrow Y \otimes X$  in $\catT$ is 
called an \emph{adjunction}  if the composition
$X \simeq X \otimes \triv \To[{X \otimes \eta}] X \otimes Y \otimes X
\To[{\ep \otimes X}] \triv \otimes X \simeq X$
is the identity of $X$, and
 the composition $Y \simeq \triv \otimes Y \To[{\eta \otimes Y}]
 Y \otimes X \otimes Y \To[{Y \otimes \ep}]Y \otimes \triv \simeq Y$
is the identity of $Y$.
In the case when $(\ep, \eta)$ is an adjunction, we say that
$X$ is a \emph{left dual} to $Y$ and $Y$ is a \emph{right dual} to $X$ in $\catT$.
A monoidal category $\catT$ is \emph{left \ro respectively, right\rf rigid} if every object in $\catT$ has a left (respectively, right) dual. We call $\catT$ is  rigid, if it is left rigid and right rigid.

\subsubsection{Real commuting family of graded braiders} \label{Sec: RCB}  \

 In this subsection, we recall the notions of braiders and localization
introduced in \cite{KKOP21}. We refer the reader to loc.\ cit.\ for more details.

\begin{df}\ \label{def: braider}
A \emph{left braider}, simply a braider in the sequel, of a monoidal category $\catT$ is a pair $(C, R_C)$ of an object $C$ and a morphism
\begin{align*}
R_C(X)\cl  C \otimes X \longrightarrow X \otimes C
\end{align*}
which is functorial in $X \in \catT$ such that the following diagrams commutes:
\begin{equation} \label{Eq: central obj}
\begin{aligned}
\xymatrix{
C \otimes X \otimes Y   \ar[rr]^{R_C(X)\otimes Y}  \ar[drr]_{R_C(X \otimes Y)\ \ }  &  &   X \otimes C \otimes Y  \ar[d]^{X \otimes R_C(Y)}   \\
& &   X \otimes Y \otimes C  ,
}
\ \
\xymatrix{
C \otimes \triv   \ar[rr]^{R_C(\triv)}  \ar[drr]_{ \simeq }  &  &   \triv \otimes C   \ar[d]^{ \wr}   \\
& &    C.
}
\end{aligned}
\end{equation}

A braider $ ( C, R_{C}  )$ is called a \emph{central object} if $ R_{C}(X)$ is an isomorphism for any $X \in \catT$.

\end{df}

Let $\bR$ be a commutative ring and let  $\Lambda$ be a $\Z$-module. A $\bR$-linear monoidal category $\catT$ is \emph{$\Lambda$-graded} if $\catT$ has a decomposition
$ \catT = \bigoplus_{\lambda \in \Lambda} \catT_\lambda $ such that $\triv \in \catT_0$ and $\otimes$ induces a bifunctor $\catT_{\lambda} \times \catT_{\mu} \rightarrow \catT_{\lambda+\mu}$
for any $\lambda, \mu \in \Lambda$.
Let $q$ be an invertible central object in a $\Lambda$-graded category $\catT$, which belongs to $\catT_0$. 
We write $q^n$ ($n\in\Z$) for $q^{\tens n}$ for the sake of simplicity. 

\begin{df} \label{def:graded braider}
A \emph{graded braider} is a triple $(C, R_C, \dphi)$ of an object $C$, a $\Z$-linear map $\dphi\cl  \Lambda \rightarrow \Z$ and a morphism
$$
R_C(X) \cl  C \otimes X \longrightarrow q^{\dphi(\lambda)} \otimes X \otimes C
$$
such that the diagrams
$$
\xymatrix{
C \otimes X \otimes Y   \ar[rr]^{R_C(X)\otimes Y}  \ar[drr]_{R_C(X \otimes Y) \ \ }  &  &   q^{\dphi(\lambda)} \otimes X \otimes C \otimes Y  \ar[d]^{ X \otimes R_C(Y)}   \\
  & &   q^{ \dphi(\lambda+\mu) } \otimes X \otimes Y \otimes C
}
\qtq 
\xymatrix{
C \otimes \triv   \ar[rr]^{R_C(\triv)}  \ar[drr]_{ \simeq }  &  &   \triv \otimes C   \ar[d]^{ \wr}   \\
& &    C
}
$$
commute for any $X \in \catT_\lambda$ and $Y \in \catT_{\mu}$.
\end{df}

We denote by $\catTc$ the category of graded  braiders in $\catT$.
A morphism from $(C,R_C, \phi)$ to $(C', R_{C'},\phi')$ in $\catTc$ is 
a morphism $f \in \Hom_\catT(C,C')$ such that   $\phi=\phi'$ 
and the following diagram commutes for any $\la \in \Lambda$
and $X \in \catT_\la$:
$$
\xymatrix{
C \otimes X    \ar[rr]^{f \otimes X}  \ar[d]_{R_C(X )}  &  &  C' \otimes X  \ar[d]^{ R_{C'}(X)}   \\
q^{\phi(\la)} \tens X \otimes C  \ar[rr]^{q^{\phi(\la)} \tens  X \otimes f} & &   q^{\phi(\la)} \tens  X \otimes C'.
}
$$
For graded braiders $(C_1, R_{C_1}, \phi_1)$ and $(C_2, R_{C_2},\phi_1)$ of $\catT$, let
$ R_{C_1 \otimes C_2}(X)$ be the composition
$$
 C_1 \otimes C_2 \otimes X \To[{ R_{C_2}(X)}] q^{\phi_1(\la)}\tens C_1 \otimes X \otimes C_2
\To[{R_{C_1}(X) }]  q^{\phi_1(\la)+\phi_2(\la)}\tens  X \otimes C_1 \otimes C_2
$$
for $X \in \catT_\la$. Then $(C_1 \otimes C_2, R_{C_1\otimes C_2}, \phi_1+\phi_2) $ is also a graded braider of $\catT$.
Hence the category $\catTc$ is a monoidal category with a  a canonical faithful monoidal functor $\catTc \rightarrow \catT$.

\vskip 2em
Let $I$ be an index set and let $ \st{(C_i, R_{C_i}, \dphi_i )}_{i\in I} $ be a 
family of graded braiders.
We say that $\st{(  C_i ,  R_{C_i} ,  \dphi_i )}_{i\in I}$ is a \emph{real commuting family of graded braiders} in $\catT$ if
\bna
\item  \label{Eq: 1 in grcf} $ C_i \in \catT_{\lambda_i}$  for some $\lambda_i \in \Lambda$, and
$\dphi_i(\lambda_i) = 0$, $\dphi_i( \lambda_j ) + \dphi_j( \lambda_i ) = 0$,
\item  $R_{C_i}(C_i) \in \bR^\times \id_{C_i \otimes C_i}$ for $i\in I$,
\item   \label{Eq: 2 in grcf} $R_{C_j}(C_i) \circ R_{C_i}(C_j) \in \bR^\times \id_{C_i \otimes C_j}$ for $i,j\in I$.
\end{enumerate}

Set
$$
\lG \seteq  \Z^{\oplus I} \quad \text{ and } \quad \lG_{\ge 0} \seteq  \Z_{\ge 0}^{\oplus I}.
$$

Let $\{ e_i \mid i\in I \} $ be the natural basis of $\lG$.
We define a $\Z$-linear map
$$
\gL\cl  \lG \rightarrow \Lambda, \qquad  e_i \mapsto  \lambda_i \text{ for } i\in I,
$$
and a $\Z$-bilinear map
\begin{align*}
\dphi &\cl  \lG \times \Lambda \rightarrow \Z, \qquad (e_i, \lambda)  \mapsto \dphi_i(\lambda).
\end{align*}

We choose a $\Z$-bilinear map $\gH\cl    \lG \times \lG \rightarrow \Z$  such that  $ \dphi_i( \lambda_j ) =  \gH(e_i, e_j) - \gH(e_j, e_i) $ for any $i,j\in I$.
Then we have
\begin{align} \label{Eq: dphi gL}
 \dphi(\alpha, \gL(\beta)) = \gH(\alpha, \beta) - \gH(\beta , \alpha) \quad \text{for any $\alpha, \beta \in \lG $.}
\end{align}
 Let us denote by $\phi_\al$ the $\Z$-linear map $\phi(\alpha, -)\cl \Lambda \to \Z$ for each $\al \in \lG$.

\begin{lem}[{\cite[Lemma 2.3, Lemma 1.16]{KKOP21}}]
Let $\st{( C_i ,  R_{C_i}, \phi_i )}_{i\in I}$
be a real commuting family of graded braiders in  $\catT$.
\bnum
\item
There exists a family $\{\eta_{ij}\}_{i,j\in I}$
of elements in $\corp^\times$ such that
\eqn
R_{C_i}(C_i)&&=\eta_{ii}\; \id_{C_i \otimes C_i},\\
R_{C_j}(C_i) \circ R_{C_i}(C_j)&& = \eta_{ij}\eta_{ji}\;\id_{C_i \otimes C_j}
\eneqn
for all $ i ,j \in I$.

\item 
There exist
a graded braider $C^\al=(C^\alpha, R_{C^\alpha}, \phi_\alpha )$ for each $\alpha \in \lG_{\ge0}$,
and an isomorphism $\xi_{\alpha, \beta}\cl  C^\alpha \otimes C^\beta \buildrel \sim\over \longrightarrow q^{H(\al,\beta)}\tens C^{\alpha+\beta}$ in $\catT_{br}$ for $\alpha, \beta \in \lG_{\ge0}$
such that 
\bna
\item  $C^0=\triv$  and $ C^{e_i} = C_i $ for $i\in I$,
\item  the diagram in $\catT_{br}$ 
\begin{equation} \label{Eq: xi sum}
\begin{aligned}
\xymatrix{
C^\alpha \otimes C^\beta \otimes C^\gamma \ar[d]_{ C^\alpha \otimes \xi_{\beta, \gamma}} \ar[rr]^{\xi_{\alpha, \beta} \otimes C^\gamma} && \ar[d]^{\xi_{\alpha+\beta, \gamma}} q^{H(\al,\beta)}\tens C^{\alpha+\beta}\otimes C^\gamma \\
q^{H(\beta,\gamma)} \tens C^\alpha \otimes C^{\beta + \gamma} \ar[rr]^{\xi_{\alpha, \beta+\gamma}} && q^{H(\al,\beta)+H(\al,\gamma)+H(\beta,\gamma)} \tens C^{\alpha+\beta +\gamma}
}
\end{aligned}
\end{equation}
commutes for any $\alpha, \beta, \gamma \in \lG_{\ge0}$,
\item the diagrams  in $\catT_{br}$ 
\begin{equation} \label{Eq: CF ij}
\begin{aligned}
\ba{ccc}
\xymatrix{
C^0 \otimes C^0  \ar[d]^-{ \bwr } \ar[rr]^{\xi_{0,0} } && \ar[d]^-{ \bwr } C^{0}\\
\triv \otimes \triv  \ar[rr]^{ \simeq} && \ \ \triv  \ ,
}
\ba{c}\\[3ex]\qtq\ea\quad
\xymatrix{
C^\alpha \otimes C^\beta \ar[d]_{ \xi_{\alpha, \beta}} \ar[rr]^{R_{C^\alpha}(C^\beta) } && \ar[d]^{\xi_{ \beta, \alpha}}q^{\phi(\al,L(\beta))}\tens  C^\beta \otimes C^\alpha \\
q^{H(\al,\beta)}\tens C^{\alpha + \beta} \ar[rr]^{\bce(\alpha,\beta)\; \id_{C^{\alpha+\beta}} } && q^{H(\al,\beta)}\tens C^{\alpha+\beta}
}
\ea
\end{aligned}
\end{equation}

commute for any $i,j\in I$ and $\alpha, \beta, \gamma \in \lG_{\ge 0}$, where
\begin{align} \label{Eq: eta}
\bce(\alpha, \beta) \seteq  \prod_{i,j \in I} \bce_{i,j}^{a_ib_j} \in \bR^\times
\quad \text{for $\alpha = \sum_{i\in I} a_i e_i$ and $\beta = \sum_{j\in I} b_j e_j$ in $\lG$.}
\end{align}

\end{enumerate}
\end{enumerate}

\end{lem}

 Note that we have $\bce(\alpha,0) = \bce(0,\alpha)=1$, and
$
\bce(\alpha, \beta+\gamma) = \bce(\alpha, \beta)\cdot \bce(\alpha, \gamma)\ \  \text{ and } \ \  \bce(\alpha+ \beta, \gamma) = \bce(\alpha, \gamma)\cdot \bce(\beta, \gamma)
$
for $\alpha, \beta, \gamma \in \lG$.

\medskip

We define an order $\preceq$ on $ \lG$ by
$$
\alpha \preceq \beta  \quad \text{ for } \alpha, \beta \in \lG \text{ with }  \beta - \alpha \in \lG_{\ge0},
$$
and set
$$
\Ds_{\alpha_1, \ldots, \alpha_k} \seteq  \{ \delta \in \lG \mid \alpha_i + \delta \in \lG_{\ge0} \text{ for any }i =1, \ldots, k \}
$$
for $\alpha_1, \ldots, \alpha_k \in \lG$.

For $X \in \catT_\lambda$,  $Y \in \catT_\mu$  and $ \delta  \in \Ds_{\alpha, \beta}$,
we set
$$
\gHm_\delta( (X, \alpha  ), (Y, \beta  ) ) \seteq  \Hom_{\catT}(  C^{\delta + \alpha}\otimes X, q^{\gH(\delta, \beta-\alpha) + \dphi(\delta+\beta, \mu)} \otimes  Y \otimes C^{ \delta + \beta} ).
$$
For $ \delta, \delta' \in \Ds_{\alpha, \beta}$ with $\delta \preceq  \delta'$ and $ f \in \gHm_\delta( (X, \alpha  ), (Y, \beta  ) )$,
we define $\gzeta_{\delta', \delta}(f) \in \gHm_{\delta'}( (X, \alpha  ), (Y, \beta  ) )$  to be the morphism such that the following diagram commutes:
$$
\xymatrix@C=4em{
C^{\delta' - \delta} \otimes C^{ \delta + \alpha} \otimes X   \ar[dd]^{\bwr}_{  \xi_{\delta'-\delta,  \delta+\alpha} }  \ar[rr]^{ C^{\delta'-\delta} \otimes f \qquad \qquad }   &&
q^{\gH(\delta, \beta-\alpha) + \dphi(\delta+\beta, \mu)} \otimes   C^{\delta' - \delta} \otimes Y \otimes C^{ \delta+\beta}   \ar[d]^{  R_{C^{\delta'-\delta}} (Y) } \\
 &&  q^{\gH(\delta, \beta-\alpha) + \dphi(\delta'+\beta, \mu)} \otimes Y \otimes C^{\delta' - \delta} \otimes C^{ \delta + \beta}
  \ar[d]_{\bwr}^{\xi_{ \delta'-\delta,   \delta+\beta }}  \\
q^{\gH(\delta'-\delta,  \delta+\alpha)} \otimes C^{  \delta'+\alpha} \otimes X  \ar[rr]^{ q^{H(\delta'-\delta,\delta+\alpha)} \tens \gzeta_{\delta', \delta}(f) \qquad \qquad} &&
q^{\gH(\delta, \beta-\alpha) + \dphi(\delta'+\beta, \mu) + \gH(\delta'-\delta,  \delta+\beta) } \otimes Y \otimes C^{ \delta'+\beta}.
}
$$
Then,
$\gzeta_{\delta', \delta } $ is a map from $  \gHm_\delta( (X, \alpha  ), (Y, \beta  ) )$  to $ \gHm_{\delta'}( (X, \alpha  ), (Y, \beta  ) )$ and   $\zeta^\gr_{\delta'', \delta' }  \circ \zeta^\gr_{\delta', \delta } = \zeta^\gr_{\delta'', \delta } $
for $\delta \preceq \delta' \preceq \delta''$, so that $\{\zeta^\gr_{\delta', \delta} \}_{\delta,\delta' \in \Ds_{\al,\beta}}$ forms an inductive system indexed by $\Ds_{\al,\beta}$.

Hence we can define a new category $\lT$ as
\begin{align*}
\Ob (\lT) &\seteq  \Ob(\catT) \times \lG, \\
\Hom_{\lT}( (X, \alpha), (Y, \beta) ) &\seteq   \indlim[{  \substack{\delta \in \Ds_{\alpha, \beta}, \\  \lambda + \gL(\alpha) = \mu + \gL(\beta)  }    }]
\Hm^\gr_{\delta}( (X, \alpha  ), (Y, \beta  ) ) ,
\end{align*}
where $X \in \catT_{\lambda}$ and $Y \in \catT_{\mu}$.
For the composition of morphisms in $\lT$ and its associativity, see \cite[Section 2.2, Section 2.3]{KKOP21}.

 By the construction, we have the decomposition
$$
\lT = \bigoplus_{\mu \in \Lambda} \lT_{\mu}, \qquad \text{where }  \lT_{\mu} \seteq  \{ (X, \alpha) \mid X \in \catT_\lambda, \ \lambda + \gL(\alpha)=\mu  \}.
$$

The category $\lT$  is a monoidal category with the following  tensor product: For $\alpha, \alpha', \beta, \beta' \in \lG $, $X \in \catT_\lambda$, $X' \in \catT_{\lambda'}$, $Y \in \catT_\mu$ and $Y' \in \catT_{\mu'}$,
we define
$$
(X, \alpha) \otimes (Y, \beta) \seteq  ( q^{- \dphi(\beta, \lambda) + \gH(\alpha, \beta)} \otimes X \otimes Y, \alpha+\beta  ),
$$
and, for $f \in \gHm_\delta((X, \alpha), (X', \alpha'))$ and $g \in \gHm_\epsilon((Y, \beta), (Y', \beta'))$, we define
$$
\gT_{\delta, \epsilon}(f,g) \seteq  \eta(\epsilon, \alpha-\alpha') \gtT_{\delta, \epsilon}(f,g) ,
$$
where $ \gtT_{\delta, \epsilon}(f,g)$ is the morphism such that the following diagram commutes:
$$
\xymatrix{
C^{\delta+\alpha}  \otimes X \otimes C^{\epsilon+\beta}\otimes  Y  \ar[rr]^{f \otimes g} &&  q^{b}\otimes  X' \otimes C^{\delta+\alpha'}  \otimes  Y' \otimes  C^{\epsilon+\beta'}   \ar[d]^{R_{C^{\delta+\alpha'}} (Y') } \\
q^{-\dphi(\epsilon+\beta, \lambda) }\otimes C^{\delta+\alpha} \otimes C^{\epsilon+\beta}\otimes X \otimes Y  \ar[u]^{R_{C^{\epsilon+\beta}} (X) }  \ar[d]^\bwr_{\xi_{\delta+\alpha, \epsilon+\beta}}
&&  q^{c}\otimes  X' \otimes Y' \otimes C^{\delta+\alpha'} \otimes C^{\epsilon+\beta'}  \ar[d]_\bwr^{\xi_{\delta+\alpha', \epsilon+\beta'}} \\
q^{a}\otimes  C^{\delta+\epsilon+\alpha+\beta}\otimes X \otimes Y \ar[rr]^{\gtT_{\delta,\epsilon} (f, g)}  &&  q^{d}\otimes  X' \otimes Y' \otimes C^{\delta+\epsilon+\alpha'+\beta'},
}
$$
where
\begin{align*}
a &= -\dphi(\epsilon+\beta, \lambda) + \gH(\delta+\alpha, \epsilon+\beta), \\
b &=  \gH(\delta, \alpha' - \alpha) + \dphi(\delta+\alpha', \lambda') + \gH(\epsilon, \beta'-\beta) + \dphi(\epsilon+\beta', \mu'), \\
c &= b + \dphi(\delta + \alpha', \mu'), \qquad d = c + \gH(\delta + \alpha', \epsilon+\beta').
\end{align*}
Then we have
$$
\gT_{\delta, \epsilon}(f,g) \in \gHm_{\delta+\epsilon}((X, \alpha)\tens (Y,\beta), (X', \alpha') \otimes (Y', \beta') ).
$$
Then  the map $\gT_{\delta, \epsilon}$ is compatible with  the maps $\gzeta_{\delta,\delta'}$,  in the inductive system,  and moreover it yields a bifunctor $\tens$ on $\lT$ (\cite[Proposition 2.5]{KKOP21})
\begin{equation*} 
\begin{aligned} 
 \Hom_{\lT}(   (X, \alpha) , (X', \alpha') ) & \times  \Hom_{\lT}(  (Y, \beta), (Y', \beta') ) \buildrel \otimes \over  \longrightarrow \\
 & \qquad  \Hom_{\lT}(  (X, \alpha) \otimes (Y, \beta) , (X', \alpha') \otimes (Y', \beta') ).
\end{aligned}
\end{equation*}

For $(X,\alpha) \in  \lT$, define 
$R_{(q,0)}((X,\al)) \in \Hom_{\lT}((q\tens X,\al),(X\tens q,\al))$ as the image of $R_q(X) \in  \Hom_\catT(q\tens X, X\tens q)=\gHm_{-\alpha}( (q\tens X, \alpha ), (X \tens q, \alpha) )$. 
Then $((q,0), R_{(q,0)})$ is an invertible braider in $\lT$.

\begin{thm}  \label{Thm: graded localization}
Let $ \st{C_i=(C_i, R_{C_i}, \dphi_i )}_{i\in I} $ be a 
real commuting family of graded braiders  in $\catT$. Then the category
$\lT$ defined above becomes a monoidal category. 
There exists a monoidal functor 
$\Upsilon\cl \catT \to \lT$ 
and a real commuting family of graded braiders $\st{\tC_i=(\tC_i, R_{\tC_i}, \phi_i)}_{i\in I}$ in $\lT$ satisfy the following properties:

\bnum
\item
for $i\in I$, $\Upsilon(C_i) $ is isomorphic to $ \widetilde{C}_i$ and it is invertible in $(\lT)_{\,\mathrm{br}}$, 
\item 
for $i\in I$ and $X\in\catT_\la$, the diagram
$$
\xymatrix{
\Upsilon(C_i \otimes X)  \ar[r]^\sim  \ar[d]_{\Upsilon( R_{C_i} (X)  )\ms{10mu}}^-\bwr  & \widetilde{C}_i \otimes \Upsilon(X) \ar[d]_{ R_{\widetilde{C}_i} (\Upsilon(X)  )\ms{5mu}} ^-\bwr \\
\Upsilon(q^{\phi_i(\la)}\tens X \otimes  C_i )  \ar[r]^\sim & q^{\phi_i(\la)}\tens  \Upsilon(X)\otimes  \widetilde{C}_i
}
$$
commutes.
\setcounter{myc}{\value{enumi}}

\end{enumerate}

Moreover, the functor $\Upsilon$ satisfies the following universal property:
\bnum\setcounter{enumi}{\value{myc}}
\item  If there are another $\La$-graded monoidal category $\catT'$ 
with an invertible central object $q\in\catT'_0$
with and a $\La$-graded  monoidal functor $\Upsilon'\cl  \catT \rightarrow \catT'$ 
such that  
\bna
\item  $\Upsilon'$ sends the central object $q\in\catT_0$ to $q\in\catT'_0$, 
\item   \label{Eq: loc 1}
$\Upsilon'(C_i) $ is invertible in $\catT'$ for any $i\in I$ and
\item 
for any $i\in I$ and $X\in\catT$, $\Upsilon'(R_{C_i}(X))\cl
\Upsilon'(C_i\tens X)\to\Upsilon'(q^{\phi_i(\la)}\tens X\tens C_i)$ is an isomorphism,
\end{enumerate}
then there exists a monoidal functor $\mathcal F$, which is unique up to a unique isomorphism,  such that
the diagram
$$
\xymatrix{
\catT \ar[r]^{\Upsilon} \ar[dr]_{\Upsilon'}  & \lT \ar@{.>}[d]^{\mathcal F }\\
& \catT'
}
$$
commutes.
\end{enumerate}

\end{thm}

We denote by $\catT [ C_i^{\otimes -1} \mid i\in I]$ the localization $\lT$ in Theorem \ref{Thm: graded localization}.
Note that
$$
(X, \alpha+\beta) \simeq  q^{-\gH(\beta, \alpha)} \otimes (C^\alpha \otimes X, \beta), \quad
(\triv, \beta) \otimes  (\triv, -\beta) \simeq q^{-\gH(\beta, \beta)} (\triv, 0)
$$
for $\alpha \in \lG_{\ge 0}$ and  $\beta \in \lG$.

\begin{prop}
Let $( C_i ,  R_{C_i}  ,   \dphi_i  )_{i\in I}$ be a real commuting family of graded braiders in a graded monoidal category $\catT$,
and  set $\lT \seteq   \catT  [ C_i^{\otimes -1} \mid i \in I ]$.
Assume that
\bna
\item $\catT$ is an abelian category,
\item $\otimes$ is exact.
\end{enumerate}
Then $\lT$ is an abelian category with exact $\tens$,
and the functor $\Upsilon\cl  \catT \rightarrow \lT $ is exact.
\end{prop}

\subsection{Quiver Hecke algebras} \
\subsubsection{Cartan data} \

Let $I$ be an index set.
A {\it Cartan datum} $ \bl\cmA,\wlP,\Pi,\Pi^\vee,(\cdot,\cdot) \br $
consists of
\begin{enumerate}[{\rm (i)}]
\item a free abelian group $\wlP$, called the {\em weight lattice},
\item $\Pi = \{ \alpha_i \mid i\in I \} \subset \wlP$,
called the set of {\em simple roots},
\item $\Pi^{\vee} = \{ h_i \mid i\in I \} \subset \wlP^{\vee}\seteq
\Hom( \wlP, \Z )$, called the set of {\em simple coroots},
\item a $\Q$-valued symmetric bilinear form $(\cdot,\cdot)$ on $\wlP$,
\end{enumerate}
which satisfy
\begin{enumerate} [{\rm (a)}]
\item  $(\alpha_i,\alpha_i)\in 2\Z_{>0}$ for $i\in I$,
\item $\langle h_i, \lambda \rangle =\dfrac{2(\alpha_i,\lambda)}{(\alpha_i,\alpha_i)}$ for $i\in I$ and $\lambda \in \Po$,
\item $\cmA \seteq (\langle h_i,\alpha_j\rangle)_{i,j\in I}$ is
a {\em generalized Cartan matrix}, i.e.,
$\langle h_i,\alpha_i\rangle=2$ for any $i\in I$ and
$\langle h_i,\alpha_j\rangle \in\Z_{\le0}$ if $i\not=j$,
\item $\Pi$ is a linearly independent set,
\item for each $i\in I$, there exists $\Lambda_i \in \wlP$
such that $\langle h_j, \Lambda_i \rangle = \delta_{ij}$ for any $j\in I$.
\end{enumerate}
Let $\Delta$ (resp.\ $\prD$, $\nrD$) be the set of roots  (resp.\ positive roots, negative roots).
We set $\wlP_+ \seteq  \{ \lambda \in \wlP \mid  \langle h_i, \lambda\rangle \ge 0  \text{ for }i\in I\}$,
$ \rlQ = \bigoplus_{i \in I} \Z \alpha_i$, and $ \rlQ_+ = \sum_{i\in I} \Z_{\ge 0} \alpha_i$,
and write $\Ht (\beta)=\sum_{i \in I} k_i$  for $\beta=\sum_{i \in I} k_i \alpha_i \in \rlQ_+$.
For $i\in I$, we define
$$s_i(\lambda)=\lambda-\langle h_i, \lambda\rangle\alpha_i \qquad
\text{for $\lambda\in \wlP$},$$
and $\weyl$ is the subgroup of $\mathrm{Aut}(\wlP)$ generated by
$\{s_i\}_{i\in I}$.

For $w,v \in \weyl$, we write $w \bge v$
if there exists a reduced expression of $v$ which
appears in a subexpression of a reduced expression of $w$ (the Bruhat order on $\weyl$).

For $w \in \weyl$, we  say that an element $\la \in \wlP$ is \emph{$w$-dominant} if 
\eq
\text{$(\beta,\la)\ge0$ for any $\beta\in\prD\cap w^{-1}\nrD$.}
\label{cond:dominant}
\eneq
This condition is equivalent to
$$\text{$\lan h_{i_k}, s_{i_{k+1}} \cdots s_{i_r} \la \ran \ge 0 $
for $1 \le  k \le r$,}$$
where $w=s_{i_1}\cdots s_{i_r}$ is a reduced expression of $w$.

Note that any $\la \in \wlP_+$ is $w$-dominant for any $w \in \weyl$.

\subsubsection{Quiver Hecke algebras}\

Let $\bR$ be a field.
 For $i,j\in I$, we choose polynomials
$\qQ_{i,j}(u,v) \in \bR[u,v]$ such that
\bna

\item $\qQ_{i,j}(u,v) = \qQ_{j,i}(v,u)$,

\item it is of the form
\begin{align*}
\qQ_{i,j}(u,v) =\bc
                   \sum\limits
_{p(\alpha_i , \alpha_i) + q(\alpha_j , \alpha_j) = -2(\alpha_i , \alpha_j) } t_{i,j;p,q} u^pv^q &
\text{if $i \ne j$,}\\[3ex]
0 & \text{if $i=j$,}
\ec
\end{align*}
where  $t_{i,j;-a_{ij},0} \in  \bR^{\times}$.
\end{enumerate}
For $\beta\in \rlQ_+$ with $ \Ht(\beta)=n$, we set
$$
I^\beta\seteq  \Bigl\{\nu=(\nu_1, \ldots, \nu_n ) \in I^n \bigm| \sum_{k=1}^n\alpha_{\nu_k} = \beta \Bigr\},
$$
on which the symmetric group $\mathfrak{S}_n = \langle s_k \mid k=1, \ldots, n-1 \rangle$ acts  by place permutations.

\begin{df}
\ For $\beta\in\rlQ_+$,
the {\em quiver Hecke algebra} $R(\beta)$ associated with $\cmA$ and $(\qQ_{i,j}(u,v))_{i,j\in I}$
is the $\bR$-algebra generated by
$$
\{e(\nu) \mid \nu \in I^\beta \}, \; \{x_k \mid 1 \le k \le n \},
 \; \{\tau_l \mid 1 \le l \le n-1 \}
$$
satisfying the following defining relations:
\begin{align*}
& e(\nu) e(\nu') = \delta_{\nu,\nu'} e(\nu),\ \sum_{\nu \in I^{\beta}} e(\nu)=1,\
x_k e(\nu) =  e(\nu) x_k, \  x_k x_l = x_l x_k,\\
& \tau_l e(\nu) = e(s_l(\nu)) \tau_l,\  \tau_k \tau_l = \tau_l \tau_k \text{ if } |k - l| > 1, \\[5pt]
&  \tau_k^2 e(\nu) = \qQ_{\nu_k, \nu_{k+1}}(x_k, x_{k+1}) e(\nu), \\[5pt]
&  (\tau_k x_l - x_{s_k(l)} \tau_k ) e(\nu) = \left\{
                                                           \begin{array}{ll}
                                                             -  e(\nu) & \hbox{if } l=k \text{ and } \nu_k = \nu_{k+1}, \\
                                                               e(\nu) & \hbox{if } l = k+1 \text{ and } \nu_k = \nu_{k+1},  \\
                                                             0 & \hbox{otherwise,}
                                                           \end{array}
                                                         \right. \\[5pt]
&( \tau_{k+1} \tau_{k} \tau_{k+1} - \tau_{k} \tau_{k+1} \tau_{k} )  e(\nu) \\[4pt]
&\qquad \qquad \qquad = \left\{
                                                                                   \begin{array}{ll}
\bQ_{\,\nu_k,\nu_{k+1}}(x_k,x_{k+1},x_{k+2}) e(\nu) & \hbox{if } \nu_k = \nu_{k+2}, \\
0 & \hbox{otherwise}, \end{array}
\right.\\[5pt]
\end{align*}
\end{df}
where
\begin{align*}
\bQ_{i,j}(u,v,w)\seteq\dfrac{ \qQ_{i,j}(u,v)- \qQ_{i,j}(w,v)}{u-w}\in \bR[u,v,w].
\end{align*}
The algebra $R(\beta)$ has the $\Z$-grading defined by
\begin{align*}
\deg(e(\nu))=0, \quad \deg(x_k e(\nu))= ( \alpha_{\nu_k} ,\alpha_{\nu_k}), \quad  \deg(\tau_l e(\nu))= -(\alpha_{\nu_{l}} , \alpha_{\nu_{l+1}}).
\end{align*}

We denote by $R(\beta) \Mod$  the category of graded $R(\beta)$-modules with degree preserving homomorphisms.
We  write $R(\beta)\gmod$ for the full subcategory of $R(\beta)\Mod$ consisting of the graded modules which are  finite-dimensional over $\bR $, and 
 $R(\beta)\proj$ for  the full subcategory of $R(\beta)\Mod$ consisting of finitely generated  projective graded $R(\beta)$-modules.
We set $R\Mod \seteq \bigoplus_{\beta \in \rlQ_+} R(\beta)\Mod$, $R\proj \seteq \bigoplus_{\beta \in \rlQ_+} R(\beta)\proj$, and  $R\gmod \seteq \bigoplus_{\beta \in \rlQ_+} R(\beta)\gmod$.
The trivial $R(0)$-module of degree 0 is denoted by $\trivialM$.
For simplicity, we write ``a module" instead of ``a graded module''.
We define the grading shift functor $q$
by $(qM)_k = M_{k-1}$ for a graded module $M = \bigoplus_{k \in \Z} M_k $.
For $M, N \in R(\beta)\Mod $, $\Hom_{R(\beta)}(M,N)$ denotes the space of degree preserving module homomorphisms.
We define
\[
\HOM_{R(\beta)}( M,N ) \seteq \bigoplus_{k \in \Z} \Hom_{R(\beta)}(q^{k}M, N),
\]
and set $ \deg(f) \seteq k$ for $f \in \Hom_{R(\beta)}(q^{k}M, N)$.
When $M=N$, we write $\END_{R(\beta)}( M ) = \HOM_{R(\beta)}( M,M)$.
We sometimes write $R$ for $R(\beta)$ in $\HOM_{R(\beta)}( M,N )$ for simplicity.

For $M \in R(\beta)\gmod$, we set $M^\star \seteq  \HOM_{\bR}(M, \bR)$ with the $R(\beta)$-action given by
$$
(r \cdot f) (u) \seteq  f(\rho(r)u), \quad \text{for  $f\in M^\star$, $r \in R(\beta)$ and $u\in M$,}
$$
where $\rho$ is the antiautomorphism of $R(\beta)$ which fixes the generators
 $e(\nu)$, $x_k$, $\tau_k$. 
We say that $M$ is \emph{self-dual} if $M \simeq M^\star$ in $R\gmod$.

 For $\beta,\beta'\in\prtl$, set
$
e(\beta, \beta') \seteq \sum_{\nu \in I^\beta, \nu' \in I^{\beta'}} e(\nu\ast\nu'),
$
where $\nu\ast\nu'$ is the concatenation of $\nu$ and $\nu'$. 
Then there is an injective ring homomorphism
$$R(\beta)\tens R(\beta')\to e(\beta,\beta')R(\beta+\beta')e(\beta,\beta')$$
given by
$e(\nu)\tens e(\nu')\mapsto e(\nu,\nu')$,
$x_ke(\beta)\tens 1\mapsto x_ke(\beta,\beta')$,
$1\tens x_ke(\beta')\mapsto x_{k+\height{\beta}}e(\beta,\beta')$,
$\tau_ke(\beta)\tens 1\mapsto \tau_ke(\beta,\beta')$ and
$1\tens \tau_ke(\beta')\mapsto \tau_{k+\height{\beta}}e(\beta,\beta')$.
For $a\in R(\beta)$ and $a'\in R(\beta')$, the image of $a\tens a'$ is sometimes denoted by $a\etens a'$.

For $M \in R(\beta)\Mod$ and $N \in R(\beta')\Mod$, we set
$$
M \conv N \seteq R(\beta+\beta') e(\beta, \beta') \otimes_{R(\beta) \otimes R(\beta')} (M \otimes N).
$$
For $u\in M$ and $v\in N$, the image of
$u\tens v$ by the map $M\tens N\to M\conv N$ is sometimes denoted by
$u\etens v$. 
We also write $M\etens N\subset M\conv N$ for the image of $M\tens N$ in $M\conv N$.

For $\al,\beta\in\prtl$, let $X$ be an $R(\al+\beta)$-module.
Then $e(\al,\beta)X$ is an $R(\al)\tens R(\beta)$-module.
We denote it
by $$\Res_{\al,\beta}X.$$
We have
\eq \label{eq:adjoints}
&&\ba{rl}
\Hom_{R(\al)\tens R(\beta)}\bl M\tens N,\Res_{\al,\beta}(X)\br
&\simeq\Hom_{R(\al+\beta)}(M\conv N,X),\\
\Hom_{R(\al)\tens R(\beta)}\bl\Res_{\al,\beta}(X), M\tens N\br
&\simeq\Hom_{R(\al+\beta)}(X, q^{(\al,\beta)}N\conv M)\ea
\eneq
for any $R(\al)$-module $M$, any $R(\beta)$-module $N$ and
any $R(\al+\beta)$-module $X$.

\medskip
We denote by $M \hconv N$ the head of $M \conv N$ and by $M \sconv N$ the socle of $M \conv N$.
We say that simple $R$-modules $M$ and $N$ \emph{strongly commute} if $M \conv N$ is simple.  A simple $R$-module
$L$ is \emph{real} if $L$ strongly commutes with itself. 
Note that if $M$ and $N$ strongly commute,  then $M$ and $N$ commute, i.e., $M\conv N \simeq N \conv M$ up to a grading shift.

For $i\in I$ and the functors $E_i$ and $F_i$ are defined by
\eqn&&
\ba{rl}E_i(M)& = e(\alpha_i, \beta-\al_i) M \in  R(\beta-\al_i) \Mod\\
 F_i(M) &= R(\alpha_i) \conv M \in  R(\beta+\al_i) \Mod 
\ea\hs{5ex}\qt{for an $R(\beta)$-module $M$.}
\eneqn
For $i\in I $ and $n\in \Z_{>0}$, let $L(i)$ be the simple $R(\alpha_i)$-module concentrated on  degree 0 and
 $P(i^{n})$  the indecomposable  projective $R(n \alpha_i)$-module
whose head is isomorphic to $L(i^n) \seteq q_i^{\frac{n(n-1)}{2}} L(i)^{\conv n}$,
where $q_i\seteq q^{(\al_i,\al_i)/2}$. Then, for $M\in  R(\beta) \Mod $, we define
\eqn&&
\ba{ll}
E_i^{(n)} M& \seteq \HOM_{R(n\alpha_i)} \bl P(i^{n}),\, e(n\alpha_i, \beta - n\alpha_i) M\br \in  R(\beta-n\al_i) \Mod,\\
F_i^{(n)} M& \seteq  P(i^{n}) \conv M \in  R(\beta+n\al_i) \Mod. 
\ea
\eneqn
For $i \in I$ and a non-zero  $M \in R(\beta)\Mod$, we define
\begin{align*}
\wt(M) = - \beta, \ \   \ep_i(M) = \max \{ k \ge 0 \mid E_i^k M \not\simeq 0 \}, \ \  \ph_i(M) = \ep_i(M) + \langle h_i, \wt(M) \rangle.
\end{align*}
For a simple module $M$, we set
$$E^{\max}_i (M)\seteq E_i^{(\eps_i(M))}M.$$
We can also define $E_i^*$, $F_i^*$, $\ep^*_i$, etc.\ in the same manner as above if we replace  $e(\alpha_i, \beta-\al_i)$, $R(\alpha_i)\conv -$,
etc.\  with 
$e(\beta-\al_i, \alpha_i)$, $- \conv R(\alpha_i)$, etc.

We denote by $K(R\proj)$ and $K(R\gmod)$ the Grothendieck groups of $R\proj$ and $R\gmod$, respectively.

The following proposition will be  used frequently.

\begin{prop}[Shuffle lemma,  { \cite[Proposition 2.7]{McNamara15},  \cite[Theorem 4.3]{KM17}}]
\label{prop: shuffle lemma}
Let $\{\beta_j\}_{1\leqslant j \leqslant r}$ and 
 $\{\gamma_k\}_{1\le k \leqslant s}$ be two families of elements in $\rl_+$ such that 
$\sum_{j=1}^r\beta_j=\sum_{k=1}^s \gamma_k$.
Let $M_j$ be an $R(\beta_j)$-module for each $1 \leqslant j \leqslant r$.
Then $e(\gamma_1,\ldots,\gamma_s)(M_1 \conv \cdots \conv M_r)$ has a filtration of $ \bigotimes_{k=1}^s R(\gamma_k)$-modules 
whose graduations are isomorphic  to the modules of the form 
$$q^N\left(   \bigotimes_{k=1}^s R(\gamma_k)
 e(\beta_{1,k},\ldots,\beta_{r,k}) \right)  \bigotimes_{\bigotimes_{j,k} R(\beta_{j,k})}
\left( \bigotimes_{j=1}^r e(\beta_{j,1}, \ldots \beta_{j,s }) M_j \right).
$$
Here \be[$\bullet$]
\item $\{\beta_{j,k} \}_{1\leqslant j\le r, \ 1\leqslant k\le s}$ is a family of elements in $\rl_+$ such that $\beta_j=\sum_{k=1}^s \beta_{j,k}$
and $\gamma_k=\sum_{j=1}^r \beta_{j,k}$,
\item
 the right action of $\bigotimes_{j,k}R(\beta_{j,k})$ on $\bigotimes_{k=1}^s R(\gamma_k) e(\beta_{1,k},\ldots,\beta_{r,k}) $ is induced by the action of $R(\beta_{j,k})$ on $R(\gamma_k) e(\beta_{1,k},\ldots,\beta_{r,k})$,
\item the left action of $\bigotimes_{j,k}R(\beta_{j,k})$ on $ \bigotimes_{j=1}^re(\beta_{j,1}, \ldots \beta_{j,s}) M_j $ is induced by the left action of $R(\beta_{j,k})$ on $e(\beta_{j,1}, \ldots \beta_{j,s}) M_j $,
\item $N=-\hs{-2ex}\displaystyle\sum\limits_{1\le j<j'\le r,\;1\le k'<k\le s}(\beta_{j,k},\beta_{j',k'})$. 
\ee
\end{prop}

Let us denote by $\g$ the symmetrizable Kac-Moody algebra associated with the Cartan datum $ (\cmA,\wlP,\Pi,\Pi^\vee,(\cdot,\cdot) ) $.
Then the quiver Heck algebra $R$ associated with the same Cartan datum categorifies the negative half $\Uqm$ of the quantum group $\Uq$ and its crystal/global basis (\cite{KL09,R08,R11,VV09,LV11}).
Let $B(\infty)$ be the crystal basis of $\Uqm$. Then there is a bijection between $B(\infty)$ and the set of the isomorphism classes of self-dual simple $R$-modules.
For $b \in B(\infty)$, let $S(b)$ be the self-dual simple $R$-module corresponding to $b$. Then we have
\eqn
S(\tf_i(b))\simeq L(i)\hconv S(b)\quad \text{and} \quad
S(\te_i(b))\simeq \hd\bl E_i (S(b))\br
\eneqn
for $i \in I$  up to grading shifts, 
where $\tf_i$ and $\te_i$ denote the Kashiwara operators on $B(\infty)$.

Let $\psi\cl R(\beta)\isoto R(\beta)$ be the ring automorphism
\eq \label{eq:psi}
\psi&\;:\hs{2ex}&\ba{l}
e(\nu_1,\ldots ,\nu_n) \mapsto e(\nu_n,\ldots ,\nu_1), \\
 x_k \mapsto x_{n+1-k} \qt{($1\le k\le n$),}\\
\tau_k\mapsto -  \tau_{n-k}\qt{($1\le k<n$),}
\ea\label{def:antipsi}
\eneq
 where $n=\height{\beta}$.
It induces a monoidal functor
$$\psi_*\cl R\gmod\isoto(R\gmod)^\rev.$$
Here, for a monoidal category $\sht$, $\sht^\rev$ is the new monoidal category with the reversed tensor product $\otimes^\rev$ (see \S\,\ref{subsub:mon}).
Hence, there is a functorial isomorphism
$$\psi_*(M\conv N) \simeq \psi_*(N) \conv \psi_*(M)$$
for graded $R$-modules $M$ and $N$.
Since $\psi$ is involutive, so is $\psi_*$.
Note that $\psi_*(L(i^n)) \simeq L(i^n)$ for $i \in I$ and $n \ge 0$.

\subsection{Affinizations and R-matrices}\

Let $R$ be a quiver Hecke algebra. 
We recall the notions of affinizations and R-matrices introduced  in \cite{KP18}.
For $\beta \in \rlQ_+$ and $i\in I$, let
\begin{align} \label{Eq: def of p}
\mathfrak{p}_{i, \beta}  \seteq \sum_{\nu \in I^\beta} \Bigl(\hs{1ex}  \prod_{a \in \{1, \ldots, \Ht(\beta) \},\ \nu_a=i} x_a \Bigr) e(\nu)\in R(\beta).
\end{align}
Then $\mathfrak{p}_{i, \beta} $ belongs to the center of $R(\beta)$.
\begin{df} \label{Def: aff}
Let $M$ be a simple $R(\beta)$-module. An \emph{affinization} of $M$ 
 with degree $d_{\Ma}$ is an $R(\beta)$-module $\Ma$ with an endomorphism $z_{\Ma}$ of $\Ma$
with degree $d_{\Ma} \in \Z_{>0}$
and an isomorphism $\Ma / z_{\Ma} \Ma \simeq M$ such that
\begin{enumerate}[\rm (i)]
\item $\Ma$ is a finitely generated free module over the polynomial ring $\bR[z_{\Ma}]$,
\item $\mathfrak{p}_{i,\beta} \Ma \ne 0$ for all $i\in I$.
\end{enumerate}
\end{df}

Let $\beta \in \rlQ_+$ with $m =  \Ht(\beta)$. For  $k=1, \ldots, m-1$ and $\nu \in I^\beta$, the \emph{intertwiner} $\varphi_k  $ is defined by 
$$
\varphi_k e(\nu) =
\bc
 (\tau_k x_k - x_k \tau_k) e(\nu) 
= (x_{k+1}\tau_k - \tau_kx_{k+1}) e(\nu) \\
\hs{5ex} =\bl\tau_k(x_k-x_{k+1})+1\br e(\nu) 
=\bl(x_{k+1}-x_{k})\tau_k-1\br e(\nu) 
& \text{ if } \nu_k = \nu_{k+1}, 
 \\
 \tau_k e(\nu) & \text{ otherwise.}
\ec
$$ 

\begin{lem} [{\cite[Lemma 1.5]{KKK18}}] \label{Lem: intertwiners} \
\begin{enumerate}[\rm (i)]
\item $\varphi_k^2 e(\nu) = \bl Q_{\nu_k, \nu_{k+1}} (x_k, x_{k+1} )+ \delta_{\nu_k, \nu_{k+1}} \br\, e(\nu)$.
\item $\{  \varphi_k \}_{k=1, \ldots, m-1}$ satisfies the braid relation.
\item For a reduced expression $w = s_{i_1} \cdots s_{i_t} \in \sg_m$, we set $\varphi_w \seteq  \varphi_{i_1} \cdots \varphi_{i_t} $. Then
$\varphi_w$ does not depend on the choice of reduced expression of $w$.
\item For $w \in \sg_m$ and $1 \le k \le m$, we have $\varphi_w x_k = x_{w(k)} \varphi_w$.
\item For $w \in \sg_m$ and $1 \le k < m$, if $w(k+1)=w(k)+1$, then $\varphi_w \tau_k = \tau_{w(k)} \varphi_w$.
\end{enumerate}
\end{lem}

For $m,n \in \Z_{\ge 0}$, we set $w[m,n]$ to be the element of $\sg_{m+n}$ such that
$$
w[m,n](k) \seteq
\left\{
\begin{array}{ll}
 k+n & \text{ if } 1 \le k \le m,  \\
 k-m & \text{ if } m < k \le m+n.
\end{array}
\right.
$$

Let $\beta, \gamma \in \rlQ_+$ and set $m\seteq  \Ht(\beta)$ and $n\seteq \Ht(\gamma)$.
For $M \in R(\beta)\Mod$ and $N \in R(\gamma)\Mod$, the $R(\beta)\otimes R(\gamma)$-linear map $M \otimes N \rightarrow N \conv M$ defined by $$u \otimes v \mapsto \varphi_{w[n,m]}(v \etens u)$$
can be extended to an $R(\beta+\gamma)$-module homomorphism (up to a grading shift)
$$
\RR_{M,N}\cl  M\conv N \longrightarrow N \conv M.
$$

Let $\Ma$ be an affinization of a simple $R$-module $M$, and  let $N$ be a non-zero $R$-module. We define a homomorphism (up to a grading shift)
$$
\nR_{\Ma, N} \seteq  \z^{-s} \RR_{\Ma, N}\cl  \Ma \conv N \longrightarrow N \conv \Ma,
$$
where $s$ is the largest integer such that $ \RR_{\Ma, N}(\Ma \conv N) \subset \z^s (N \conv \Ma)$.
We define
$$
\Rr{M,N} \cl M \conv N \longrightarrow N \conv M
$$
to be the homomorphism (up to a grading shift) induced from $\nR_{\Ma, N}$ by specializing at $\z=0$. By the definition, $\Rr{M,N}$ never vanishes.
We now define
\begin{align*}
\La(M,N) &\seteq  \deg (\Rr{M,N}) , \\
\tLa(M,N) &\seteq   \frac{1}{2} \bl \La(M,N) + (\wt(M), \wt(N)) \br , \\
\Dd(M,N) &\seteq  \frac{1}{2} \bl\La(M,N) + \La(N,M)\br.
\end{align*}

Then we have (\cite[Lemma 3.11]{KKOP21})
\eqn
\text{$\de(M,N)$ and $\tLa(M,N)$ are non-negative integers.}
\eneqn
We also have (cf.\ \cite[Lemma 3.15]{KKKO18})
\eq \label{eq:Lalinear}
\La(M,N_1\conv N_2) = \La(M, N_1)+\La(M,N_2)
\eneq 
for  non-zero modules $N_1,N_2$ and a simple module $M$ which admits an affinization.

\Def \label{def:affreal}
We say that an $R$-module module $M$ is {\em \afr} if $M$ is real simple and $M$ admits an affinization.
\edf

\Prop[{\cite[Proposition~2.10]{KP18}}]\label{prop:ru}
Let $M$ and $N$ be simple $R$-modules. Assume that one of them is \afr.
Then we have
$$\HOM(M\conv N,N\conv M)=\cor \Rr{M,N}.$$

\enprop

In the case that $N$ has an affinization $(\Na,z_\Na)$,
we can define $\nR_{M,\Na}$ in a similar way as above.
Then we have
$\deg \nR_{\Ma,N}=\deg \nR_{M,\Na}$, and
$\nR_{\Ma,N}\vert_{z_{\Ma}=0}=\nR_{M,\Na}\vert_{z_{\Na}=0}$ up to a constant multiple
if $M$ or $N$ is real.
Hence $\Lambda(M,N)$ and $\rmat{M,N}$ (up to a constant multiple) are
well defined when either $M$ or $N$ is \afr.
Moreover, they do not depend on the choice of affinizations.

\subsection{Determinantial modules} \

Let $\Lambda \in \wlP_+$ and let $t_i$ be an indeterminate  for each $i\in I$.
Set
$$a_{\La, i}(t_i) \seteq t_i^{\langle h_i, \Lambda \rangle} \in \bR[t_i] \qquad \text{for } \ i\in I.$$

Let $\lambda\in\Lambda-\rlQ_+$, and write 
$ \beta \seteq  \Lambda - \lambda\in\rlQ_+$ and $n\seteq \Ht(\beta)$.
The \emph{cyclotomic quiver Hecke algebra} is the quotient of $R(\beta)$ given by
\eq
R^\Lambda(\lambda)\seteq\dfrac{R(\beta)}{ \sum_{i\in I} R(\beta)a_{\Lambda, i}(x_n e(\beta-\alpha_i, \alpha_i))R(\beta)}.
\label{eq:cyclo}
\eneq

See \cite{KKOP21} for more details.

Let $R^{\Lambda}(\lambda) \Mod$ be the category of graded $R^{\Lambda}(\lambda)$-modules.

We define the functors
\begin{align*}
F_i^{\Lambda} &\cl  R^{\Lambda}(\lambda)\Mod \rightarrow  R^{\Lambda}(\lambda-\alpha_i)\Mod , \\
E_i^{\Lambda} &\cl  R^{\Lambda}(\lambda)\Mod \rightarrow R^{\Lambda}(\lambda+\alpha_i)\Mod
\end{align*}
by
$
F_i^{\Lambda}M = R^{\Lambda}(\lambda-\alpha_i)e(\alpha_i,\beta)\otimes_{R^{\Lambda}(\lambda)}M $ and $
E_i^{\Lambda}M = e(\alpha_i,\beta-\alpha_i)M
$
for $M\in R^{\Lambda} (\lambda)\Mod$.
Similarly, for $m\in\Z_{\ge0}$, we define
\eqn
 F^{\Lambda}_{i}{}^{(m)}M&&=R^\La(\la-m\al_i)\tens_{R(\La-\la+m\al_i)}
F_i^{(m)}M\in R^{\Lambda}(\lambda-m\alpha_i)\Mod,\\
E^{\Lambda}_{i}{}^{(m)}M&&=E_i^{(m)}M\in R^{\Lambda}(\lambda+m\alpha_i)\Mod.
\eneqn

For $\lambda,\mu\in\wtl$, we write
$\lambda \wle \mu$ 
if there exists a sequence of
real positive roots $\beta_k$ ($1\le k\le \ell$)
such that $\la=s_{\beta_\ell}\cdots s_{\beta_1}\mu$ and
$(\beta_k,s_{\beta_{k-1}}\cdots s_{\beta_1}\mu)\ge0$ for $1\le k\le\ell$.
Here $s_{\beta}(\la)=\la-(\beta^\vee,\la)\beta$ with
$\beta^\vee=\frac{2}{(\beta,\beta)}\beta$.

Let $\lambda, \mu \in \weyl \Lambda$ such that $\lambda \wle \mu$.
The module $\dM(\lambda, \mu)$, called  the \emph{determinantial module},
is defined as follows.
Choose $w$, $v\in \weyl$ such that $ \lambda  = w\Lambda$ and $\mu  = v\Lambda$,
and then take their reduced expressions $\underline{w} = s_{i_1} \cdots s_{i_l}$ and  $\underline{v} = s_{j_1} \cdots s_{j_t}$, and
set $m_k = \langle h_{i_k},   s_{i_{k+1}} \cdots s_{i_l}\Lambda \rangle$ for $k=1,\ldots l$, and $n_k = \langle h_{j_k},   s_{j_{k+1}} \cdots s_{j_t}\Lambda \rangle$ for $k=1, \ldots, t$.
We define 
\begin{align*}
\dM(\lambda, \Lambda) & \seteq F^{\Lambda}_{i_1}{}^{ ( m_1) }  \cdots F^{\Lambda}_{i_l}{}^{  \,( m_l) } \bR, \\
\dM(\lambda, \mu) & \seteq E_{j_1}^{\,*\hskip 0.1em \,( n_1) }  \cdots E_{j_t}^{\,* \hskip 0.1em  \,( n_t) } \dM(\lambda, \La).
\end{align*}
The determinantial module $\dM(\lambda, \mu)$ does not depend on the choice of $w$, $v$ and their reduced expressions. 

We summarize properties of determinantial modules.
\begin{prop} [{\cite{KK11}, \cite[Lemma 1.7, Proposition 4.2]{KKOP18} and  \cite[Theorem 3.26]{KKOP21}}] \label{Prop: dM properties}
Let $\Lambda \in \wlP_+$, and $\lambda, \mu \in \weyl \Lambda$ with $\lambda \wle \mu$.
\begin{enumerate} [\rm (i)]
\item  $\dM(\lambda, \mu)$ is a real simple $R^{\La}$-module with an affinization.
\item If $ \langle h_i, \lambda \rangle \le 0$ and $s_i \lambda \preceq \mu $, then
$$
\ep_i( \dM(\lambda, \mu)) = - \langle h_i, \lambda \rangle \quad \text{and} \quad  E_i^{(- \langle h_i, \lambda \rangle)} \dM(\lambda, \mu) \simeq \dM(s_i\lambda, \mu).
$$
\item If $ \langle h_i, \mu \rangle \ge 0$ and $ \lambda \preceq s_i \mu $, then
$$
\ep^*_i( \dM(\lambda, \mu)) =  \langle h_i, \mu \rangle \quad \text{and} \quad  E_i^{* (\langle h_i, \mu \rangle)} \dM(\lambda, \mu) \simeq \dM(\lambda, s_i\mu).
$$
\item For $\La,\La'\in\pwtl$ and $w,v \in\weyl$ such that $v\ble w$, we have
$$\Mm(w\La,v\La)\conv\Mm(w\La',v\La')\simeq q^{-( v\La,  v\La'-w\La')}
 \Mm\bl w(\La+\La'),v(\La+\La')\br.$$ 
\item For $\lambda, \mu,\zeta \in \weyl \Lambda$ with $\lambda \wle \mu\wle\zeta$, we have
$\dM(\la,\zeta)\simeq\dM(\la,\mu)\hconv\dM(\mu,\zeta)$.
\end{enumerate}
\end{prop}

\subsection{The categories $\catC_w$ and  $\catC_{*,v}$} \label{Sec: catC} 

In this subsection, we recall  the definition of categories $\catC_w$, $\catC_{*,v}$ and $\catC_{w,v}$ appeared in  \cite{KKOP18} (see also \cite{TW16}).

For $M \in  R(\beta)\Mod$, we define
\begin{align*}
\gW(M) \seteq  \{  \gamma \in  \rlQ_+ \cap (\beta - \rlQ_+)  \mid  e(\gamma, \beta-\gamma) M \ne 0  \}, \\
\gW^*(M) \seteq  \{  \gamma \in  \rlQ_+ \cap (\beta - \rlQ_+)  \mid  e(\beta-\gamma, \gamma) M \ne 0  \}.
\end{align*}
Hence if $M=0$, then $\gW(M)=\emptyset$, and
if $M\not=0$, then $0,\beta\in\gW(M)$.

\Prop[{\cite[Proposition 3.7]{TW16}}]\label{prop:W}
For any $R$-module $M$,
we have
$$\gW(M)\subset\Sp\bl \gW(M)\cap\prD\br.$$
Here, for a subset $A\subset\R\tens_\Z\rlQ$,
we denote by $\Sp(A)$ the smallest subset of
$\R\tens_\Z\rlQ$ that is stable by the multiplication of $\R_{\ge0}$ and contains $A\cup\st{0}$.
\enprop
For $w\in \weyl$, we denote by $\catC_{w}$ the  full subcategory of $R\gmod$ 
consisting of  objects $M$ such that 
\begin{align*}
\gW(M) \subset \Sp( \prD \cap w \nrD ).
\end{align*}
By Proposition~\ref{prop:W}, this condition is equivalent to
$$\gW(M) \cap\prD\subset w\nrD.$$

Similarly,  for $v\in \weyl$, we denote by $\catC_{*,v}$ the  full subcategory of $R\gmod$  consisting of  objects $M$ such that
\begin{align*}
\gW^*(M) \subset \Sp( \prD \cap v \prD ).
\end{align*}
For $w,v \in \weyl$, we define $\catC_{w,v}$ to be the full subcategory of $R\gmod$ whose objects are contained in both of the subcategories $\catC_{w}$ and $\catC_{*,v}$.

The categories $\catC_w$, $\catC_{*,v}$, and $\catC_{w,v}$ are stable under  convolution products, grading shifts, extensions, taking subquotients.

\subsection{Graded braiders in $R\gmod$ associated with a Weyl group element} \

\begin{df}
A graded braider $(M,R_M, \phi)$  in $R\gmod$  is called \emph{non-degenerate} if 
$$R_M(L(i)) \cl  M \conv L(i) \rightarrow q^{\phi(\al_i)}L(i)\conv M$$ is a non-zero homomorphism for each $i\in I$.
\end{df}
If $(M,R_M, \phi)$ is non-degenerate, then $\phi(\al_i)=-\La(M,L(i))$ and $R_M(L(i))=c\,\rmat{M,L(i)}$ for some $c\in \corp^\times$.

\begin{thm}[{\cite[Proposition 4.1, Lemma 4.3]{KKOP21}}] \label{lem: c braider}
Let $R$ be a quiver Hecke algebra 
and let $M$ be a simple $R$-module. Then there exists a non-degenerate graded braider $(M,R_M,\phi)$ in $R \gmod$.
If $(M,R'_M,\phi')$ is another  non-degenerate graded braider,
then $\phi=\phi'$ and there exists a group homomorphism
$c\cl  \rlQ \rightarrow \bR^\times$ such that
$ R'_M(X) = c(\beta) R_M(X)$ for any $X \in R(\beta)\gmod$.
\end{thm}

Fix an element $w$ in the Weyl group $\weyl$.
For $\La \in \wlP_+$, set
\begin{align*}
\dC_\La=\dC_{w,\, \La} \seteq  \dM(w \Lambda, \Lambda ) \qquad \text{and} \quad \dC_i \seteq  \dC_{\La_i} \quad (i \in I).
\end{align*}

For $i\in I$, we set
$$
\la_i \seteq
\left\{
\begin{array}{ll}
w \La_i + \La_i& \text{ if } w \La_i \ne \La_i,  \\
0& \text{ if } w \La_i  =  \La_i.
\end{array}
\right.
$$
Note that $\dC_i \simeq \triv$ if and only if $ w\La_i = \La_i$. We have
\begin{align}
\La(\dC_i, L(j)) = (\alpha_j, \alpha_j) \ep_j^*(\dC_i) + (\alpha_j, w\La_i - \La_i) = (\la_i, \alpha_j).
\end{align}

Applying Theorem~\ref{lem: c braider}, we have a non-degenerate
graded braider
$(\dC_i, \coR_{\dC_i}, \dphi_i)$ for $i\in I$.

\begin{prop}[{\cite[Proposition 5.1]{KKOP21}}] \label{Prop: canonical braiders} 
The family $(\dC_i, \coR_{\dC_i}, \dphi_i)_{ i\in I}$
is a real commuting family of  non-degenerate  graded braiders in $R\gmod$ and
$$
\dphi_i(\beta)  = - ( \la_i , \beta) \qquad \text{for any $\beta\in \rtl$.}
$$
\end{prop}

\begin{thm}[{\cite[Theorem 5.2]{KKOP21}}] \label{Thm: R Ci iso}  
For $i\in I$ and any $R(\beta)$-module $N$ in $\catC_w$,
$ \coR_{\dC_i}(N)$ is an isomorphism.
\end{thm}

We set $\lG \seteq \soplus_{i\in I}\Z\La_i$ and define a $\Z$-bilinear map $\gH\cl \lG \times \lG \rightarrow \Z$ defined by
$$
\gH( \La_i, \La_j ) \seteq  - \tLa(\dC_i, \dC_j) = (\La_i, w \La_j - \La_j).
$$
Then, for $i,j\in I$, we have
\begin{align*}
\dphi_i( -  \wt( \dC_j )) &= ( w\La_i + \La_i, w\La_j - \La_j)
 = (   \La_i, w\La_j) - (  \La_j, w\La_i)\\
&= \gH(\La_i, \La_j) - \gH(\La_j, \La_i).
\end{align*}
Thanks to Theorem \ref{Thm: graded localization} and
Proposition~\ref{Prop: canonical braiders}, we have the localization  of $\catC_w$  
by the non-degenerate graded braiders $\set{\dC_i}{i\in I}$ which we denote by
$$
\lRg \seteq  \catC_w [ \dC_i^{\conv -1 } \mid i \in I ]. 
$$
By the choice above of $ \gH$, for $ \La = \sum_{i\in I} m_i \La_i \in  \lG_{\ge0}$, we have
$$
 \dC_\La  \simeq (\triv, \La) ,\qquad \dC_\La^{\conv-1}\simeq q^{\gH(\La,\La)}(\triv, -\La).
$$
Thus, for $\La = \sum_{i\in I} a_i \La_i \in \lG\subset\wtl$, 
we simply write $\dC_\La \seteq ( \triv,  \La )\in \lRg$. 
We have 
$$\dC_\La\conv \dC_{\La'} \simeq q^{\gH(\La,\La')} \dC_{\La+\La'}.$$

The following is a summary of the results in \cite{KKOP21}
on the localization $\lRg$ of $\catC_w$ and  the localization functor $\Phi_w\cl \catC_w \to \lRg$: 
\bnum
\item the objects $\Phi_w(\dC_i)$ are invertible in $\lRg$,
\item  $\Phi_w\bl R_{C_\La}(X) \br$ is an isomorphism for any $\La\in\wlP_+$ and $X\in \Cw$, 
\item for any simple object $S$ of $\Cw$,
the object $\Phi_w(S)$ is simple in $\tcatC_w$,
 \item
every simple object of $\lRg$ is isomorphic to $\dC_{ \La } \circ \Phi_w(S) $ for some simple object $S$ of $\catC_w$ and $ \La \in  \wlP $,
\item
 for two simple objects $S$ and $S'$ in $\catC_w$ and $\La,\La'\in \wlP$,
$\dC_{\La} \conv  \Phi_w(S) \simeq \dC_{\La'} \conv  \Phi_w(S') $ in $\lRg$
if and only if $q^{\gH(\La,\mu)} \dC_{\La+\mu} \conv S  \simeq  q^{\gH(\La',\mu)} \dC_{\La'+\mu} \conv  S'  $ in
$\catC_w$
for some $\mu\in \wlP$
such that $\La+\mu, \La'+\mu \in \wlP_+$, 
\item the category $\lRg$ is abelian and every objects has
 finite length. 
\item  the grading shift functor $q$ and the contravariant functor $M \mapsto M^\star$ on $\catC_w$   are extended to $\lRg$.
\item 
for any simple module $M\in\lRg$, there exists a unique $n\in\Z$ such that
$q^nM$ is self-dual.
\end{enumerate}

\begin{thm} [{\cite[Theorem 5.7]{KKOP21}}]
Every simple object in $\lRg$ has a right dual.
\end{thm}

Applying Theorem \ref{Thm: graded localization} and
Proposition~\ref{Prop: canonical braiders} again, we obtain the localization 
$$\tRm[w]\seteq R\gmod [ \dC_i^{\conv -1 } \mid i \in I ]$$ of  the category  $R\gmod$ by the  real commuting family  $\set{\dC_i}{i\in I}$ of non-degenerated graded braiders. 
We denote  the localization functor by 
$$Q_w\cl R\gmod \to \tRm[w].$$
Since $\catC_w$ is a full subcategory of $R\gmod$, there is a fully faithful monoidal functor 
$$\iota_w \cl \lRg \to \tRm[w].$$

Set $I_w\seteq\set{i\in I}{w\La_i\neq \La_i}$. Note that $I_w=\{i_1,\ldots,i_l\}$ for any reduced expression $\underline{w}=s_{i_1}\cdots s_{i_l}$ of $w$.
\begin{thm}[{\cite[Theorem 5.8, Theorem 5.9]{KKOP21}}] \label{thm: Cwequiv}
Assume that $I=I_w$.
\bnum
\item The functor $\iota_w \cl \lRg \to \tRm[w]$ is an equivalence of categories.
\item  The category $\lRg$ is left rigid, i.e., every object of $\lRg$ has a left dual.
\end{enumerate}
\end{thm}

We shall prove later in Theorem~\ref{th:rigid} that $\lRg$ is rigid.

\Prop 
\label{prop:Locsim}
Assume that $I=I_w$.
Let $X\in R\gmod$ be a simple module.
Then we have
\bnum
\item $Q_w(X)$ is either a simple module or zero,
\item  $Q_w(X)\simeq q^{H(\La,\La)}\dC_{-\La}\conv Q_w(\dC_\La\hconv X)$ 
for any $\La\in\pwtl$.
\ee
\enprop
\Proof
(i) follows from  {\cite[Proposition 4.8 (i)]{KKOP21}}. 

\snoi
(ii)\ Since there is an epimorphism $Q_w(\dC_\La) \conv Q_w(X)  \epito Q_w(\dC_\La\hconv X)$, we may assume that  $Q_w(X)\not\simeq 0$.  

Applying the exact monoidal functor $Q_w$,
$$\dC_\La\conv X\epito\dC_\La\hconv X\monoto X\conv\dC_\La,$$
we obtain
$$Q_w(\dC_\La)\conv Q_w(X)\epito Q_w(\dC_\La\hconv X)\monoto Q_w(X)\conv Q_w(\dC_\La)$$
whose composition is an isomorphism.
Hence we obtain
$Q_w(\dC_\La\hconv X)\simeq Q_w(\dC_\La)\conv Q_w(X)$.
\QED

\section{Normal sequences and Generalized determinantial modules}

\subsection{Unmixed pair} \label{subsubsec:unmixed}
We say that an ordered pair $(M,N)$ of $R$-modules is {\em unmixed}
if $$\sgW(M)\cap\gW(N)\subset\{0\}.$$ 

\Prop [{\cite[Lemma 2.6]{TW16},  \cite[Proposition 2.12]{KKOP18}}]  \label{prop:unmixedr}
Let $\beta,\gamma\in\prtl$ with $\height{\beta}=m$ and $\height{\gamma}=n$.
Let $M$ and $N$ be an $R(\beta)$-module and an $R(\gamma)$-module, respectively.
 Assume that $(M,N)$ is an unmixed pair. 
Then, we have $e(\beta,\gamma)(M\conv N)=M \etens N$ and 
$e(\beta,\gamma)(N\conv M)=\tau_{w[n,m]}(N\etens M)$.
There is an $R(\beta)\tens R(\gamma)$-module isomorphism
  $q^{-(\beta,\gamma)}M\tens N \to e(\beta,\gamma)(N\conv M)$ 
  given by
$$r(u\tens v)=\tau_{w[n,m]}(v\tens u)\qt{for any $u\in M$ and $v\in N$.}$$
In particular, it induces a homomorphism
$r\cl M\conv N\to q^{(\beta,\gamma)}  N\conv M$. 
\enprop
We denote by $\rmat{M,N}$ the above morphism  $r$.

\Prop [{\cite[Lemma 2.6]{TW16} }] \label{cor:resthd}  
Let $(M,N)$ be an unmixed pair of simple $R$-modules. 
Then we have
$$\HOM(M\conv N,N\conv M)=\corp\,\rmat{M,N}.$$
Moreover, the image of  $\rmat{M,N}\cl M\conv N\to q^{(\beta,\gamma)}  N\conv M$ 
is simple and isomorphic to 
$M\hconv N$ and $q^{(\beta,\gamma)}N\sconv M$.
\enprop
\begin{proof}
The first assertion follows from \eqref{eq:adjoints}.
For any non-zero submodule $S$ of $N\conv M$, we have $e(\beta,\gamma) S=\tau_{w[n,m]}(N\etens M)$  by \eqref{eq:adjoints} and Proposition \ref{cor:resthd}.
It follows that $N\conv M$ has a simple socle which is generated by $\tau_{w[n,m]}(N\etens M)$.
Since the image of $\rmat{M,N}$ is generated by $\tau_{w[n,m]}(N\etens M)$, we get the second assertion.
\end{proof}

\Cor \label{cor:unmixedtLa}
Let $(M,N)$ be an unmixed pair of simple modules such that
one of them is \afr. 
Then we have
$$\tLa(M,N)=0.$$
\encor
\begin{proof}
Let $r$ be the morphism in the above proposition. Since $r=\rmat{M,N}$ up to a constant multiple,  we have $\La(M,N)=-(\beta,\gamma)$. It follows that $\tLa(M,N)=\frac{1}{2}(\La(M,N)+(\beta,\gamma))=0$, as desired.
\end{proof}

\Prop\label{prop:varunmixed}
Let $\al,\beta,\gamma\in\prtl$, and let $L$ be an $R(\al)$-module, $M$ an $R(\beta)$-module
and $N$ an $R(\gamma)$-module.
Assume that
$$\bl\sgW(L)+\sgW(M)\br\cap\gW(N)=\st{0}.$$
Then we have
\eq
&&e(\al+\beta,\gamma)\bl L\conv M\conv N\br\simeq (L\conv M)\tens N,\label{eq:20}\\
&&e(\al+\beta,\gamma)\bl L\conv N\conv M\br\simeq q^{-(\beta,\gamma)}(L\conv M)\tens N,\label{eq:22}\\
&&e(\al+\beta,\gamma)\bl N\conv L\conv M\br\simeq q^{-(\al+\beta,\gamma)}(L\conv M)\tens N.\label{eq:23}
\eneq
Assume further that $L,M,N$ are simple. Then we have
\eq
&&e(\al+\beta,\gamma)\bl L\conv (M\hconv N)\br\simeq (L\conv M)\tens N,\label{eq:21}\\
&&e(\al+\beta,\gamma)\bl(L\hconv N)\conv M\br\simeq q^{-(\beta,\gamma)}(L\conv M)\tens N.\label{eq:24}
\eneq
\enprop

\Proof
The isomorphisms \eqref{eq:20} and \eqref{eq:23}
follow from Proposition~\ref{prop:unmixedr} since
$\sgW(L\conv M)\cap\gW(N)\subset\st{0}$.

\medskip
Let us prove the second isomorphism \eqref{eq:22}.
By the shuffle lemma (Proposition~\ref{prop: shuffle lemma}),
the $R(\al+\beta)\tens R(\gamma)$-module $e(\al+\beta,\gamma)L\conv N\conv M$
has a filtration whose graduations are of the form
$$G\seteq \bl R(\al+\beta)e(\al_1,\beta_1,\gamma_1)\tens R(\gamma)
e(\al_2,\beta_2,\gamma_2)\br
\tens_A\bl e(\al_1,\al_2)L\tens e(\gamma_1,\gamma_2)N\tens
e(\beta_1,\beta_2)M\br$$
up to grade shifts.
Here 
\be[{$\bullet$}]
\item
$\al_k,\beta_k,\gamma_k\in\prtl$ ($k=1,2$) such that $\al=\al_1+\al_2$, $\beta=\beta_1+\beta_2$, $\gamma=\gamma_1+\gamma_2$,
$\al+\beta=\al_1+\beta_1+\gamma_1$ and $\gamma=\al_2+\beta_2+\gamma_2$,
\item $A=R(\al_1)\tens R(\al_2)\tens R(\gamma_1)\tens R(\gamma_2)\tens
R(\beta_1)\tens R(\beta_2)$,
\item $A\to  R(\al+\beta)\tens R(\gamma)$ is given by
$a_1\tens a_2\tens c_1\tens c_2\tens b_1\tens b_2\mapsto (a_1c_1b_1)
\tens (a_2c_2b_2)$.
\ee
If $G\not=0$, then $\al_2\in\sgW(L)$ and $\beta_2\in\sgW(M)$,
$\gamma_1\in\gW(N)$.
Since $\al_2+\beta_2=\gamma_1$, we have
$\al_2=\beta_2=\gamma_1=0$.
Hence only one $G$ survives, and
we have
$$G=\bl R(\al+\beta)e(\al,\beta)\tens R(\gamma)\br\stens_{R(\al)\tens R(\gamma)\tens R(\beta)} (L\tens N\tens M),$$
where
$R(\al)\tens R(\gamma)\tens R(\beta)\to R(\al+\beta)\tens R(\gamma)$
is given by
$a\tens c\tens b\mapsto (a\cdot b)\tens c$.
Hence we obtain
$G\simeq(L\conv M)\tens N$.
It implies the isomorphism \eqref{eq:22}.
Note that
$$q^{-(\beta,\gamma)}(L\conv M)\tens N 
\to e(\al+\beta,\gamma)\bl L\conv N\conv M\br$$
is induced by
$q^{-(\beta,\gamma)}L\tens M\tens N
\to L\tens(N\conv M)$.

\medskip

Finally, let us show the isomorphisms  \eqref{eq:21} and \eqref{eq:24}.

We have commutative diagrams
$$\xymatrix@C=2.5ex{
e(\al+\beta,\gamma)\bl L\conv M\conv N\br\ar@{->>}[r]\ar[d]^\bwr&
e(\al+\beta,\gamma)\bl L\conv (M\hconv N)\br\akew[1ex]\ar@{>->}[r]&e(\al+\beta,\gamma)q^{(\beta,\gamma)}\bl L\conv N\conv M\br\ar[d]^\bwr\\
(L\conv M)\tens N\ar[rr]^\sim&&
q^{(\beta,\gamma)-(\beta,\gamma)}(L\conv M)\tens N,}
$$
$$\xymatrix@C=2.5ex{
e(\al+\beta,\gamma)\bl L\conv N\conv M\br\ar@{->>}[r]\ar[d]^\bwr&
e(\al+\beta,\gamma)\bl(L\hconv N)\conv M\br\akew[1ex]\ar@{>->}[r]&e(\al+\beta,\gamma)q^{(\al,\gamma)}\bl N\conv L\conv M\br\ar[d]^\bwr\\
q^{-(\beta,\gamma)}(L\conv M)\tens N\ar[rr]^\sim&&
q^{(\al,\gamma)-(\al+\beta,\gamma)}(L\conv M)\tens N.}
$$
Here the vertical arrows are isomorphisms by \eqref{eq:20},
\eqref{eq:22} and \eqref{eq:23}.
Hence we obtain \eqref{eq:21} and \eqref{eq:24}.
\QED

\Prop
Let $M$ be an \afr simple module, and $L$ an $R$-module.
We assume that 
$$\text{$\La(M,S)=\La(M,L)$ for any simple quotient $S$ of $L$.}$$
Then the head of $M\conv \hd(L)$ is equal to the head of
$M\conv L$.
\enprop
\Proof
Note that, for any simple quotient $S$,
the following diagram commutes up to a constant multiple
by \cite[Proposition 3.2.8]{KKKO18}:
$$\xymatrix{
M\conv L\ar@{->>}[d]\ar[r]^{\rmat{M,L}}&L\conv M\ar@{->>}[d]\\
M\conv S\ar[r]^{\rmat{M,L}}&S\conv M.
}
$$
In particular, we have
\eq&&\text{the composition
$M\conv L\To[{\rmat{M,L}}]L\conv M\to S\conv M$ does not vanish.}\label{eq:nonvan}
\eneq

Let $K$ be a maximal submodule of $M\conv L$.
In order to see the proposition, it is enough to show that
$M\conv T\subset K$ for some maximal module $T$ of $L$.
 Let us consider the following commutative diagram
where $r$ is the $R$-matrix $\rmat{M,L}\cl M\conv L\to L\conv M$: 
$$\xymatrix@C=8ex{
M\conv K\ar[d]\ar[rr]&&K\conv M\ar[d]\\
M\conv M\conv L\ar[r]^{\rmat{M,M}}&
M\conv M\conv L\ar[r]^{M\circ\rmat{}}&M\conv L\conv M\,.}
$$
Hence we have
$M\conv\rmat{}(K)\subset K\conv M$.
Hence there exists a submodule $P\subset L$
such that $\rmat{}(K)\subset P\conv M$ and
$M\conv P\subset K$.
Hence $ P\not= L$. Let us take a maximal submodule $T\subsetneqq L$
such that $P\subset T$.
Since the composition
$M\conv L\To[{\rmat{}}]L\conv M\to (L/T)\conv M$
does not vanish by \eqref{eq:nonvan}, 
$K\subset \rmat{}^{-1}(T\conv M)\not=M\conv L$.
Hence we obtain
$\rmat{}^{-1}(T\conv M)=K$. 
Hence
$M\conv T\subset\rmat{}^{-1}(T\conv M)=K$.
\QED

\Cor
Let $N_j$ $(j=1,\ldots n)$ be a simple module,
and set $L\seteq N_1\conv\cdots\conv N_n$.
Let $M$ be an \afr simple module.
We assume that 
$M$ and $N_j$ commute.
Then $M\conv \hd(L)$ is semisimple and is equal to the head of
$M\conv L$.
\encor

\begin{proof}
By\ \cite[Proposition 3.2.10]{KKKO18}, 
we have $\La(M,S) = \sum_{i=1}^n \La(M,N_i)$ for any  quotient  $S$ of $L$. 
Note that every simple quotient of $L$ commutes with $M$. 
Thus $M\conv \hd(L)$ is semisimple.
Then by the proposition above,
\eqn
M\conv \hd(L) \simeq \hd(M\conv \hd(L)) = \hd(M\conv L),
\eneqn
 as desired. \end{proof}

\subsection{Normal sequences}
\begin{df}
 Let $(M_1,\ldots,M_r)$ be a sequence  of simple modules  in 
$R \gmod$ such that   $M_k $  is \afr except for  possibly one $k$. 
The sequence $(M_1,\ldots,M_r)$
is called a \emph{normal sequence} 
if   the composition of r-matrices 
\eqn \rmat{M_1,\ldots,M_r}\seteq
&&(\rmat{M_{r-1},M_r})  \circ \cdots \circ (\rmat{M_2,M_r}\circ \cdots \circ \rmat{M_2,M_3})  \circ (\rmat{M_1,M_r} \circ \cdots  \circ \rmat{M_1,M_2}) 
\\
  &&\cl M_1\conv \cdots \conv M_r \longrightarrow M_r \conv \cdots \conv  M_1
\eneqn
does not vanish.
\end{df}

Note that if $(M_1,\ldots,M_r)$ is a normal sequence,
then
$\Im (\rmat{M_1,\ldots,M_r})$ is simple, and it is isomorphic to the head of $M_1\conv \cdots\conv  M_r$
and to the socle of $M_r\conv \cdots\conv  M_1$.

\Lemma [{\cite[Lemma 2.7,  Lemma 2.8]{KK19}}] \label{lem:normal}
 Let $(L_1,\ldots,L_r)$ be a sequence  of  simple modules  in 
$R \gmod$ such that   $L_k $  are \afr except  for  possibly one $k$. 
Then  the following three conditions  are equivalent.
\bna
\item
$(L_1,\ldots,L_r)$    is a normal sequence. 
\item
$(L_2,\ldots,L_r)$ is a normal sequence and 
$$\La(L_1, \hd(L_2\conv\cdots \conv L_r)) = \sum_{2\le j\le r} \La(L_1,L_j).$$
\item
$(L_1,\ldots,L_{r-1})$ is a normal sequence and 
$$\La(\hd(L_2\conv\cdots \conv L_{r-1}), L_r) = \sum_{1\le j\le r-1} \La(L_j,L_r).$$
\end{enumerate}
\enlemma

\Prop
Let $M_j$ $(j=1,\ldots n)$ be an \afr module.
We assume that
\eq&&\text{$\tLa(M_j,M_k)=0$ if $1\le j<k\le n$ and $3\le k$.}
\label{cond:tL}
\eneq
Then $(M_1,\ldots,M_n)$ is a normal sequence.
\enprop
Remark that if $(M_1,\ldots, M_n)$ satisfies condition \eqref{cond:tL}, then
$(M_1^{\circ m_1},\ldots, M_n^{\circ m_n})$ also satisfies \eqref{cond:tL}
for any $m_1\ldots, m_n\in\Z_{\ge0}$.
\Proof
Let us show it by induction on $n$.
We may assume that $n\ge3$.
We have $$0\le\tLa(\hd(M_1\conv\cdots\conv M_{n-1}),M_n)\le\sum_{i=1}^{n-1}\tLa(M_i, M_n)=0,$$
which implies
$$\La(\hd(M_1\conv\cdots\conv M_{n-1}),M_n)=
\sum_{i=1}^{n-1}\La(M_i, M_n).$$
Now the  conclusion follows from  Lemma \ref{lem:normal}.
\QED

For $m,n\in\Z_{\ge1}$ and $v\in\Sym_m$ and $w\in\Sym_n$, let $v\cct w$ be the element of $\Sym_{m+n}$ defined by
$$\bl v\cct w\br (k)=
\bc v(k)&\text{if $1\le k\le m$,}\\
w(k-m)+m&\text{if $m<k\le m+n$.}
\ec$$

Let $\ble$ be the Bruhat order on $\sym_m$. 
For $1\le k < m$ and $w\in\sym_m$,  we have
\eqn
s_k w \ble w \quad \text{if and only if} \quad w^{-1} (k) > w^{-1} (k+1).
\eneqn
Note also that
\eqn
\text{if $u \ble w$,  $s_k w \ble w $ and $ u \ble s_k u$, then
$u \ble s_kw$ and $s_k u \ble w$.}
\eneqn

We denote by $\sym_{\ell,m}$ the set of minimal length left coset representatives in $\sym_{\ell+m}$ with respect to the subgroup  
$\sym_\ell \times \sym_m$.
Namely, 
\eqn
\sym_{\ell,m}
=\bigl\{w \in \sym_{\ell+m}\bigm|\text{$w$ is increasing on $[1,\ell]$ and on $[\ell+1,\ell+m]$}\bigr\}.
\eneqn

The following lemma will be used in the proof of Theorem~\ref{th:strong},
one of the main results of this section.
In the lemma, $\one_n$ denotes the unit of $\Sym_n$.

\begin{lem}\label{sublem:emn}
Let $\ell, m,n\in\Z_{>0}$,  $w\in\Sym_{m,\ell}$, $v\in\Sym_{n,\ell}$.
Set $v_0=w[n,\ell]$.
 Then one has
\bnum
\item\label{item:a}
if $1\le k<m+\ell$ satisfies $s_kw\ble w$,
then 
\bna
\item $w^{-1}(k+1)\le m<w^{-1}(k)\le m+\ell$,
\item $s_kw\in\Sym_{m,\ell}$,
\item 
$(s_kw\cct \one_n)\cdot(\one_m\cct  v)
\ble(w\cct \one_n)\cdot(\one_m\cct  v)$,
\ee
\item $\ell\bl(w\cct \one_n)\cdot(\one_m\cct  v)\br
=\ell(v)+\ell(w)$, \label{item:additive}
\item $\one_m\cct  v\ble(w\cct \one_n)\cdot(\one_m\cct  v)$,
\label{item:c}
\item $(w\cct \one_n)\cdot(\one_m\cct  v_0)\in\Sym_{m+n,\ell}$,
\label{item mnl}
\item
If $\one_m\cct v_0\ble(w\cct \one_n)\cdot(\one_m\cct  v)$,
then $v=v_0$.
\label{item:v0}
\ee
\end{lem}
\Proof
\eqref{item:a}\ Let us set
$a\seteq w^{-1}(k+1)<b\seteq w^{-1}(k)$.
Then we have
$w(b)<w(a)$.
Hence we have $a\le m<b$.
We have
\eqn
s_kw(i)&&=\bc
w(i)<k&\text{if $1\le i<a$,}\\
 k &\text{if $i=a$,}\\
w(i)>  k+1 &\text{if $a<i\le m$,}\\
w(i)< k &\text{if $m<i<b$,}\\
k+1 &\text{if $i=b$,}\\
w(i)> k&\text{if $b<i\le m+\ell$.}\\
\ec
\eneqn
Hence we have $s_kw\in\Sym_{m,\ell}$.

Now we have
\eqn
(\one_m\cct v)^{-1} \cdot (w\cct\one_n)^{-1}(k)
&&=(\one_m\cct v^{-1})(b)>m,\\
(\one_m\cct v)^{-1} \cdot (w\cct\one_n)^{-1}(k+1)
&&=(\one_m\cct v^{-1})(a)=a\le m.
\eneqn
Hence we have
$(s_kw\cct \one_n)\cdot(\one_m\cct  v)
\ble(w\cct \one_n)\cdot(\one_m\cct  v)$.

\snoi
\eqref{item:additive} follows from \eqref{item:a} by induction on $\ell(w)$.

\snoi
\eqref{item:c} follows from \eqref{item:a}.

\mnoi
\eqref{item mnl}\ For $1\le k\le m+n+\ell$, we have
\eqn
(w\cct \one_n)\cdot(\one_m\cct  v_0)(k)
&&=\bc
w(k)\le m+\ell &\text{if $1\le k\le m$,}\\
k+\ell>m+\ell&\text{if $m<k\le m+n$,}\\
w(k-n)&\text{if $m+n<k\le m+n+\ell$.}
\ec
\eneqn
Hence $(w\cct \one_n)\cdot(\one_m\cct  v_0)$
is increasing on $[1,m+n]$ and $[m+n+1,m+n+\ell]$.

\mnoi
\eqref{item:v0}\ \; We shall prove it by induction on $\ell(w)$.
When $w=\one_{m+\ell}$, it is obvious. Assume that $\ell(w)>0$.
Take $k$ such that $1\le k<m+\ell$ and $s_kw\ble w$.
Set $x=(w\cct \one_n)\cdot(\one_m\cct  v)$. Then
$s_kx\ble x$ by \eqref{item:a}.

On the other hand, we have, for any $i$ such that $1\le i\le m+\ell$,
\eqn
\bl\one_m \cct w[n,\ell]\br^{-1}(i)&&=\bl\one_m\cct w[\ell,n]\br(i)\\
&&=\bc
i&\text{if $1\le i\le m$,}\\
i+n&\text{if $m<i\le m+\ell$.}\ec
\eneqn
Hence $(\one_m \cct v_0)^{-1}$ is increasing on $[1,m+\ell]$ and hence
$\one_m \cct v_0\ble s_k(\one_m \cct v_0)$.
Since $s_kx\ble x$ and $\one_m \cct v_0\ble x$,
we obtain
$$\one_m \cct v_0\ble s_kx=(s_kw\cct \one_n)\cdot(\one_m\cct  v).$$
Thus the induction hypothesis implies that $v=v_0$.
\QED

\Th \label{th:strong}
Let $L$ be an \afr simple module,
and let $M$, $N$ be simple modules.
\bnum
\item For any 
 simple quotient $S$ of $M\conv N$,
we have
$$\tLa(L,M)\le \tLa(L,S).$$
\item For any 
 simple quotient $S$ of $M\conv N$,
we have
$$\tLa(N,L)\le \tLa(S,L).$$
\item
For any 
 simple submodule $S$ of $M\conv N$,
we have
$$\tLa(L,N)\le \tLa(L,S).$$
\item
For any 
 simple submodule $S$ of $M\conv N$,
we have
$$\tLa(M,L)\le \tLa(S,L).$$
\ee
\enth

\Proof
We shall prove only (i), since the other statements are obtained from (i) by applying $\psi_*\cl (R\gmod)^\rev\isoto R\gmod$ or the duality functor $\star$.

Let $(\tL,z)$ be an affinization of $L$ with $\deg z=1$.

Let $L\in R(\al)\gmod$, $M\in R(\beta)\gmod$,
$N\in R(\gamma)\gmod$, and $\ell=\height{\al}$,
$m=\height{\beta}$, $n=\height{\gamma}$.
Set $w_0=w[m,\ell]\in\Sym_{m+\ell}$ and
$v_0=w[n,\ell]\in\Sym_{n+\ell}$.
Note that 
$R(\al+\beta)e(\beta,\al)=\sum_{w\in\Sym_{m,\ell}}\tau_w\bl R(\beta)\etens R(\al)\br$.
Write
$$\vphi_{w[m,\ell]}e(\beta,\al)=\sum_{w\in\Sym_{m,\ell}}\tau_wa_{w}^{(\beta,\al)},$$
where $a_{w}^{(\beta,\al)}\in R(\beta)\etens R(\al)\subset 
e(\beta,\al)R(\beta+\al)e(\beta,\al)$.
Similarly we define $a_{v}^{(\gamma,\al)}\in R(\gamma)\etens R(\al)$ for $v\in\Sym_{n,\ell}$.
Note that
\eq  \label{eq:av0}
a_{v_0}^{(\gamma,\al)}=\sum_{\nu\in I^\gamma,\;\mu\in I^\al}\hs{3ex}
\Bigl(\prod_{\substack{1\le i\le n<j\le n+\ell,\\
\nu_i=\mu_{j-n}}}(x_i-x_j)\Bigr)e(\nu)\etens e(\mu).
\eneq
Then we have
$$\Rmat_{\tL, M}(u\etens x)
=\sum_{w\in\Sym_{m,\ell}}\tau_wa_{w}^{(\beta,\al)}(x\etens u)
\qt{for $u\in\tL$ and $x\in M$.}$$

Note that $a_{w}^{(\beta,\al)}(x\etens u)\in M\etens\tL$.

By the shuffle lemma, we have
$$M\conv \tL=\soplus\limits_{w\in\Sym_{m,\ell}}\tau_w(M\etens \tL).$$
Hence, in order to see (i), it is enough to show that
\eq
&&a_{w}^{(\beta,\al)}(M\etens \tL)\subset 
z^{-2\tLa(L,S)+\deg a_{w_0}^{(\beta,\al)}}(M\etens \tL)\label{eq:aw}
\eneq
for all $w\in \Sym_{m,\ell}$.

\medskip
Let $\xi\cl R(\beta)\etens R(\al)\to R(\al+\beta+\gamma)$
be the algebra homomorphism $b\etens a\mapsto b\etens e(\gamma)\etens a$.
Then, for $u\in\tL$, $x\in M$ and $y\in N$, we have
\eqn
&&\Rmat_{\tL, M\circ N}(u\etens x\etens y)
=\bl M\conv \Rmat_{\tL, N}\br
\Bigl(\bl\sum_{w\in\Sym_{m,\ell}}\tau_wa_{w}^{(\beta,\al)}(x\etens u)\br\etens y\Bigr)\\
&&=\sum_{w\in\Sym_{m,\ell}}\bl\tau_w\etens e(\gamma)\br
\bl e(\beta)\etens \vphi_{w[n,\ell]}\br
\xi\bl a_{w}^{(\beta,\al)}\br(x\etens y\etens u)\\
&&=\hs{-2ex}\sum_{w\in\Sym_{m,\ell},\;v\in\Sym_{n,\ell}}\hs{-2ex}\bl\tau_w\etens e(\gamma)\br
\bl e(\beta)\etens\tau_v\br \bl e(\beta)\etens a_v^{(\gamma,\al)}\br
\xi\bl a_{w}^{(\beta,\al)}\br(x\etens y\etens u).
\eneqn
Let us denote by $h\cl M\tens N\monoto S$
 the composition
$M\tens N\to M\conv N\to S$. 
It is
$R(\beta)\tens R(\gamma)$-linear and
injective,
since $M\tens N$ is a simple $R(\beta)\tens R(\gamma)$-module.

Then we have
\eqn
\Rmat_{\tL, S} \bl u\etens h(x\tens y)\br
&&=\hs{-3ex}\sum_{w\in\Sym_{m,\ell},\;v\in\Sym_{n,\ell}}\hs{-3ex}
\tau_{(w\cct \one_n)(\one_m\cct  v)}
 a_{w,v}^{(\beta,\gamma,\,\al)}(h(x\tens y)\etens u)\\
&&=\sum_{w\in\Sym_{m,\ell}}
\tau_{(w\cct \one_n)(\one_m\cct  v_0)} 
a_{w,v_0}^{(\beta,\gamma,\,\al)}(h(x\tens y)\etens u)\\
&&\hs{10ex}+\hs{-3ex}\sum_%
{\substack{w\in\Sym_{m,\ell},\\v\in\Sym_{n,\ell}\setminus\st{v_0}}}\hs{-3ex}
\tau_{(w\cct \one_n)(\one_m\cct  v)} 
a_{w,v}^{(\beta,\gamma,\,\al)}(h(x\tens y)\etens u),
\eneqn
where $ a_{w,v}^{(\beta,\gamma,\,\al)}=\bl e(\beta)\etens a_v^{(\gamma,\al)}\br
\xi\bl a_{w}^{(\beta,\al)}\br\in R(\beta)\etens R(\gamma)\etens R(\al)$.

We write $\Rnor_{\tL,S}=z^{-s} \Rmat_{\tL,S}$.
Note that
$$\tLa(L,S)=\bl\deg  a_{w_0,v_0}^{(\beta,\gamma,\,\al)}-s\br/2
=\bl \deg a_{w_0}^{(\beta,\,\al)}+\deg a_{v_0}^{(\gamma,\,\al)}-s\br/2 .$$

Thus we obtain
\eq
&&\ba{l}
\hs{3ex}\sum\limits_{w\in\Sym_{m,\ell}}
\tau_{(w\cct \one_n)(\one_m\cct  v_0)} 
a_{w,v_0}^{(\beta,\gamma,\,\al)}(h(x\tens y)\etens u)\\
\hs{20ex}
+\sum\limits_%
{\substack{w\in\Sym_{m,\ell},\\v\in\Sym_{n,\ell}\setminus\st{v_0}}}\hs{-3ex}
\tau_{(w\cct \one_n)(\one_m\cct  v)} 
a_{w,v}^{(\beta,\gamma,\,\al)}(h(x\tens y)\etens u)\\
\hs{40ex}\in z^s (S\conv \tL).
\ea
\eneq

By Lemma~\ref{sublem:emn}, we have
$(w\cct \one_n)(\one_m\cct  v)\not\succeq  \one_m\cct  v_0$
for $v\in\Sym_{n,\ell}\setminus\{v_0\}$.
Hence, we have
\eqn
&&\tau_{(w\cct \one_n)(\one_m\cct  v)} 
a_{w,v}^{(\beta,\gamma,\,\al)}\bl S\etens \tL\br
\subset \tau_{(w\cct \one_n)(\one_m\cct  v)} 
\bl S\etens \tL\br\\
&&\hs{10ex}\subset \sum_{g\in\Sym_{m+n,\ell}\;g\not\succeq\one_m\cct  v_0}
\tau_g(S\etens \tL).\eneqn
On the other hand, the shuffle lemma implies
$$S\conv \tL=\soplus\limits_{g\in\Sym_{m+n,\ell}}\tau_g(S\etens \tL).$$
By Lemma~\ref{sublem:emn}, we have
$(w\cct \one_n)(\one_m\cct  v_0)\in\Sym_{m+n,\ell}$, and
$(w\cct \one_n)(\one_m\cct  v_0)\succeq\one_m\cct  v_0$.
Hence we have
\eqn
a_{w,v_0}^{(\beta,\gamma,\,\al)}(h(x\tens y)\etens u)\in z^s (S\etens \tL)\cap&& \bl h(M\tens N)\etens \tL\br\\
&&=h(M\tens N)\etens z^s\tL\qt{for any $w\in\Sym_{m,\ell}$.}
\eneqn
Since we have $a_{v_0}^{(\gamma,\,\al)}\vert_{N\etens\tL}
\in\bR^\times z^c \id_{N\etens \tL}$  with $c=\deg a_{v_ 0}^{(\gamma,\,\al)}$ by \eqref{eq:av0}, we have
$$z^c \bl a_{w}^{(\beta,\al)}(x\etens u)\br\etens y
\in z^s (M\etens \tL\etens N).$$
Finally, we obtain
$$a_{w}^{(\beta,\al)}(x\etens u)
\in z^{s-c}(M\etens \tL)=z^{-2\tLa(L,S)+\deg a_{w_0}^{(\beta,\al)}}
(M\etens \tL).$$
It is nothing but \eqref{eq:aw}.
\QED

\Cor \label{cor:tLaMN=0}
Let $L$ be an \afr simple module, and
let $M$, $N$ be simple modules.
Let $S$ be a simple quotient of $M\conv N$.

If $\tLa(L,N)=0$, then we have
$$\tLa(L,S)=\tLa(L,M)$$
and hence
$$\La(L,S)=\La(L,M)+\La(L,N).$$
\encor
\Proof
We have
$$\tLa(L,M)\le\tLa(L,S)\le \tLa(L,M \conv N) = \tLa(L,M)+\tLa(L,N)=\tLa(L,M),$$
where the first inequality  follows form Theorem \ref{th:strong}.
Hence we have $\tLa(L,S)=\tLa(L,M)$ and $\tLa(L,S)=\tLa(L,M)+\tLa(L,N)$, which is equivalent to
$\La(L,S)=\La(L,M)+\La(L,N)$, as desired.
\QED

\Cor\label{cor:Normal} 
Let $L$  be   an \afr simple module, and
let $M$, $N$ be simple modules. Assume that $M$ or $N$ is  \afr.
\bnum
\item
If $\tLa(L,N)=0$, then
$(L,M,N)$ is a normal sequence.
\item
If $\tLa(N,L)=0$, then
$(N,M,L)$ is a normal sequence.
\ee
\encor
\Proof
(i)  By Corollary \ref{cor:tLaMN=0}, we have $\La(L,M\hconv N) = \La(L,M) +\La(L,N)$.
Hence $(L,M,N)$ is a normal sequence by Lemma \ref{lem:normal}.

(ii) A similar proof to  (i) works for (ii).
\QED

\subsection{Head simplicity of convolutions with $L(i)$}
\

\begin{df}
For $i\in I$, $\beta\in\rtlp$ and a simple $R(\beta)$-module $M$, define
$$\de_i(M)\seteq \eps_i(M)+\eps^*_i(M)+\ang{h_i,\wt(M)}.$$
\end{df}
Recall the following lemma.
\Lemma [{\cite[Corollary 3.8]{KKOP18}}]
For $i\in I$, $\beta\in\rtlp$ and a simple module $R(\beta)$-module $M$,
we have
\eqn
\tLa(L(i),M)&&=\dfrac{(\al_i,\al_i)}{2}\eps_i(M),\\
\tLa(M,L(i))&&=\dfrac{(\al_i,\al_i)}{2}\eps^*_i(M),\\
\de(L(i),M)&&=\dfrac{(\al_i,\al_i)}{2}\de_i(M).\\
\eneqn
\enlemma

For an $R(\beta)$-module $M$ and $i \in I$, set
$$\wt_i(M) \seteq\lan h_i, \wt(M)\ran =-\lan h_i, \beta \ran.$$

\Prop[{cf.\ \cite{LV11}}]\label{prop:epsBi}
Let $i\in I$ and let $M$ be a simple module.
\bnum
\item 
If $\de_i(M)=0$,
then we have $L(i)\hconv M\simeq L(i)\conv M \simeq  M\conv L(i) \simeq  M\hconv L(i)$ \ro up to a grading shift\/\rf  \,
and $\de_i(L(i)\conv M)=0$.

\noi
If $\de_i(M)>0$,
then we have
\eqn
\de_i(L(i)\hconv M)&=&\de_i(M\hconv L(i))=\de_i(M)-1,\\
\eps_i(M\hconv L(i))&=&\eps_i(M),\quad\eps^*_i(L(i)\hconv M)=\eps^*_i(M).
\eneqn
\item We have
\eq
&&\ba{l}
\de_i(L(i^n)\hconv M)=\max\bl\de_i(M)-n, 0\br,\\
\de_i(M\hconv L(i^n))=\max\bl\de_i(M)-n, 0\br.
\ea\label{eq:dei}
\eneq
\item We have
\eqn
\eps_i(M\hconv L(i^n))&&=\max\bl\eps_i(M),\; n-\wt_i (M)-\eps^*_i (M) \br,\\
\eps^*_i(L(i^n)\hconv M)&&=\max\bl\eps^*_i(M),\; n-\wt_i(M)-\eps_i(M)\br.\\
\eneqn
\ee
\enprop
\Proof
(i)\ 
The first statement follows from 
\cite[Lemma 3.2.3]{KKKO18} and \cite[{Corollary 3.18}] {KKOP21}.

\noi
Assume that $\de_i(M)>0$.
By
\cite[{Corollary 3.18}] {KKOP21},
we have $\de_i(M\hconv L(i))<\de_i(M)$.
On the other hand we have
$$\eps_i(M\hconv L(i))\ge\eps_i (M), \quad \text{and}$$
 $$ \eps^*_i(M\hconv L(i))+\wt_i(M\hconv L(i)) =\eps^*_i (M)+1 +\wt_i(M) -2 = \eps^*_i (M)+ \wt_i(M) -1. $$ 
Hence, we obtain
$\de_i\bl M\hconv L(i)\br\ge \de_i(M)-1$.
It follows that
$\de_i\bl M\hconv L(i)\br=\de_i(M)-1$ and $\eps_i(M\hconv L(i))=\eps_i( M)$.
For the statements for $L(i)\hconv M$ can be similarly proved.

\smallskip\noi
(ii) follows from (i).

\smallskip\noi
(iii)\
By (ii),  we have
$$\max(\de_i(M)-n,0)=\de_i(M\hconv L(i^n))=\eps_i(M\hconv L(i^n))+
\eps^*_i(M)+n+\wt_i(M)-2n.$$
Hence, we obtain
\eqn
\eps_i(M\hconv L(i^n))
&&=
\max\bl\eps_i(M)+\eps^*_i(M)+\wt_i(M)-n,\;0\br
-\eps^*_i(M)-\wt_i(M)+n\\
&&=\max\bl\eps_i(M),\;
-\eps^*_i(M)-\wt_i(M)+n\br.
\eneqn
The statement for $\eps_i^*$ is similarly proved.
\QED

\Th\label{th:2istring}
Let $i\in I$ and let $M$ be a simple module.
Assume that  $a,b\in\Z_{\ge0}$ satisfy
$$\de_i(M)\ge a+b.$$
Then $L(i^a)\conv M\conv L(i^b)$ has a simple head and a simple socle.
In particular, we have
$$\tF_i^a(\tF_i^*)^bM\simeq(\tF^*_i)^b\tF_i^aM.$$
\enth
\Proof
Assume that $M$ is an $R(\beta)$-module with $n=\height{\beta}$,
and set $L_1=L(i)^{\circ a}$ and $L_2=L(i)^{\circ b}$ 

\mnoi(i)\ First assume that $\eps_i(M)=\eps^*_i(M)=0$.
We shall show that
\eq\label{eq:mult}
\hs{5ex}\parbox{70ex}{the graded $\bl R(a \al_i)\etens R(\beta)\etens R(b \al_i )\br$-module $L_1\etens M\etens L_2$
appears only once in 
$e(i^a,\beta,i^b)(L_1\conv M\conv L_2)$ as a composition factor
(including the grading).}
\eneq
By the shuffle lemma, we have
$$e(i^a,\beta,i^b)(L_1\conv M\conv L_2)=\hs{-2ex}
\soplus_{\substack{w\in \Sym'_{a+n+b},\\
\nu\in I^{\beta+(a+b)\al_i}}} \hs{-2ex}  e(i^a,\beta,i^b)\tau_we(\nu)(L(i)^{\etens\; a}\etens 
M\etens L(i)^{\etens \;b}).$$
Here $ \Sym'_{a+n+b}$ is the set of $w\in \Sym_{a+n+b}$
such that $w\,\vert_{\,[a+1,a+n]}$ is increasing, 
Since $\eps_i(M)=0$ and $\eps_i^*(M)=0$, we may assume that $\nu_{a+1}\not=i$ and $\nu_{a+n}\not=i$.
We may assume $\nu_{w^{-1}k}=i$ for $k\in[1,a]\sqcup[a+n+1,a+n+b]$.
Hence we have
$$a+1\le w(a+1)\qtq w(a+n)\le a+n.$$
Hence we have
$w\vert_{[a+1,a+n]}=\id_{\,[a+1,a+n]}$.
Thus we obtain
$$e(i^a,\beta,i^b)(L_1\conv M\conv L_2)=\hs{-2ex}
 \sum  \limits_{v\in \Sym''_{a+n+b}}  \hs{-2ex} \bl R(a\al_i)\etens e(\beta)
\etens\;R( b\al_i)\br
\tau_{v}\bl L(i)^{\etens\; a}\etens 
M\etens L(i^b)^{\etens\; b}\br.$$
Here $\Sym''_{a+n+b}$ is the set
of $v\in\Sym_{a+n+b}$ such that
$v^{-1}\vert_{[1,a]}$ and $v^{-1}\vert_{[a+n+1,a+n+b]}$  are increasing and
$v\vert_{[a+1,a+n]}=\id_{[a+1,a+n]}$.
The above gives an  $R(a\al_i)\etens R(\beta)\etens R(b\al_i)$-submodule filtration of $e(i^a,\beta,i^b)(L_1\conv M\conv L_2)$ which is compatible with $\preceq$ on $ \Sym''_{a+n+b}$.

More precisely, 
we have
the following equality in the Grothendieck group of
$\bl R(a\al_i)\tens R(\beta)\tens R(b\al_i)\br\gmod$:
$$
[e(i^a,\beta,i^b)(L_1\conv M\conv L_2)]
=\sum\limits_{v\in \Sym''_{a+n+b}}\hs{-1ex} 
q^{N(v)}[L_1\tens M\tens L_2].
$$
Here, $N(v)$ is the homogeneous degree of
$e(i^a,\beta,i^b)\tau_{v}$.

For $v\in \Sym''_{a+n+b}$, let $k$ be the number of
$s\in[1,a]$ such that $v^{-1}(s)\in[a+n+1,a+n+b]$.
Then $k$ is also equal to the number of $t\in [a+n+1,a+n+b]$
such that $v^{-1}(t)\in[1,a]$.
Then $0\le k\le\min(a,b)$ and
we have
\eqn
&&[1,a]\cap v([1,a])=[1,a-k],\\
&&{[}1,a]\cap v([a+n+1,a+n+b])=[a-k+1,a],\\
&&{[}a+n+1,a+n+b]\cap v([1,a])=[a+n+1,a+n+k],\\[1ex]
&&{[}a+n+1,a+n+b]\cap v([a+n+1,a+n+b])\\
&&\hs{34ex}=[a+n+k+1,a+n+b].
\eneqn
Then the homogeneous degree $N(v)$ of
$e(i^a,\beta,i^b)\tau_{v^{-1}}$ is equal to
$$N(v)\seteq-2k(\al_i,\beta)-(\al_i,\al_i)\sharp A(v).$$
Here
\eqn
A(v)
\seteq&&
\set{(s,t)}{%
s\in[1,a],\;t\in[a+n+1,a+n+b],\;v(s)> v(t)}\\
=&&A_1\sqcup A_2\sqcup A_3
\eneqn
with
\eqn
A_1&&=
\{(s,t)\;;\;
s\in[a-k+1,a],\;t\in[a+n+k+1,a+n+b],\\
&&\hs{50.3ex}\; v(s)>v(t)\},\\[1ex]
A_2&&=\{(s,t)\;;\;
s\in[1,a-k],\;t\in[a+n+1,a+n+k],\; v(s)>v(t)\},\\
A_3&&=
\{(s,t)\;;\;
s\in[a-k+1,a],\;t\in[a+n+1,a+n+k]\}.\\
\eneqn
Hence, one has
$$\sharp A(v)=\sharp A_1+\sharp A_2+\sharp A_3
\le k(b-k)+k(a-k)+k^2=k(a+b)-k^2.$$
Since $-(\al_i,\beta)=\dfrac{(\al_i,\al_i)}{2}\de_i(M)$, we obtain
\eqn
N(v)&&=(\al_i,\al_i)\bl k\de_i(M)-\sharp A(v)\br\\
&&\ge
(\al_i,\al_i)\Bigl( k(a+b)-\bl k(a+b)-k^2\br\Bigr)=(\al_i,\al_i)k^2.
\eneqn
Hence $N(v)=0$ implies $k=0$ which is equivalent to $v=\id$.
Thus we obtain \eqref{eq:mult}.

\smallskip
Any $R$-submodule $K$ of $L_1\conv M\conv L_2$ is a proper submodule if
and only if 
$L_1\etens M\etens L_2$ does not appear in the
$\bl R(a\al_i)\etens R(\beta)\etens R(b\al_i)\br$-module $e(i^a,\beta,i^b)K$
as a composition factor.
Since the last property is stable by taking sums,
a proper maximal submodule of 
$L_1\conv M\conv L_2$ is unique, and hence
$L(i^a)\conv M\conv L(i^b)$ has a simple head. By duality, 
$L(i^a)\conv M\conv L(i^b)$ has a simple socle.

\mnoi
(ii) General case. 
Set $b'=\eps^*_i(M)$ and $a'=\eps_i(E^*_i{}^{(b')}M)$ and
$M_0=E_i^{(a')}E^{*}_i{}^{(b')}(M)$. Then $\eps_i(M_0)=\eps_i^*(M_0)=0$ and
 we have
$M\simeq \bl L(i^{a'})\hconv M_0\br\hconv L(i^{b'})$.
Then, \eqref{eq:dei} implies that
\eqn
\de_i\bl  L(i^{a'})\hconv M_0\br&&=\max\bl \de_i(M_0)-a',0\br,\\
\de_i(M)&&=\max\bl \de_i\bl  L(i^{a'})\hconv M_0\br-b',0)\\
&&=\max(\de_i(M_0)-a'-b',0).
\eneqn
Hence we obtain
$\de_i({M_0})\ge a+a'+b+b'$. Therefore, (i) implies that the convolution 
$L(i^a)\conv L(i^{a'})\conv M_0\conv L(i^{b'})\conv L(i^ {b})$ has a simple head.
Hence $L(i^a)\conv M\conv L(i^ {b})$ has a simple head.
\QED

\subsection{ Generalization of determinantial modules } 

Recall that $\la \in \wtl$ is $w$-dominant if 
$(\beta,\la)\ge0$ for any $\beta\in\prD\cap w^{-1}\nrD$.
In this case, we have $\la-w\la\in\prtl$. 

\Th \label{thm:gdm}
Let $w\in W$ and let $\la\in\wtl$.
Assume that $\la$ is $w$-dominant.
Then there exists a self-dual simple $R(\la-w\la)$-module
$\Mm_w(w\la,\la)$ which satisfies the following conditions.
\bna
\item
If $i\in I$ satisfies $\ang{h_i,w\la}\ge0$,
then $\eps_i\bl\Mm_w(w\la,\la)\br=0$.\label{item eps}

\item
If $i\in I$ satisfies $\ang{h_i,\la}\le0$,
then $\eps^*_i\bl\Mm_w(w\la,\la)\br=0$.\label{item epsstar}
\item 
If $i\in I$ satisfies $s_iw  \prec w$, then
$\Mm_w(w\la,\la)\simeq L(i^m)\hconv \Mm_{s_iw}(s_iw\la,\la)$
where $m=\ang{h_i,s_iw\la} = \eps_i(\Mm_w(w\la,\la)) \in\Z_{\ge0}$.\label{item a}
\item
If $i\in I$ satisfies $ws_i\prec w$, then
$\Mm_w(w\la,\la)\simeq \Mm_{ws_i}(w\la,s_i\la)\hconv L(i^m) $
where $m=\ang{h_i,\la} = \eps^*_i(\Mm_w(w\la,\la))\in\Z_{\ge0}  $.\label{item: sla} \

\item
For any $\mu\in\wlP_+$ such that
$\la+\mu\in\wlP_+$,  we have
$$\Mm(w\mu,\mu)\hconv \Mm_w(w\la,\la)\simeq\Mm\bl w(\la+\mu),\la+\mu\br$$
up to a grading shift.\label{item:wla}
\ee
Moreover such an $\Mm_w(w\la,\la)$ is unique up to an isomorphism.
\enth

\Proof
The uniqueness of $\Mm_w(w\la,\la)$ follows from \eqref{item:wla} together with 
\cite[Corollary 3.7]{KKKO15}.

Let us show the existence of $\Mm_{w}(w\la,\la)$ satisfying
\eqref{item eps}--\eqref{item:wla} for a $w$-dominant $\la$ by induction on $\ell(w)$.
Assume that $\ell(w)>0$.

\medskip
Take $i\in I$ such that $w'\seteq s_iw \prec w$.
Then there exists $\Mm_{w'}(w'\la,\la)$ satisfying
\eqref{item eps}--\eqref{item:wla} with $w'$ instead of $w$,  since $\la$ is $w'$-dominant. 
By \eqref{item eps}, we have
$\eps_i\bl\Mm_{w'}(w'\la,\la)\br=0$.
Set $m=\ang{h_i,w'\la}\ge0$ and
$\Mm_w(w\la,\la)\seteq L(i^m)\hconv \Mm_{w'}(w'\la,\la)$.
Then $\Mm_w(w\la,\la)$ is self-dual by \cite[Lemma 3.1.4]{KKKO18}, since $\tLa\bl L(i^m), \Mm_{w'}(w'\la,\la)\br=0$ 
 by Corollary \ref{cor:unmixedtLa}.

Let us take $\mu\in\wlP_+$ such that
$\eta\seteq \la+\mu\in\wlP_+$.
Then we have
$$\Mm(w'\mu,\mu)\hconv \Mm_{w'}(w'\la,\la)\simeq\Mm(w'\eta,\eta)
\qt{up to a grading shift}$$
by \eqref{item:wla} for $w'$.
Set $n=\ang{h_i,w'\mu}\ge0$. Note that $n$ is non-negative, since any dominant weight is  $w$-dominant.
Then $m+n=\ang{h_i,w'\eta}$.
Since $\eps_i\bl\Mm_{w'}(w'\la,\la)\br=0$ by \eqref{item eps} for $w'$, 
the triple
$\bl L(i^{m+n}),\, \Mm(w'\mu,\mu),\, \Mm_{w'}(w'\la,\la)\br$ is a normal sequence
by Corollary~\ref{cor:Normal} and Corollary~\ref{cor:unmixedtLa}.
Hence, we conclude that the convolution
$ L(i^{m+n})\conv \Mm(w'\mu,\mu)\conv \Mm_{w'}(w'\la,\la)$ has a simple head.

We have epimorphisms
\eqn
&&L(i^{m+n})\conv \Mm(w'\mu,\mu)\conv \Mm_{w'}(w'\la,\la)\\
&&\hs{10ex}\simeq
L(i^{m})\conv L(i^n)\conv\Mm(w'\mu,\mu)\conv \Mm_{w'}(w'\la,\la)\\
&&\hs{10ex}\epito
L(i^{m})\conv \Mm(w\mu,\mu)\conv \Mm_{w'}(w'\la,\la)\\
&&\hs{10ex}\underset{*}{\simeq}
\Mm(w\mu,\mu)\conv L(i^{m})\conv \Mm_{w'}(w'\la,\la)
\epito
\Mm(w\mu,\mu)\conv \Mm_w(w\la,\la)\\
&&\hs{10ex}\epito
\Mm(w\mu,\mu)\hconv \Mm_w(w\la,\la),
\eneqn
 where $\underset{*}{\simeq}$ follows from \cite[Lemma 4.9]{KKOP18}. 

On the other hand,  by \eqref{item:wla} for $w'$,  we have
\eqn
&&L(i^{m+n})\conv \Mm(w'\mu,\mu)\conv \Mm_{w'}(w'\la,\la)\\
&&\hs{10ex}\epito
L(i^{m+n})\conv \Mm(w'\eta,\eta)
\epito
 \Mm(w\eta,\eta).
\eneqn
Hence we obtain
$\Mm(w\mu,\mu)\hconv \Mm_w(w\la,\la)\simeq \Mm(w\eta,\eta)$.

Thus we obtain \eqref{item:wla}.
Since $\Mm_w(w\la,\la)$ satisfying \eqref{item:wla} is unique,
$\Mm_w(w\la,\la)$ satisfies \eqref{item a}.

\medskip
Let us show \eqref{item epsstar}.
Let us take $j\in I$ such that $w'\seteq s_jw<w$.
Set $m=\ang{h_j,w'\la} \ge 0$.
Then, by \eqref{item a} we have
$$\Mm_w(w\la,\la)\simeq L(j^m)\hconv \Mm_{w'}(w'\la,\la).$$
Hence,  if $i\not=j$, then we have
$\eps^*_i\bl\Mm_w(w\la,\la)\br=\eps^*_i\bl\Mm_{w'}(w'\la,\la)\br=0$, 
where the last equality follows from \eqref{item epsstar} for $w'$. 

\noi
Assume that $i=j$.
Since $\eps_i\bl\Mm_{w'}(w'\la,\la)\br=0$  by \eqref{item eps} for $w'$, Proposition~\ref{prop:epsBi}
implies that
\eqn
&&\eps^*_i\bl\Mm_w(w\la,\la)\br=
\eps^*_i\bl  L(i^m) \hconv \Mm_{w'}(w'\la,\la)\br\\
&&\hs{5ex}=\max\bl\eps^*_i(\Mm_{w'}(w'\la,\la)),
\;m-\wt_i(\Mm_{w'}(w'\la,\la))\br\\
&&\hs{5ex}=\max\bl0,\;m-\ang{h_i,w'\la-\la}\br
=\max\bl0,\,\ang{h_i,\la}\br=0.
\eneqn

\medskip
Let us show \eqref{item: sla}.
It is trivial for $\ell(w)\le 1$.
Hence we assume that $\ell(w)>1$.
Let us take $j\in I$ such that
$s_jws_i<ws_i$.
Set $w'=s_jw$  and 
 $n=\ang{ h_j  ,w'\la}\ge0$.
Then we have
\eqn
&&\Mm_w(w\la,\la)\simeq L(j^n)\hconv \Mm_{w'}(w'\la,\la)\qtq\\
&&\Mm_{ws_i}(w\la,s_i\la)\simeq L(j^n)\hconv \Mm_{w's_i}(w'\la,s_i\la),
\eneqn
 where the second isomorphism follows from \eqref{item: sla} for the pair $ws_i$ and $s_i\la$.

By the induction hypothesis, we have
$$\Mm_{w'}(w'\la,\la)\simeq \Mm_{w's_i}(w'\la,s_i\la)\hconv L(i^m) $$
where $m=\ang{h_i,\la}$.
Let us show that
$L(j^n)\conv \Mm_{w's_i}(w'\la,s_i\la)\conv L(i^m)$ has a simple head. 

If $i=j$, Theorem~\ref{th:2istring}
implies that $L(j^n)\conv \Mm_{w's_i}(w'\la,s_i\la)\conv L(i^m)$ has a simple head,
since we have
\eqn\de_i\bl\Mm_{w's_i}(w'\la,s_i\la)\br&=&\eps_i\bl\Mm_{w's_i}(w'\la,s_i\la)\br
+\eps_i^*\bl\Mm_{w's_i}(w'\la,s_i\la)\br
+\wt_i\bl\Mm_{w's_i}(w'\la,s_i\la)\br\\
&=&\ang{h_i,w'\la-s_i\la}=n+m,
\eneqn
 where the second equality follows from \eqref{item eps}  and \eqref{item epsstar} for   $w's_i$.

\noi
If $i\not=j$, 
$L(j^n)\conv \Mm_{w's_i}(w'\la,s_i\la)\conv L(i^m)$ has a simple head
by Corollary~\ref{cor:Normal}.

Hence we conclude that 
$L(j^n)\conv \Mm_{w's_i}(w'\la,s_i\la)\conv L(i^m)$ has a simple head
in any case.
Hence we have
\eqn
\Mm_w(w\la,\la)\simeq L(j^n)\hconv \Mm_{w'}(w'\la,\la)
&&\simeq\hd\bl L(j^n)\conv \Mm_{w's_i}(w'\la,s_i\la)\conv L(i^m) \br\\
&&\simeq \Mm_{ws_i}(w\la,s_i\la)\hconv L(i^m).
\eneqn
Thus we obtain \eqref{item: sla}. 

\medskip
\noi
Finally let us show \eqref{item eps}.
Let us take $j\in I$ such that $w'\seteq ws_j<w$.
Set $m=\ang{h_j,\la}$.
Then by \eqref{item: sla} we have
$$\Mm_w(w\la,\la)\simeq \Mm_{ws_j}(w\la,s_j\la)\hconv L(j^m).$$
By the induction hypothesis,
$\eps_i\bl\Mm_{ws_j}(w\la,s_j\la)\br=0$.
Hence if $i\not=j$ then
$\eps_i\bl\Mm_w(w\la,\la)\br=0$.
If $i=j$, then $\eps^*_i\bl\Mm_{ws_i}(w\la,s_i\la)\br=0$  by \eqref{item epsstar} and hence 
Proposition \ref{prop:epsBi}
implies 
\eqn
\eps_i\bl\Mm_w(w\la,\la)\br
&&=\eps_i\bl\Mm_{ws_i}(w\la,s_i\la)\hconv L(i^m)\br\\
&&=\max\Bigl(\eps_i\bl\Mm_{ws_i}(w\la,s_i\la)\br,\; m-\ang{h_i,\wt\bl
\Mm_{ws_i}(w\la,s_i\la)\br}
\Bigr)\\
&&=\max\bl 0,\; m-\ang{h_i,w\la-s_i\la}\br
=\max\bl 0,\; -\ang{h_i,  w  \la}\br
=0,
\eneqn
as desired.
\QED

\Lemma
Let $w\in W$ and let $\la,\mu\in \wtl$
be $w$-dominant weights.
We assume that $\Mm_w(w\la,\la)$ 
and $\Mm_w(w\mu,\mu)$ strongly commute. 
Then we have
$$\Mm_w\bl w(\la+\mu),\la+\mu\br\simeq\Mm_w(w\la,\la)\conv\Mm_w(w\mu,\mu)$$
up to a grading shift.
\enlemma
\Proof
Let us argue by induction on $\ell(w)$.
If $\ell(w)>0$, take $i\in I$ such that $w'\seteq s_iw<w$.
Set $m=\ang{h_i,w'\la}$ and $n=\ang{h_i,w'\mu}$. Then
 $m=\eps_i(\Mm_w(w\la,\la))$ and $n=\Mm_w(w\mu,\mu)$ by Theorem  \ref{thm:gdm} \eqref{item a}, and hence \
$\Mm_{w'}(w'\la,\la)$ commutes with $\Mm_{w'}(w'\mu,\mu)$
 by \cite[Lemma 3.1]{KKOP18}.
Hence
we have
\eqn
E_i^{(m+n)}\bl\Mm_w(w\la,\la)\conv\Mm_w(w\mu,\mu)\br
&&\simeq
\Mm_{w'}(w'\la,\la)\conv\Mm_{w'}(w'\mu,\mu)\\
&&\simeq\Mm_{w'}\bl w'(\la+\mu),\la+\mu\br
\eneqn
 by the induction hypothesis, 
which implies that
\eqn
\Mm_w(w\la,\la)\conv\Mm_w(w\mu,\mu)
&&\simeq L(i^{m+n})\hconv\Mm_{w'}\bl w'(\la+\mu),\la+\mu\br\\
&&\simeq\Mm_{w}\bl w(\la+\mu),\la+\mu\br
\eneqn
 by Theorem \ref{thm:gdm} \eqref{item a}.
\QED

\begin{rem}
For $w$-dominant $\la,\mu\in\wtl$,
$\Mm_w(w\la,\la)$ and $\Mm_w(w\mu,\mu)$ do not commute in general.
For example,
when $\g=A_2$, $I=\{1,2\}$, $w=s_1s_2$, $\la=\La_1$, $\mu=s_1\La_1$,
$\Mm_w(w\la,\la)\simeq L(1)$ and $\Mm_w(w\mu,\mu)\simeq L(2)$ do not commute.

Conjecturally, $\Mm_w(w\la,\la)$ and $\Mm_w(w\mu,\mu)$ commute 
if $\la$ and $\mu$ are in the same Weyl chamber (i.e.,
 $(\beta,\la)(\beta,\mu)\ge0$ for any real root $\beta$). 
\end{rem}

\Lemma \label{lem:MhconvMw}
Let $w\in W$ and let $\la\in \wtl$
be a $w$-dominant weight.
Then, for any $\mu\in\wlP_+$,  we have
$$\Mm(w\mu,\mu)\hconv \Mm_w(w\la,\la)\simeq\Mm_w\bl w(\la+\mu),\la+\mu\br$$
up to a grading shift.
\enlemma
\Proof
Let us take $\La\in\wlP_+$ such that
$\La+\la\in\wlP_+$.
 Since
$\Mm(w\La,\La)\conv \bl\Mm(w\mu,\mu) \conv \Mm_w(w\la,\la)\br \simeq
\Mm(w(\La+\mu),\La+\mu) \conv \Mm_w(w\la,\la) $ has a simple head,
it follows that 
\eqn
&&\Mm(w\La,\La)\hconv\bl\Mm(w\mu,\mu)\hconv \Mm_w(w\la,\la)\br\\
&&\hs{10ex}\simeq
 \hd\bl \Mm(w\La,\La)\conv\Mm(w\mu,\mu)\conv \Mm_w(w\la,\la)\br \\ 
&&\hs{10ex}\simeq \Mm\bl w(\La+\mu),\La+\mu\br\hconv\Mm_w(w\la,\la)\\
&&\hs{10ex}\simeq \Mm\bl w(\La+\mu+\la),\La+\mu+\la\br\\
&&\hs{10ex}\simeq\Mm(w\La,\La)\hconv\Mm_w\bl w(\mu+\la),\mu+\la).
\eneqn

Hence we obtain
$$\Mm(w\mu,\mu)\hconv \Mm_w(w\la,\la)\simeq\Mm_w\bl w(\mu+\la),\mu+\la),$$
as desired.
\QED

\begin{rem} 
\bnum
\item
If $\la\in W\La$ for some $\La\in\wlP_+$,
then $\Mm_w(w\la,\la)$ coincides with 
the determinantial module $\Mm(w\la,\la)$.
\item The simple module $\Mm_w(w\la,\la)$ may not be real.
For example, for $\g=A^{(1)}_1$, $I=\{0,1\}$,
$\la=\La_1-\La_0$ and $w=s_0s_1$, the module
$\Mm_w(w\la,\la)\simeq L(0)\hconv L(1)$ is not real.\label{item:A11}
\item In general, the class $[\bl\Mm_w(w\la,\la)]\in \Uqm\simeq K(R\gmod)$ 
depends on the choice of $\{Q_{i,j}(u,v)\}_{i,j\in I}$.
For example, for $\g=A^{(1)}_1$, $I=\{0,1\}$,
$\la=2(\La_1-\La_0)$ and $w=s_0s_1$, the class of the module
$\Mm_w(w\la,\la)\simeq L(0^2)\hconv L(1^2)$ depends on the choice of 
$Q_{01}(u,v)$ (see \cite[Example 3.3]{Kas12}). 

\item
Let $\la,\mu\in \wtl$.
Let $w\in W$ be an element such that $\mu=w\la$ and $\la$ is $w$-dominant.
Then $\Mm_w(\mu,\la)$ does depend on the choice of $w$ in general.

Let $\g=A^{(1)}_2$, $I=\{0,1,2\}$,
and $\la=\La_1+\La_2-2\La_0$, $w=s_2s_1s_0s_2s_1$,
and $v=s_1s_2s_0s_1s_2$. Then $\mu=w\la=v\la=\la-3(\al_1+\al_2)-\al_0$.
We have
\eqn
\Mm_w(\mu,\la)&\simeq& \Mm_w(w\la_1,\la_1)\conv\Mm_w(w\la_2,\la_2)
\simeq \ang{2,1,0,2,1}\conv\ang{1,2}
,\\
\Mm_v(\mu,\la)&\simeq&\Mm_v(v\la_2,\la_2)\conv\Mm_v(v\la_1,\la_1)
\simeq
\ang{1,2,0,1,2}\conv\ang{2,1}
\eneqn
where $\la_k=\La_k-\La_0$ ($k=1,2$).
Note that for $(\nu_1,\ldots,\nu_n)\in I^n$ such that
$(\al_{\nu_k},\al_{\nu_{k+1}})<0$ and $\al_{\nu_k}\not=\al_{\nu_{k+2}}$, we denote by 
$\ang{\nu_1,\ldots,\nu_n}$ the one-dimensional 
$R(\;\sum_{k=1}^n\al_{\nu_k})$-module such that
$e(\nu_1,\ldots,\nu_n)\ang{\nu_1,\ldots,\nu_n}=\ang{\nu_1,\ldots,\nu_n}$.
\item
When $\la\in\wtl$ is $w$-dominant and $\mu,\la+\mu\in\pwtl$,
we have $\Mm(w\mu,\mu)\hconv \Mm_w(w\la,\la)\simeq\Mm\bl w(\la+\mu),\la+\mu\br$
as seen in Theorem~\ref{thm:gdm}.
However, $\Mm(w\mu,\mu)$ and $\Mm_w(w\la,\la)$ may not commute in general.
For example, take $\g=A^{(1)}_1$,
$w=s_0s_1$, $\la=\La_1-\La_0$ as in \eqref{item:A11}.
If we take $\mu=\La_0$, then,
$\Mm(w\mu,\mu)\hconv \Mm_w(w\la,\la)\simeq \Mm(w\La_1,\La_1)\simeq
L(0)\hconv\ang{0,1}\simeq L(0^2)\hconv L(1)$ and
$\Mm_w(w\la,\la)\hconv\Mm(w\mu,\mu)\simeq \ang{0,1,0}$ is one-dimensional.
\ee
\end{rem}

Recall the definition of $\psi$ (see \eqref{def:antipsi}).

\Lemma \label{lem:psiMw}
Let $w\in W$ and let $\la\in\wtl$.
Assume that $\la$ is $w$-dominant.
Then we have  an isomorphism of graded modules
$$\psi_*\bl\Mm_w(w\la,\la)\br
\simeq\Mm_{w^{-1}}(-\la,-w\la).$$
\enlemma
\Proof
Let us argue by induction on $\ell(w)$.
Take $i\in I$ such that $w'=s_iw<w$.
Set $n=\ang{h_i,w'\la} \ge 0$.
Then we have
\eqn
\psi_*(\Mm_w(w\la,\la))
&&\simeq
\psi_*\Bigl(L(i^n)\hconv\Mm_{w'}(w'\la,\la)\Bigr)
\simeq
\psi_*(\Mm_{w'}(w'\la,\la) )\hconv L(i^n)\\
&&\simeq\Mm_{w'^{-1}}(-\la,-w'\la)\hconv L(i^n)\\
&&\simeq\Mm_{w'^{-1}s_i}(-\la,-s_iw'\la)
\simeq\Mm_{w^{-1}}(-\la,-w\la),
\eneqn
 where the the first and fourth isomorphisms follow from 
Theorem \ref{thm:gdm} \eqref{item a} and \eqref{item: sla}.
\QED

\Lemma
Let $w\in W$ and let $\la\in\wtl$.
Assume that $\la$ is $w$-dominant.
Then for any $\La\in\wlP_+$, we have
\eq
\ba{rl}\La\bl \Mm(w\La,\La),\,\Mm_w(w\la,\la)\br
&=-\bl w\La+\La,\,\wt(\Mm_w(w\la,\la))\br,\\
\tLa\bl \Mm(w\La,\La),\,\Mm_w(w\la,\la)\br
&=-\bl\La,\,\wt(\Mm_w(w\la,\la))\br.
\ea\label{eqe:LL}
\eneq
\enlemma

\Proof
Let us take $\mu\in\wlP_+$ such that $\eta\seteq\la+\mu\in\wlP_+$.
Set $M_\La=\Mm(w\La,\La)$, $M_\la=\Mm_w(w\la,\la)$,
$M_\mu=\Mm(w\mu,\mu)$ and $M_\eta=\Mm(w\eta,\eta)$.
Then we have
$$M_\mu\hconv M_\la\simeq M_\eta\qt{up to a grading shift.}$$
Since $M_\La$ commutes with $M_\mu$, by \cite[Theorem 4.12]{KKOP18} we have
\eqn
&&-\bl w\La+\La,\wt(M_\mu)+\wt(M_\la)\br=
-\bl w\La+\La,\wt(M_\eta)\br=
\La(M_\La, M_\eta)\\
&&\hs{10ex}=\La(M_\La, M_\mu\hconv M_\la)
=\La(M_\La, M_\mu)+\La(M_\La,M_\la)\\
&&\hs{10ex}=-\bl w\La+\La,\wt(M_\mu)\br+\La(M_\La,M_\la).
\eneqn
The last equality in \eqref{eqe:LL} follows from
\eqn
2\;\tLa\bl \Mm(w\La,\La),\,\Mm_w(w\la,\la)\br
&&=\La\bl \Mm(w\La,\La),\,\Mm_w(w\la,\la)\br+\bl\wt(M_\La), \wt(M_\la)\br\\
&&=-\bl w\La+\La,\wt(M_\la)\br+\bl w\La-\La,\wt(M_\la)\br\\
&&=-2\bl\La,\wt(M_\la)\br.
\eneqn
\QED

\section{Rigidity of the category $\tCw[{w}]$}

\subsection{Kernel of the localization functor}

There exists a unique family of subsets $\st{B_{w}(\infty)}_{w\in\weyl}$ of $B(\infty)$ satisfying the following properties (see \cite{Kas93}):
\begin{enumerate}
\item $B_{w}(\infty)=\left\{u_{\infty}\right\}$ if $w=1$,
\item if $s_{i} w<w$, then
$$
B_{w}(\infty)=\bigcup_{k \geqslant 0} \tilde{f}_{i}^{k} B_{s_{i} w}(\infty).
$$
\end{enumerate}
For $i\in I$ and a simple module $M$, set $\tE_{i}^{\max} M\seteq E_{i}^{(n)} M$ where $n=\eps_{i}(M)$.

Let $w\in\weyl$ and $\underline{w} = s_{i_1}s_{i_2}\cdots s_{i_l}$ a reduced expression of $w$.

For $b \in B(\infty)$, we denote by $S(b)$ the self-dual simple $R$-module corresponding to $b$.

\begin{prop} \label{prop:binBw}
Let $M$ be a self-dual simple $R$-module.
Then the following conditions are equivalent.
\bna
\item $\tE_{i_l}^{\max} \tE_{i_{l-1}}^{\max} \cdots \tE_{i_2}^{\max} \tE_{i_1}^{\max} M \simeq \one$.
\item There exists $(a_k)_{1\le k \le l} \in (\Z_{\ge 0}){\,}^{l }$ such that 
$$\text{$\wt(M)=-\displaystyle\sum_{k=1}^l a_k\alpha_{i_k}$ and \quad
$e(i_1^{a_1},\ldots,i_{l-1}^{a_{l-1}}, i_l^{a_l})M\neq 0$.}$$
\item  There exists $(a_k)_{1\le k \le l} \in \Z_{\ge 0}^l$ such that 
$L(i_1^{a_1}) \conv \cdots \conv L(i_{l-1}^{a_{l-1}})\conv L(i_l^{a_l})\epito M$.
\item $M\simeq S(b)$ for some  $b \in B_w(\infty)$.
\end{enumerate}
\end{prop}
\begin{proof}
$(a) \Rightarrow (b)  \Leftrightarrow (c)$ are trivial.
$(a) \Leftrightarrow (d)$ follows from that 
$$B_w(\infty) = \set{b \in B(\infty)}{\te_{i_l}^{\max} \cdots \te_{i_1}^{\max} b \simeq \one},$$
 which is \cite[Proposition 3.2.5]{Kas93}.

Assume (c).
We will show (a) by induction on $l$.
When $l=0$, it is trivial.
Assume that $l>0$.
Let $n=\eps_{i_1}(M)$ and $M_0=\tE_{i_1}^{\max} M \simeq E_{i_1}^{(n)} M$.
Then there exists a non-zero homomorphism
$$K\seteq E_{i_1}^{(n)} \left(L(i_1^{a_1}) \conv \cdots \conv L(i_{l-1}^{a_{l-1}})\conv L(i_l^{a_l})\right)\To[\phi]M_0.$$
 Then by the shuffle lemma, there exists a filtration $(F_s)_{0\le s\le t}$ of $K$ such that
\eqn
F_s /F_{s-1} \simeq L(i_1^{b_1}) \conv \cdots \conv L(i_l^{b_l})
\eneqn
for some $b_1,\ldots,b_l \ge 0$.

Take the smallest $s$ such that $\phi(F_s)\neq 0$.
Then $\phi$ induces a non-zero homomorphism $F_s /F_{s-1} \to M_0$.
Since $M_0$ is simple, we obtain a surjective homomorphism
$L(i_1^{b_1}) \conv \cdots \conv L(i_l^{b_l})\epito M_0$
for some $b_1,\ldots,b_l \ge 0$.
 Since $\eps_{i_1}(M_0)=0$, we conclude that $b_1=0$. 
By the induction hypothesis, we have
$\tE_{i_l}^{\max} \tE_{i_{l-1}}^{\max} \cdots \tE_{i_2}^{\max}  M_0 \simeq \one$  and hence $$\tE_{i_l}^{\max} \tE_{i_{l-1}}^{\max} \cdots \tE_{i_2}^{\max} \tE_{i_1}^{\max}  M \simeq \one,$$
as desired.
\end{proof}

\begin{prop} 
\label{prop:tEMneq1}
Let $M$ be a simple $R$-module.
The followings are equivalent.
\bna
\item $\La(\Mm(w\la,\la),M) < -(w\la+\la, \wt(M))$ for some $\la \in \wlP_+$,
\item $\tE_{i_l}^{\max} \cdots \tE_{i_1}^{\max} M \not\simeq \one$.
\end{enumerate}
\end{prop}
Note that the first condition is equivalent to 
$\tLa(\Mm(w\la,\la),M) < -(\la, \wt(M))$.

\begin{proof}
We will proceed by induction on $l= \ell(w)$.

If $l=0$, then $\Mm(w\la,\la) \simeq \one$. 
Hence 
$0 < -(2\la, \wt(M))$ for some $\la \in \wlP_+$
if and only if
$\wt(M)\neq 0$. 
This is equivalent to $M \not \simeq \one$.

\medskip
Assume that $l >0$.
Set  $i\seteq i_1$,
$w'\seteq s_{i}w$, $s\seteq\lan h_{i_1}, w'\la \ran$ and 
$M_0\seteq\tE_{i_1}^n (M)$ , where $n=\eps_{i_1}M$.
For $\la \in \wlP_+$, we have
\eqn 
&&\tLa(\Mm(w\la,\la),M) = \tLa(\Mm(w\la,\la),L(i^n)\hconv M_0)
 = \tLa(\Mm(w\la,\la),L(i^n)) +  \tLa(\Mm(w\la,\la),M_0) \\
 &&=(\la,n\al_{i})+ \tLa(L(i^s)\hconv \Mm(w'\la,\la),M_0) 
 =(\la,n\al_{i})+ \tLa(\Mm(w'\la,\la),M_0),
\eneqn
where the second and the third equalities follow from that $\de(\Mm(w\la,\la),L(i))=0$, 
the  last  follows from Corollary \ref{cor:tLaMN=0} together with $\tLa(L(i^s), M_0)=0$.

Hence we obtain
\eqn 
&&\tLa(\Mm(w\la,\la),M) +(\la, \wt(M)) 
=(\la,n\al_{i})+ \tLa(\Mm(w'\la,\la),M_0) +(\la, \wt(M)) \\
&&= \tLa(\Mm(w'\la,\la),M_0) +(\la, \wt(M_0)). 
\eneqn
It follows that
$\tLa(\Mm(w\la,\la),M) +(\la, \wt(M)) <0$ if and only if $ \tLa(\Mm(w'\la,\la),M_0) +(\la, \wt(M_0)) <0 $.
It is obvious that
$\tE_{i_l}^{\max} \cdots \tE_{i_1}^{\max} M \not\simeq \one$ if and only if
$\tE_{i_l}^{\max} \cdots \tE_{i_2}^{\max} M_0 \not\simeq \one$.
Hence the induction hypothesis implies that
$\tLa(\Mm(w\la,\la),M) < -(\la, \wt(M))$  if and only if $\tE_{i_l}^{\max} \cdots \tE_{i_1}^{\max} M \not\simeq \one$.
\end{proof}

Assume that $I=I_w\seteq\st{i_1,\ldots, i_l}$.
Recall that  $Q_{w} \cl  R\gmod \to \tRm[w]\simeq\tCw$ is the localization of $R \gmod$ via the real commuting family of graded braiders  $\st{(\dC_i, \coR_{\dC_i}, \dphi_i)}_{ i\in I}$.

Then we have the following proposition.
\begin{prop}[{cf.\ \cite{Kimura12}}] \label{prop:kernel}
Assume $I=I_w$.
Let $X$ be a module in $R(\beta) \gmod$.
Then the following conditions are equivalent.
\bna
\item $Q_w(X) \simeq 0$,
\item every simple subquotient $S$ of $X$ satisfies that 
$\tE_{i_l}^{\max} \cdots \tE_{i_1}^{\max} S \not\simeq \one$,
\item every simple subquotient $S$ of $X$ is isomorphic to $S(b)$ for some $b \notin B_w(\infty)$,
\item $e(i_1^{a_1},\ldots, {i_l}^{a_l} )X =0$ for 
any $(a_k)_{1\le k \le l} \in \Z_{\ge 0}^l$ such that 
$\beta=\displaystyle\sum_{k=1}^l a_k\alpha_{i_k}$.
\end{enumerate}
\end{prop}

\begin{proof}
For a simple module $S$ in $R\gmod$ and $\la \in \wlP_+$, $\coR_{\dC_\la}(S)=0$ if and only if 
$\La(\dC_\la,S) < -(w\la+\la, \wt(S))$ by \cite[Proposition 4.4, Proposition 5.1]{KKOP21}.
Recall that by the definition of localization, the identity $\id_{Q_w(S)}$ is the limit of  the morphisms
$R_{C_\la}(S)$ for $\la \in \wlP_+$.
Hence $Q_w(S)\simeq 0$ if and only if   $R_{C_\la}(S)=0$ for some $\la \in \wlP_+$.
Thus the desired result follows from
Proposition \ref{prop:binBw} and  Proposition \ref{prop:tEMneq1}.
\end{proof}

\Cor\label{cor:detnonzero}
For any $w$-dominant $\la\in\wtl$, we have
$$Q_w(\Mm_w(w\la,\la))\not\simeq0.$$
\encor
\Proof
We have
$\tE_{i_l}^{\max} \cdots \tE_{i_1}^{\max}\Mm_w(w\la,\la)\simeq \one$.
\QED

\Rem
Even if $\la$ is $w$-dominant, $\Mm_w(w\la,\la)$ may not belong to 
$\Cw$.
For example,take $\g=A_2$, $w=s_1s_2$ and $
\la=s_1\La_1$.
Then $\Mm_w(w\la,\la)\simeq L(2)$ does not belong to $\Cw$,
since
$\al_2\not\in\prD\cap w\nrD$.
\enrem

\Rem \label{rm:KO}
Let us denote by ${\rm Ker}\, Q_w$ the full subcategory of $R\gmod$ consisting of objects $X$ satisfying $Q_w(X)\simeq 0$. The Grothendieck group $K({\rm Ker}\, Q_w)$ is a two-sided ideal of $K(R\gmod)$.
Then
we have the following commutative diagram.
\eqn
\xymatrix{
\catC_w \ar[rr]^{\Phi_w} \ar@{^{(}->}[d] && \tcatC_w \ar[d]^{\iota_w} \\
R\gmod \ar[r]  \ar@/^2pc/[rr]^{Q_w}& R\gmod / {\rm Ker}  \,Q_w \ar[r] & \tRm[w]. 
}
\eneqn
Taking their Grothendieck groups, we have
\eqn
\xymatrix{
\Aq \ar[rr] \ar@{^{(}->}[d] && \Aq{} D(w\La,\La_i)^{-1}; i\in I] \ar[d]^{[\iota_w]} \\
\Aq[\mathfrak n] \ar[r] & \Aq[\mathfrak n]  / I_w \ar[r] &  \bl\Aq[\mathfrak n]  / I_w\br[[D(w\La,\La_i)]^{-1}; i\in I],
}
\eneqn
where $I_w$ is the ideal corresponding to ${\rm Ker}\, Q_w$, and $D(w\La,\La)$ denotes the quantum unipotent minor corresponding to $\Mm(w\La,\La)$. 
Recall that if $\g$ is symmetric and $\cor$ is of characteristic zero, then one can identify the isomorphism classes of self-dual simple modules with the elements of the upper global basis. 
In this case, the ideal $I_w$ coincides with the ideal $(U^-_{w,q})^\perp$  in \cite[Definition 3.37]{KO21}.
And the above diagram recovers \cite[Theorem 4.13]{KO21}, which asserts that $[\iota_w]$ is an isomorphism.
\enrem

\subsection{Equivalence between $\tCw^{*}$ and $\tCw[{w^{-1}}]$}

For $w\in \weyl$, let us denote by
$\Cs$ the full subcategory of $R\gmod$ consisting of 
$R$-modules $M\in R\gmod$ such that
\begin{align*}
\sgW(M) \subset \Sp( \prD \cap w \nrD ).
\end{align*}
Recall that, for a subset $S$ of $\R\tens_\Z\rtl$,
we write $\Sp S $ for the subset of linear combinations of elements in 
$S\cup\st{0}$ with non-negative coefficients.

Set $\dC^*_\La=\psi_*\bl\dM(w\La,\La)\br\simeq\dM(-\La,-w^{-1}\La)$ for $\La\in\pwtl$ and
$\dC^*_i=\dC^*_{\La_i}$.
Then $\st{\dC^*_i\mid i\in I}$ is a family of central objects of $\Cs$.
We denote by $\tCs$ the localization
$\Cs{\;[\dC^*_i{}^{\otimes-1}\mid i\in I\,]}$ of $\Cs$.

Then the anti-automorphism $\psi$ (see \eqref{def:antipsi}) of $R$
induces equivalences of monoidal categories

\eqn 
\psi_*&\cl&(\Cw)^\rev\isoto\Cs,\\
\psi_*&\cl&(\tCw)^\rev\isoto\tCs.
\eneqn
 Recall that  for a monoidal category $\sht$, 
 $\sht^\rev$ is the  category $\sht$  endowed with the new tensor product
$\tens^\rev$ defined by
$X\tens^\rev Y\seteq Y\tens X$.

\medskip
When $I=I_w\seteq  \{ i\in I \mid w\La_i \ne \La_i \}$, we denote by
$$Q_w^*\cl R\gmod\to\tCs$$
the localization functor, which is induced by
$Q_w\cl R\gmod\to \tCw$ and $\psi_*$.
That is,
 $Q^*_w$ is the composition $R\gmod \To[\psi_*] (R\gmod)^\rev  \To[Q_w] \tcatC_w^\rev \To[\psi_*]\tcatC^*_w$. 
Note that the composition
$$\Cs\to R\gmod\To[Q_w^*]\tCs$$
coincides with the localization functor of $\Cs$ by its central objects
$\st{\dC^*_\La}_{\La\in\pwtl}$.

The following  theorem is one of the main results of this paper.

\begin{thm} \label{th:main1}
There is an equivalence of  monoidal categories
$\tCs$ and $\tCw[{w^{-1}}]$.
More precisely, we have a quasi-commutative diagram \ro when $I_w=I$\rf
$$\xymatrix@C=10ex{
R\gmod\ar[d]_-{Q_w^*}\ar[dr]^(.6){Q_{w^{-1}}}\\
\tCs\ar[r]^(.4)\ssim&{\tCw[{w^{-1}}]\,.}
}
$$

\end{thm}
\begin{proof}
We may assume that  $I=I_w$ without loss of generality. 
Recall the localization functor  $ Q_w  \cl R\gmod \to \tRm[w]$ from 
$R\gmod$ to its localization via the real commuting family of graded braiders
$(\dC_i, \coR_{\dC_i}, \dphi_i)_{ i\in I}$. 
Then there  is a monoidal equivalence of categories  (Theorem \ref{thm: Cwequiv})
$$\iota_{w}\cl \tCw\isoto \tRm[w].$$

Now we consider the chain of the morphisms

$$R\gmod\To[\psi_*](R\gmod)^\rev\To[Q_w]\bl\tRm[w]\br^\rev.$$
We claim that the composition $Q_w\circ\psi_*$ factors as
\eq
\xymatrix@C=7ex{
R\gmod\ar[r]_{Q_{w^{-1}}}\ar@/^2pc/
[rr]^{Q_w\circ\psi_*}& {\tRm[w^{-1}]}\ar@{.>}[r]_{\mathcal F_w}&\bl\tRm[w]\br^\rev.}
\label{eq:Fw}
\eneq

By Theorem \ref{Thm: graded localization}, it is enough to show that 
  
\be[(a)]
\item $ ( Q_{w}\circ \psi_*)(\dC^-_\la) $ is invertible in $\bl\tRm[w]\br^\rev$ for any $\la\in\pwtl$, where
$\dC^-_\la \seteq \Mm(w^{-1}\la,\la)$, and 

\item for any $i\in I$ and $X \in R\gmod$, $(Q_{w}\circ \psi_*)(R_{C_i^-}(X))\cl
(Q_{w}\circ \psi_*)(\dC_i^-\conv X)\to(Q_{w}\circ \psi_*)(X\conv \dC_i^-)$ is an isomorphism.
\ee

\medskip
 In the course of the proof, we forget grading shifts. 

First note that
\eq
\text{For $X\in R\gmod$,
$Q_w(\psi_*(X))\simeq0$ if and only if
$Q_{w^{-1}}(X)\simeq0$,}\label{eq:QWWi}
\eneq
which follows from Proposition~\ref{prop:kernel}.

\mnoi
(a) \ By  Lemma  \ref{lem:psiMw} and Theorem \ref{thm:gdm} (e), we have
\eq \label{eq:invertible}
 \Mm(w\mu,\mu)\hconv \psi_*\bl \Mm(w^{-1}\la,\la)\br \simeq 
\Mm(w\mu,\mu)\hconv \Mm_{w}(-\la,-w^{-1}\la) 
    \simeq  \Mm(w\eta,\eta),
 \eneq
where $\la\in\wlP_+$ and
$\mu,\eta\in\wlP_+$ such that
$\eta-\mu=-w^{-1}\la$.
Hence, Proposition~\ref{prop:Locsim}
implies that
$$Q_w\bigl(\psi_* \dC^-_\la\bigr)\simeq \dC_{-w^{-1}\la}.$$
Hence (a) follows.

\mnoi
(b) \  
It is enough to show that  $(Q_{w}\circ \psi_*)(R_{\dC_\La^-}(X)) \neq 0$ for any $\La \in \wlP_+$ and any simple module $X$  in $R\gmod$ with $(Q_{w}\circ \psi_*)(X)\not\simeq 0$.
 
Indeed, when $X$ is simple,  since $(Q_{w}\circ \psi_*)(\dC_\La^-)$  is invertible by (a), the objects
$(Q_{w}\circ \psi_*)(\dC_\La^- \conv X)$ and $(Q_{w}\circ \psi_*)( X \conv \dC_\La^-)$ are simple in $\bl\tRm[w]\br^\rev$
so that  $(Q_{w}\circ \psi_*)(R_{\dC_\La^-}(X))$ is an isomorphism.
For general $X \in R\gmod$, the morphism $(Q_{w}\circ \psi_*)(R_{\dC_\La^-}(X))$ is an isomorphism by induction on the length of $X$.

Note that for any simple $X \in R\gmod$ such that  $(Q_{w}\circ \psi_*)(X)\not\simeq 0$, there exists $ \mu \in \wlP_+$ and a simple module $Y$ in $\catC_{w^{-1}}$ such that there exists an epimorphism in $R\gmod$.
\eq f\cl \dC^-_{\mu} \conv X \epito Y. \label{eq:epi}\eneq
Indeed, if $(Q_{w}\circ \psi_*)(X)\not\simeq 0$, then $Q_{w^{-1}}(X)\not \simeq 0$
by \eqref{eq:QWWi}. Therefore
$Q_{w^{-1}}(X)\simeq \dC^-_{-\mu} \conv Y$ for some $\mu\in \wlP_+$ and a simple $Y\in \catC_{w^{-1}}$, equivalently $Q_{w^{-1}}(\dC^-_{\mu} \conv X)\simeq Q_{w^{-1}}(Y)$. Hence there is a non-zero homomorphism $\dC^-_{\mu'}\conv\dC^-_{\mu} \conv X \to\dC^-_{\mu'}\conv Y$ in $R\gmod$, which is an epimorphism since $\dC^-_{\mu'}\conv Y$ is simple.  Then, replacing $\dC^-_{\mu'}\conv\dC^-_{\mu}$ and $\dC^-_{\mu'}\conv Y$
with $\dC^-_{\mu}$ and $Y$ respectively,  we obtain an epimorphism
\eqref{eq:epi}.

\smallskip
Then for any $\La\in\pwtl$, the following diagram is commutative.
$$
\xymatrix@C=5em{
& \dC^-_\mu \conv  \dC^-_\La \conv  X \ar[dr]^{ \dC^-_\mu \conv  \coR_{ \dC^-_\La}(X)} &  \\
\dC^-_\La \conv  \dC^-_\mu \conv  X \ar[ur]^{ \coR_{ \dC^-_\La}( \dC^-_\mu) \conv X}  \ar[rr]_{ \coR_{ \dC^-_\La}(  \dC^-_\mu \conv X  )}  \ar@{->>}[d]^f  & 
&  \dC^-_\mu \conv  X \conv \dC^-_\La \ar@{->>}[d]^f \\
\dC^-_\La \conv Y \ar[rr]_{ \coR_{ \dC^-_\La}(Y)}  && Y \conv \dC^-_\La \,.
}
$$

Assume that
$(Q_w \circ \psi_*) (\coR_{ \dC^-_\La}(X)) =0.$
Applying $Q_w \circ \psi_*$ to the above diagram, we obtain 
$(Q_w \circ \psi_*)(\coR_{ \dC^-_\La}(Y))=0.$
Because $\coR_{\dC^-_\La}(Y)$ is an isomorphism,  we have
$(Q_w \circ \psi_*)(\dC^-_\La \conv Y) \simeq 0.$
Since $(Q_w \circ \psi_*)(\dC^-_\La)$ is invertible by (a),
we obtain $(Q_w \circ \psi_*)(Y) \simeq 0$, which implies
$Q_{w^{-1}}(Y)\simeq 0$ by \eqref{eq:QWWi},
which contradicts that $Y$ is a simple module in $\catC_{w^{-1}}$.
Thus we obtain (b).

Thus we obtain the diagram \eqref{eq:Fw}.

\medskip
By changing the roles of $w$ and $w^{-1}$ in the above argument,  we obtain the lower square in the following commutative diagram
$$
\xymatrix@C=4em{
R\gmod  \ar[r]^{Q_{w^{-1}}} \ar[d]_{\psi_*}  & \tRm[w^{-1}] \ar@{.>}[d]^{\mathcal F_w }\\
 (R\gmod)^{\rev} \ar[r]^{Q_w}   \ar[d]_{\psi_*} & \tRm[w]\br^{\rev}\ar@{.>}[d]^{\mathcal F_{w^{-1}} } \\
R\gmod \ar[r]^{Q_{w^{-1}}}     & \tRm[w^{-1}].  \\
}
$$
Since $\psi_*$ is involutive, the composition $\mathcal F_{w^{-1}} \circ \mathcal F_w$ is isomorphic to the identity functor on $ \tRm[w^{-1}]$ by Theorem \ref{Thm: graded localization} (iii).
It follows that  $\mathcal F_w$ is an equivalence of categories, as desired.
\end{proof}

\Cor Assume that  $I=I_w$.
Let $X$ be a simple module in $R\gmod$ satisfying $Q_w(X) \not\simeq 0$.
Then the triple 
$$\bl \Mm(w\mu,\mu),X,\psi_*\bl \Mm(w^{-1}\La,\La) \br \br$$
is a normal sequence for  any $\mu, \La  \in \wlP_+$. 
\encor
\begin{proof}
In the course of the proof the above theorem,  we showed that $(Q_{w}\circ \psi_*)(R_{\dC_\La^-}(\psi_*(X))) \neq 0$
for any simple module $X \in R\gmod$ such that $Q_w(X) \not\simeq 0$
and any $\La  \in \wlP_+$.
That is,
\eqn 
X \conv \psi_*(\dC^-_\La) \To[\psi_*(\coR_{\dC^-_\La})] \psi_*(\dC^-_\La)\conv X 
\eneqn
does not vanish under the functor $Q_w$. 
Hence  the homomorphism
\eqn 
\dC_\mu \conv X \conv \psi_*(\dC^-_\La) \To[\dC_\mu \conv \psi_*(\coR_{\dC^-_\La})] 
\dC_\mu \conv \psi_*(\dC^-_\La)\conv X \To[\coR_{\dC_\mu }(\psi_*(\dC^-_\La)\conv X)]
 \psi_*(\dC^-_\La)\conv X \conv \dC_\mu
\eneqn
is non-zero for any $\mu \in \wlP_+$.
Since $\dC_\mu, \psi_*(\dC^-_\La)$ and $X$ are simple modules, the above homomorphism is equal to the composition  of the r-matrices
$\rmat{X, \psi_*(\dC^-_\La)}$, $\rmat{\dC_\mu, \psi_*(\dC^-_\La)}$ and 
$\rmat{\dC_\mu, X}$ up to a constant multiple.
Thus 
$\bl \dC_\mu,X,\psi_*\bl \dC^-_\La  \br \br$ is a normal sequence.
\end{proof}

Recall that the category $\tCw[{w^{-1}}]$ is left rigid, i.e.,  every object of $\tCw[{w^{-1}}]$ has a left dual in $\tCw[{w^{-1}}]$ (\cite[Corollary 5.11]{KKOP21}).
It follows that $(\tCw[{w}])^\rev$ is left rigid by the above theorem. 
Hence we obtain the following theorem as its corollary.
\Th\label{th:rigid}
The category $\tCw[w]$ is a rigid monoidal category.
\enth

\section{Localization of the category $\catC_{w,v}$ }
Through this section, we assume that $w,v \in \weyl$ satisfies $v \ble w$  and $I_w =I$.

Let $N$ be a (not necessarily simple) module in $\catC_{*,v}$ and $\la \in \wlP_+$.
Set $\al\seteq v\la-w\la$,  $\beta\seteq\la-v\la$ and $\gamma\seteq-\wt(N)$.
Note that the R-matrix $\rmat{\Mm(w\la,v\la), \Mm(v\la, \la)}$ decomposes into
$$\Mm(w\la,v\la) \conv \Mm(v\la, \la)\To[\pi]\Mm(w\la, \la) \To[\iota]q^{(\alpha,\beta)} \Mm(v\la,\la) \conv \Mm(w\la, v\la)$$
by \cite[Proposition 4.6]{KKOP18}.

Let $\rho_{w,v,\la}(N)$ be the composition of the following chain of homomorphisms:
\eqn
&&\Mm(w\la,v\la) \conv N \conv \Mm(v\la, \la)
\To[\Mm(w\la,v\la)  \circ \rmat{N,\Mm(v\la, \la)}]
q^{(\beta,\gamma)} \Mm(w\la,v\la) \conv \Mm(v\la, \la) \conv N  \\
&&\hs{20ex}\To[\pi \circ N] 
q^{(\beta,\gamma)} \Mm(w\la, \la) \conv N 
\To[\coR_{\Mm(w\la,\la)}(N)]
q^{-(w\la+v\la,\gamma)} N \conv \Mm(w\la,\la).
\eneqn
Here, $\coR_{\Mm(w\la,\la)}$ is the non-degenerate braider associated with $\Mm(w\la,\la)$,
and $ \rmat{N,\,\Mm(v\la, \la)}$ is well-defined because $\bl N,\Mm(v\la, \la))\br$
is unmixed (Proposition~\ref{prop:unmixedr}). 
Note that 
we have
$$\rmat{N,\Mm(v\la, \la)}(u\tens v)=\tau_{w[\Ht(\beta),\Ht(\gamma)]}(v\tens u)\qt{for any $u\in N$ and $v\in \Mm(v\la, \la)$}$$ and hence $\La(N,\Mm(v\la, \la))=-(\beta,\gamma)$. 
Since $\phi_{\Mm(w\la,\la)}(\gamma)=-(w\la+\la,\gamma)$, we have
$$-\deg(\rho_{w,v,\la}(N))=(\beta,\gamma)-(w\la+\la,\gamma) =-(w\la+v\la,\gamma). $$

\begin{lem} \label{lem:eNM}
For any $R(\gamma)$-module $N\in \catC_{*,v}$,
 We have 
$$e(\al+\gamma,\beta) \left(  \Mm(w\la,\la) \conv N  \right) \simeq q^{-(\beta,\gamma)}\left(\Mm(w\la,v\la) \conv N\right) \tens \Mm(v\la, \la)$$
and
$$\hs{-5ex}e(\al+\gamma,\beta) \left( N \conv \Mm(w\la,\la) \right) \simeq \left(N \conv \Mm(w\la,v\la) \right) \tens \Mm(v\la, \la).$$
\end{lem}
\Proof 
To obtain the first isomorphism, it is enough to apply 
\eqref{eq:24} in Proposition~\ref{prop:varunmixed}
by taking $\bl \Mm(w\la,v\la),N,\Mm(v\la,\la)\br$ as $(L,M,N)$,
and for the second we can  apply \eqref{eq:21} 
by taking $\bl N,\Mm(w\la,v\la),\Mm(v\la,\la)\br$ as $(L,M,N)$.
\QED

\begin{prop} \label{prop:psiwvla}
For any $R(\gamma)$-module $N\in\catC_{*,v}$, there exists a unique homomorphism 
$$\psi_{w,v,\la}(N)\cl \Mm(w\la,v\la) \conv N \to q^{-(w\la+v\la,\gamma)} N \conv \Mm(w\la,v\la) $$
 such that 
 the following diagrams are commutative.
\eq
&&\hs{2ex}\xymatrix@C=10ex{
\bl \Mm(w\la,v\la) \conv N\br\tens \Mm(v\la,\la)\ar[r]^-{\psi_{w,v,\la}(N)}
\ar[d]^\bwr& q^{-(w\la+v\la,\gamma)}
\bl N \conv \Mm(w\la,v\la)\br\tens \Mm(v\la,\la)\ar[d]^\bwr\\
q^{(\beta,\gamma)}e(\al+\gamma,\beta)\bl\Mm(w\la,\la) \conv N\br\ar[r]^-{\coR_{\Mm(w\la,\la)}(N)}&
 q^{-(w\la+v\la,\gamma)}e(\al+\gamma,\beta)\bl N\conv\Mm(w\la,\la)\br,
}\eneq
\eq
&& \xymatrix@R=7ex{
\Mm(w\la,v\la) \conv N\conv  \Mm(v\la, \la) \ar[dr]^{\rho_{w,v,\la}(N)}  
\ar[d]_{\psi_{w,v,\la}(N)\circ\Mm(v\la,\la)}\\
q^{-(w\la+v\la,\gamma)} N \conv \Mm(w\la,v\la)\conv\Mm(v\la,\la) 
\ar[r]&q^{-(w\la+v\la,\gamma)} N \conv \Mm(w\la,\la) .
 }\label{diag:Rwv}
\eneq
\end{prop}
\begin{proof}
Applying $e(\al+\gamma,\beta)$ to 
$$\coR_{\Mm(w\la,\la)}(N)\cl \Mm(w\la,\la)\conv N\to q^{-(w\la+\la,\gamma)}N\conv \Mm(w\la,\la),$$
we obtain by Lemma~\ref{lem:eNM}
$$q^{-(\beta,\gamma)}\bl\Mm(w\la,v\la)\conv N\br\tens \Mm(v\la,\la)\to 
q^{-(w\la+\la,\gamma)}\bl N\conv \Mm(w\la,v\la)\br\tens \Mm(v\la,\la).$$
Since we have
$\END(\Mm(v\la,\la))\simeq\cor\,{\id}$, we obtain
$$\psi_{w,v,\la}(N)\cl \Mm(w\la,v\la) \conv N \to q^{-(w\la+v\la,\gamma)} N \conv \Mm(w\la,v\la). $$
The commutativity of \eqref{diag:Rwv} is then obvious.
\end{proof}

\begin{df} For $N \in \catC_{*,v}$ and $\la \in \wlP_+$
we define 
\eqn
\coR_{\Mm(w\la,v\la)}(N) \seteq \psi_{w,v,\la}(N) \cl \Mm(w\la,v\la)\conv N \to q^{\phi_{w,v,\la}(-\wt(N))} N \conv \Mm(w\la,v\la),
\eneqn
where 
$$\phi_{w,v,\la}(\gamma) = -(w\la+v\la,\gamma) \quad \text{for} \ \gamma \in \rtl.$$
\end{df}

\begin{thm}
The family $\{(\Mm(w\La_i,v\La_i), \coR_{\Mm(w\La_i,v\La_i)}, \phi_{w,v, \La_i} )\}_{i\in I}$ 
is a real commuting family of graded braiders in the category $\catC_{*,v}$,
and also it is a family of central objects in $\catC_{w,v}$.
\end{thm}
\begin{proof}
We know that
$\{(\Mm(w\La_i,\La_i), \coR_{\Mm(w\La_i,\La_i)}, \phi_{w, \id, \La_i } )\}_{i\in I}$ 
is a real commuting family of graded braiders in the category $R\gmod$ (Proposition \ref{Prop: canonical braiders}),
and also it is a family of central objects in $\catC_{w}$
(Theorem \ref{Thm: R Ci iso}). 
Hence our assertion follows from Proposition~\ref{prop:psiwvla}.
\end{proof}

Set
$$\tcatC_{*,v}{[w]}\seteq\catC_{*,v}[\Mm(w\La_i,v\La_i)^{\conv -1}  \mid i \in I], $$
and 
$$\tcatC_{w,v}\seteq\catC_{w,v}[\Mm(w\La_i,v\La_i)^{\conv -1}  \mid i \in I]. $$
Since $\catC_{w,v}$ is a full subcategory of $\catC_{*,v}$, the canonical embedding induces  a fully faithful monoidal functor 
\eqn 
\iota_{w,v}\cl  \tcatC_{w,v} \rightarrowtail  \tcatC_{*,v}[w].
\eneqn

\begin{thm} \label{thm:equiv}
The functor $ \iota_{w,v}\cl  \tcatC_{w,v} \rightarrowtail  \tcatC_{*,v}[w]$ is an equivalence of categories.
\end{thm}
\begin{proof}
Let us denote by $Q_{w,v}\cl \catC_{*,v} \to \tcatC_{*,v}[w] $
the localization functor. 
It is enough to show that for every object $X\in \tcatC_{*,v}[w] $, there exists an object $Z \in \tcatC_{w,v}$ such that   $\iota_{w,v}(Z) \simeq X$. Since
$\tcatC_{w,v}$ is closed under taking extension by 
\cite[Proposition 2.10]{KKOP21},
we may assume further that $X$ is a simple object. Since every simple object in $\tcatC_{*,v} [w] $ is of the form $Q_{w,v}(Y)\conv \Mm(w\la,v\la)^{\circ -1}$ for some $\la\in \wtl_+$ and a simple object $Y\in\catC_{*,v}$ (\cite[Proposition 4.8]{KKOP21}), 
we may assume that $X$ is a simple module in $\catC_{*,v}$.  

Let $X$ be a simple module in $\catC_{*,v}$.
We shall show that $Q_{w,v}(X)\in\tcatC_{w,v}$. 

Recall that $Q_w\cl R\gmod\to\tRm[w]\simeq\tCw$ is the localization functor.

\snoi
(i) Assume first that $Q_w(X) \not \simeq 0$. 
Then, there exists $\La \in \wtl_+$  and a simple module $Y\in \catC_w$ such that 
$$Q_w(X) \simeq \dC_\La^{\circ -1} \conv Y,$$
where $\dC_\La=\Mm(w\La,\La)$.
Hence we have an epimorphism in $R\gmod$ (by replacing $\La$ if necessary)
$$\dC_\La \conv X \epito Y.$$
Note that by Lemma \ref{lem:eNM} we have
\eq
&& \ba{l}\Res_{*,\La-v\La} (\dC_\La \conv X) \simeq \bl\Mm(w\La,v\La)\ \conv X\br \tens\Mm(v\La,\La), \quad \text{and} \\
 \Res_{*,\La-v\La} (X\conv \dC_\La) \simeq  X \conv \Mm(w\La,v\La)\tens \Mm(v\La,\La).\ea\label{eq:res}
\eneq
Set $\beta=\La-v\La=-\wt\bl(\Mm(v\La,\La)\br$.
Applying $\Res_{*,\beta} $ to the diagram
\eqn
\xymatrix{
\dC_\La \conv X \ar[rr]^{\coR_{\dC_\La}(X)} \ar@{->>}[dr]&& X \conv \dC_\La \\
&Y\ar@{ >->}[ur]&
}
\eneqn
we get a commutative diagram
\eqn
\xymatrix{
\bl\Mm(w\La,v\La)\ \conv X \br\tens\Mm(v\La,\La) \ar[rr]^{\Res_{*,\beta} (\coR_{\dC_\La}(X)) } \ar@{->>}[dr]&& \bl X \conv \Mm(w\La,v\La)\br\tens \Mm(v\La,\La)\\
& \Res_{*,\beta} (Y) \ar@{>->}[ur]&
}.
\eneqn

Let $\psi \cl \Mm(w\La,v\La)\conv X \to X \conv \Mm(w\La,v\La)$ be the homomorphism such that $$\psi \tens \Mm(v\La,\La) = \Res_{*,\beta} (\coR_{\dC_\La}(X)). $$
Set $Z\seteq \Im(\psi)$. 
Since  $\psi=c\rmat{\Mm(w\La,v\La),X}$  for some $c\in \corp^\times$,
 $Z$ is simple.
Since $X$ and $\Mm(w\La,v\La)$ belong to $\catC_{*,v}$,
so does $Z$.
Note that  $Z $ belongs to $\catC_w$, because
$$\gW(Z) \subset \gW(Y) \in \rtl_+\cap w \rtl_-.$$
Since $Z$ is the image of
$\Mm(w\La,v\La)\conv X \To[{\RR_{\Mm(w\La,v\La)}}] X\conv\Mm(w\La,v\La)$, 
we have
$$Q_{*,v}(X) \simeq \Mm(w\La,v\La)^{\circ -1} \conv Q_{*,v}(Z) \simeq \iota_{w,v} (\Mm(w\La,v\La)^{\circ -1} \conv Q_{w,v}(Z)).$$

\mnoi
(ii)\ Assume that $X\in \catC_{*,v}$ satisfies $Q_w(X) \simeq 0$.  
Then there exists $\La\in\pwtl$ such that
$\coR_{\dC_\La}(X)\cl \dC_\La\conv X\to X\conv\dC_\La$
vanishes.
Applying $\Res_{*,\beta}$ ($\beta=\La-v\La$), we deduce from \eqref{eq:res}
that
$${\coR_{\Mm(w\La,v\La)}(X)}\tens \Mm(v\La,\La)
\cl \bl\Mm(w\La,v\La)\conv X\br\tens \Mm(v\La,\La) \To\bl X\conv\Mm(w\La,v\La)\br\tens \Mm(v\La,\La)$$
vanishes.
Hence $\coR_{\Mm(w\La,v\La)}(X)$ vanishes, which means that
$Q_{w,v}(X)\simeq0$.
\end{proof}


\begin{thebibliography}{99}




















\bibitem{GLS13}  C. Gei\ss, B. Leclerc and J. Schr\"oer, 
{\em Cluster structures on quantum coordinate rings}, Selecta Math. (N.S.) {\bf 19}  (2013),  no. 2, 337--397.


\bibitem{GY14} K.\ R.\ Goodearl and M.\ T.\ Yakimov,  {\em Quantum cluster algebras and quantum nilpotent algebras}, Proc. Natl. Acad. Sci. USA {\bf111} (2014), no. 27, 9696--9703.


\bibitem{GY17} \bysame, {\emph Quantum cluster algebra structures on quantum nilpotent algebras, Mem. Amer. Math. Soc}. {\bf247} (2017), no. 1169, vii+119. 

%


\bibitem{KK11}
S.-J. Kang and M. Kashiwara, \emph{Categorification of Highest Weight Modules via Khovanov-Lauda-Rouquier Algebras},
 Invent. Math. \textbf{190} (2012), no. 3, 699--742.



\bibitem{KKK18}
S.-J. Kang, M. Kashiwara and M. Kim, {\em Symmetric quiver
Hecke algebras and R-matrices of quantum affine algebras},
Invent. Math. \textbf{211} (2018), no. 2, 591--685.





\bibitem{KKKO15}
S.-J. Kang, M. Kashiwara,  M. Kim  and   S.-j. Oh,
\newblock{\em Simplicity of heads and socles of tensor products},
Compos. Math. \textbf{151} (2015), no. 2, 377--396.




\bibitem{KKKO18}
\bysame,
\newblock{\em Monoidal categorification of cluster algebras},
J. Amer. Math. Soc. \textbf{31} (2018), no. 2, 349--426.


\bibitem{Kas93}
M. Kashiwara, 
\newblock{\em The crystal base and Littelmann's refined Demazure Character formula},
Duke Math., \textbf{71} (1993), no. 3, 839--858.

\bibitem{Kas12} 
\bysame, \newblock{Notes on parameters of quiver Hecke algebras},  Proceedings of the Japan Academy, Series A, Mathematical Sciences \textbf{88}, (2012) no. 7,  97--102.

\bibitem{KK19}
M. Kashiwara, and  M. Kim, \emph{Laurent phenomenon and simple modules of quiver Hecke algebras}, Compos. Math. {\bf 155}, (2019), no. 12, 2263--2295. 

\bibitem{KKOP18}
M.~Kashiwara, M. Kim, S.-j. Oh, and  E.~Park,
\newblock{\em Monoidal categories associated with strata of flag manifolds},
Adv. Math. \textbf{328} (2018), 959-1009.




\bibitem{KKOP20} \bysame, \newblock {\em Monoidal categorification and quantum affine algebras},  Compos. Math. {\bf156}, (2020), no. 5, 1039--1077. 


\bibitem{KKOP21}
\bysame, \newblock{\em Localizations for quiver Hecke algebras}, Pure Appl. Math. Q. \textbf{17} (2021), no. 4, 1465--1548.

\bibitem{KP18}
M.~Kashiwara and E.~Park, \newblock{\em Affinizations and $R$-matrices for quiver Hecke algebras},
J. Eur. Math. Soc. \textbf{20}, 1161--1193.







\bibitem{KL09}
M.~Khovanov and A. Lauda, \emph{A diagrammatic approach to
categorification of quantum groups
  {I}}, Represent. Theory \textbf{13} (2009), 309--347.


\bibitem{KM17}
A. ~Kleshchev and R. ~ Muth, \emph{Stratifying KLR algebras of affine ADE types.}, J. Algebra \textbf{475} (2017),  133--170.

\bibitem{Kimura12}
Y.~Kimura, \emph{Quantum unipotent subgroup and dual canonical basis},
Kyoto J. Math. \textbf{52} (2012), no.~2, 277--331.

\bibitem{KO21} Y. Kimura and H. Oya, \emph{Twist automorphisms on quantum unipotent cells and dual canonical bases},
 International Mathematics Research Notices 2021, no. 9, 6772--6847.



\bibitem{LV11}
A.~Lauda and M.~Vazirani, \emph{Crystals from categorified quantum groups},
  Adv. Math. \textbf{228} (2011), no.~2, 803--861.

\bibitem{Lec16} 
B. Leclerc, \emph{Cluster structures on strata of flag varieties}.  Adv. Math. {\bf300} (2016)  190--228.


\bibitem{McNamara15}  P.  J.~ McNamara,  \emph{Finite dimensional representations of Khovanov-Lauda-Rouquier algebras I: Finite type}
J. reine angew. Math. {\bf707} (2015), 103--124


\bibitem{R08}
R.~Rouquier, \emph{2-Kac-Moody algebras},   arXiv:0812.5023\/v1.

\bibitem{R11}
\bysame, {\em Quiver Hecke algebras and 2-Lie algebras},
Algebra Colloq. {\bf 19} (2012), no. 2, 359--410.



\bibitem{TW16} 
P.~Tingley and B.~Webster,
{\em Mirkovi\'c-Vilonen polytopes and Khovanov-Lauda-Roquuier algebras}, Compos. Math. \textbf{152} (2016), no. 8, 1648--1696.





\bibitem{VV09}
M. Varagnolo and E. Vasserot,
 \emph{Canonical bases and KLR algebras},
J. Reine Angew. Math. \textbf{659} (2011), 67--100.

\end{thebibliography}
\end{document}